\documentclass[a4paper, 11pt, reqno]{amsart} %draft
\usepackage[normalem]{ulem}
\usepackage{amsmath, amsthm, amssymb,bbm, tikz, stmaryrd,mathtools} 
\usepackage{bm}
\usepackage[utf8]{inputenc} %kodowanie dokumentu
\usepackage{graphicx} % inserting pictures
\usepackage[small]{caption}
\usepackage{comment}
\usetikzlibrary{arrows,calc,decorations.pathreplacing}
\numberwithin{equation}{section}

\usepackage{lmodern} %searching in pdf
\usepackage[T1]{fontenc} %font encoding

\usepackage{enumerate} %enumeration
\usepackage{verbatim} %block comments
\usepackage{color}
\usepackage{tabularx}

\usepackage[colorlinks, linkcolor=blue, citecolor = green!55!black]{hyperref} %adds hyperlinks

\usepackage[letterpaper,nomarginpar, margin=1in]{geometry}

\usepackage{aliascnt}

\usepackage{xspace}

\usepackage[capitalise]{cleveref}

\newcommand{\Prob}{\mathbb{P}}
\renewcommand{\P}{\Prob}

\newcommand{\R}{\mathbb{R}}
\newcommand{\E}{\mathbb{E}}

\newcommand{\N}{\mathbb{N}}
\newcommand{\Z}{\mathbb{Z}}

\newcommand{\fB}{\mathfrak{B}}

\newcommand{\cA}{\mathcal A}

\newcommand{\cF}{\mathcal F}

\newcommand{\cI}{\mathcal{I}}

\newcommand{\ep}{\varepsilon}
\newcommand{\midd}{\ \Big|\ }

\newcommand{\mc}{\mathcal}
\newcommand{\mrm}{\mathrm}

\newcommand{\iid}{i.i.d.\ }
\newcommand{\intint}[1]{\llbracket#1\rrbracket}

\renewcommand{\th}{\textsuperscript{th}\xspace}

\newtheorem{maintheorem}{Theorem}

\newaliascnt{maincoro}{maintheorem}

\aliascntresetthe{maincoro}

\newaliascnt{mainprop}{maintheorem}

\aliascntresetthe{mainprop}

\crefname{maintheorem}{Theorem}{Theorems}
\crefname{maincoro}{Corollary}{Corollaries}
\crefname{mainprop}{Proposition}{Propositions}

\newtheorem*{theorem*}{Theorem}

\newtheorem{theorem}{Theorem}[section]

\newaliascnt{lemma}{theorem}
\newtheorem{lemma}[lemma]{Lemma}
\aliascntresetthe{lemma}

\newaliascnt{claim}{theorem}
\newtheorem{claim}[claim]{Claim}
\aliascntresetthe{claim}

\newaliascnt{proposition}{theorem}
\newtheorem{proposition}[proposition]{Proposition}
\aliascntresetthe{proposition}

\newaliascnt{observation}{theorem}

\aliascntresetthe{observation}

\newaliascnt{corollary}{theorem}
\newtheorem{corollary}[corollary]{Corollary}
\aliascntresetthe{corollary}

\newaliascnt{assumption}{theorem}
\newtheorem{assumption}[assumption]{Assumption}
\aliascntresetthe{assumption}

\newaliascnt{condition}{theorem}

\aliascntresetthe{condition}

\theoremstyle{definition}{

	\newtheorem{definition}[theorem]{Definition}
	\newtheorem*{definition*}{Definition}
	
	\newtheorem*{problem*}{Problem}
	
	\newtheorem*{question*}{Question}
	\newtheorem*{example*}{Example}
	\newtheorem*{examples*}{Examples}

	\newtheorem{remark}[theorem]{Remark}
	
}

\newcommand{\one}{\mathbbm{1}}

\newcommand{\indset}[1]{\mathbbm{1}_{#1}}

\newcommand{\given}{\;\big|\;}

\newcommand{\floor}[1]{\left\lfloor#1\right\rfloor}

\renewcommand{\bar}[1]{\overline{#1}}

\renewcommand{\u}{\mathbf{u}}
\renewcommand{\v}{\mathbf{v}}

\newcommand{\Cl}{\mathsf{Cl}}

\newcommand{\B}{\mathfrak{B}}

\newcommand{\cC}{\mathcal{C}}

\renewcommand{\L}{\mathrm{L}}
\newcommand{\rR}{\mathrm{R}}

\makeatletter
\crefname{maintheorem}{Theorem}{Theorems}
\makeatother

\allowdisplaybreaks

\title[Scaling limit \& tail bounds for a random walk model of SOS level lines]{Scaling limit and tail bounds for a\\ random walk model of SOS level lines}
\author{Milind Hegde}
\author{Yujin H. Kim}
\author{Christian Serio}

\address{ M. Hegde 
\newline
Nanyang Technological University, Division of Mathematical Sciences, 21 Nanyang Link, Singapore 637371}
\email{milind.hegde@ntu.edu.sg}

\address{ Y. H. Kim
\newline
Princeton University, Department of Mathematics, 304 Washington Rd, Princeton, NJ 08540, USA}
\email{yujin.kim@princeton.edu}

\address{ C. Serio
\newline
University of California, Berkeley, Department of Mathematics, 970 Evans Hall, Berkeley, CA 94720, USA}
\email{cdserio@berkeley.edu}
\date{\today}

\begin{document}

	\begin{abstract}
		This paper analyzes a random walk model for the level lines appearing in the entropic repulsion phenomena of three-dimensional discrete random interfaces above a hard wall; we are particularly motivated by the low-temperature (2+1)D solid-on-solid (SOS) model, where the emergence of these level lines has been rigorously established. The model we consider is a line ensemble of non-crossing random walk bridges above a wall with geometrically growing area tilts. Our main result, which in particular resolves a question of Caputo, Ioffe, and Wachtel (2019), is an edge 1:2:3 scaling limit for this ensemble as the domain size diverges, with a growing number of walks (including the number of level lines of the SOS model) and high boundary conditions  (covering the maximum upper deviation of the SOS level lines).  As a key input, we establish Tracy--Widom-type upper tail bounds for each of the relevant curves in the line ensemble. An ingredient which may be of independent interest is a ballot theorem for random walk bridges under a broader range of boundary values than available in the literature.
	\end{abstract}

	{\mbox{}
\maketitle
}
\vspace{-.5cm}
	
	\setcounter{tocdepth}{1}
	\tableofcontents

%!TEX root = ./rwareatilt.tex
\section{Introduction}\label{sec:introduction}

In this paper, we study a random walk model of the level lines emerging from entropic repulsion phenomena in the (2+1)-dimensional solid-on-solid (SOS) model above a hard wall at low temperature. Similar entropic repulsion phenomena are expected to appear in a host of low-temperature interfaces; see \cite{IV18} for discussion. The random walk model in question was introduced by Caputo, Ioffe, and Wachtel in \cite{CIW19a}. 
Considering the SOS model in a box of side length $N$, the level lines appear to fluctuate at scale $N^{1/3}$ on intervals of length $N^{2/3}$; moreover, they resemble a line ensemble of $\asymp \log N$ many non-crossing random walks with a floor at zero and geometrically growing area tilts. Rigorously establishing this picture for the SOS model has been a longstanding open problem, with the ultimate goal being to show a \textit{1:2:3 edge scaling limit} for these level lines. Our main result (\cref{thm:simplifiedmain}) is such a scaling limit for the random walk model, along with Tracy--Widom-type upper tail bounds for the top curves (\cref{thm:simplifiedmain2}). 

Our work is motivated by the great success that has been found in analyzing the (single) area-tilted random walk, and then transferring those results to (single) low-temperature interfaces via an Ornstein--Zernike coupling (see Section \ref{sec:rwmodel}). It is very reasonable to expect a similar transferring of results to be possible from {\emph{line ensembles}} of {area-tilted} random walks to {ensembles of many} interacting low-temperature interfaces, such as the SOS level lines. However, all previous results on such line ensembles either considered a finite number of random walkers or considered Brownian lines: both scenarios are in fact vast simplifications that make a transferring of results to the SOS level lines implausible (see Section~\ref{s.related} for a discussion of related literature).

Our results are the first on the area-tilted random walk ensemble with geometrically growing tilts and a \emph{diverging} number of walks. They not only give the (non-Gaussian) random walk analogue of the convergence results of \cite{CIW19b, CG23}, but also give quantitative information including tail bounds and bounds on the time to equilibrium from high points, which among other things suggest that a transfer of information from area-tilted walks to SOS level lines {may} be possible. This is {explained in detail} in Remarks \ref{rk:parabolic-decay}
 and \ref{rk:generalizations-of-theorems}. {Moreover, these are the first quantitative estimates in the case of more than a single walk.}

{The introduction is organized as follows. We discuss in \cref{s.background} the SOS model, entropic repulsion in dimension $2+1$, the aforementioned level lines, and the connection with line ensembles of area-tilted random walks. We state our main results in \cref{sub:main_result}, discuss related literature in \cref{s.related}, and outline the proof of our main theorems in \cref{sub:proof_ideas}.}

\subsection{Background and motivation}\label{s.background}
The $(2+1)$D SOS model with zero boundary conditions and a hard wall at height $0$ is a random surface $\varphi:\intint{-N,N}^2\to\Z_{\geq 0}$, where $\llbracket -N, N\rrbracket := [-N,N]\cap\mathbb{Z}$, with the probability of a given instance of $\varphi$ proportional to
\begin{align}\label{e.sos measure}
\exp\Big(-\beta\sum_{x\sim y} |\varphi(x) - \varphi(y)|\Big) \,.
\end{align}
Here, $\beta>0$ is the inverse temperature, the sum is over all pairs of adjacent sites $x,y \in \Z^2$, and we set $\varphi(x)=0$ for all $x\not \in \intint{-N,N}^2$ (zero boundary conditions). The ``hard wall'' at height $0$ refers to the restriction of the values of $\varphi$ to the nonnegative integers $\Z_{\geq 0}$.
The model (without a wall) was introduced in the 1950s (see \cite{BCF51,Temperley52}) as a low-temperature (large $\beta$) approximation of the random surface separating the macroscopic $+$ and $-$ phases in the 3D Ising model on a cube with Dobrushin boundary conditions: $+$ spins above the $xy$-plane and $-$ spins below it. 

Dimension $2+1$ for the SOS model is of special importance, as it is the unique dimension in which the no-floor model exhibits a \emph{roughening transition}: as $N\to\infty$, at low temperature the surface stays flat and localized ($\varphi(x) =O(1)$ for typical $x$), while at high temperature the surface becomes rough and delocalized ($\varphi(x)$ diverges for typical $x$). These two phases were established in \cite{brandenberger1982decay} and \cite{frohlich1981kosterlitz,frohlich1981kosterlitz2} respectively; establishing the roughening transition for the 3D Ising model remains a major open problem. See also  \cite{lammers22} showing sharpness of the phase transition.

\subsubsection{Entropic repulsion in the SOS model above a hard wall}\label{ssub:entropic_repulsion_in_the_sos_model_with_hard_wall}
Introducing a wall at $0$ in the low-temperature setting of the $(2+1)$D SOS model results in behavior drastically different from the flat geometry of the no-wall model, in a phenomenon known as \emph{entropic repulsion}. This was first observed by \cite{BEF86}, who showed that the surface raises in height so that  $\varphi(x)$ is typically of order $\log L$ for $x$ in the bulk, due to the increased entropy of the class of such height functions. 

Subsequently, a more detailed picture was obtained in \cite{CLMST12,CLMST14, caputo2016scaling}, described as follows. For large enough $\beta$, with high probability, the surface becomes rigid at a deterministic height $\mathfrak{h}$, where $\mathfrak{h}$ is either $H(N) := \smash{\lfloor \frac{1}{4\beta}\log (2N) \rfloor}$ or  $H(N)-1$.
Moreover, the surface resembles a wedding cake, featuring $\mathfrak{h}$ layers of height $1$ stacked on top of one another. Viewed from above, the boundaries of the layers form a collection of concentric loops in $\R^2$ which are called \emph{level lines}, as the boundary of the layer at height $h \in \{1,\dots,\mathfrak{h}\}$ delineates between macroscopic phases of height $\leq h$ and $\geq h+1$ in the surface. The innermost level line loop encloses a $(1-\ep_{\beta})$-fraction of sites in $\llbracket -N, N \rrbracket^2$, where $\ep_{\beta}\downarrow0$ as $\beta \uparrow\infty$. Upon scaling $[ -N, N]^2$ to $[-1,1]^2$, the level lines admit a scaling limit consisting of a nested collection of \emph{Wulff shapes}, which in particular feature four flat facets: the Wulff shapes all coincide with the boundary of the box except near the four corners, as illustrated in Figure \ref{fig:level-lines}. 

\subsubsection{Fluctuations of SOS level lines}\label{ssub:fluctuations_of_sos_level_lines}
We are interested in the fluctuations of the level lines away from these flat facets. In particular, zoom in, say, around the center of the bottom boundary of the box. The level lines restricted to this region form a stacked collection of non-crossing open contours, so that the top contour $\gamma_1$ in this stack is the restriction of the innermost level line to this region.
From now on, level lines will be identified with their restriction to this region.
Let $\gamma_i$ denote the $i^{\mathrm{th}}$ level line from the top of this stack.

\begin{figure}
\includegraphics{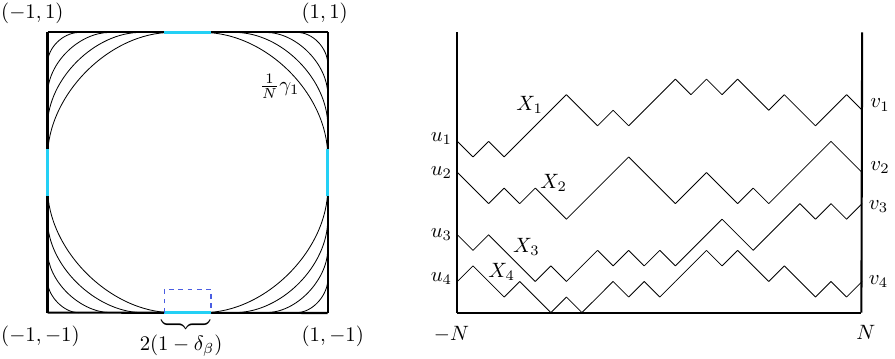}
	\vspace{-0.1 in}
	\caption{Left: an illustration of the limit shape of the SOS level lines after rescaling to $[-1,1]^2$. The loops $\frac{1}{N}\gamma_i$ converge to the nested Wulff shapes shown here. The limit shape is flat for all level lines on central $(1-\delta_\beta)$-portions of each side of the box (cyan), where $\delta_\beta\downarrow 0$ as $\beta\uparrow\infty$. Right: the effective random walk model \eqref{areatiltRW} of the level line fluctuations about their flat limit shape, as seen by zooming in on the blue dashed region in the figure on the left. In this region the level lines appear as a stack of ordered open contours above a floor. In an effective approximation one assumes these contours to be height functions with respect to the horizontal axis (ignoring possible microscopic overhangs), as shown here.}\label{fig:level-lines}
\end{figure}

In \cite{caputo2016scaling}, the fluctuations of $\gamma_1$ were shown to be, with high probability, at most $N^{1/3+\ep}$ above the bottom boundary, for any $\ep>0$.
In the recent work \cite{caddeo2024level}, the fluctuations of $\gamma_1$ were shown to be bounded from below by $N^{1/3}$ on intervals of size $N^{2/3}$.
It is an interesting open problem to prove a matching upper bound, which is complicated by the fact that the $\mathfrak{h}-1 \asymp \log N$ many level lines below $\gamma_1$ have the capacity to push $\gamma_1$ upwards. However, our results here for the area-tilted random walk model suggest that $N^{1/3}$ should indeed be the correct scale.

\subsubsection{SOS level lines and area tilting}
The following effective ``area tilt'' description of the joint law of the level lines was derived in \cite{caputo2016scaling}: 
\begin{equation}\label{areatiltZ}
	\mathbb{P}(\gamma_1,\dots,\gamma_\mathfrak{h}) \propto \exp\bigg(-\sum_{i=1}^\mathfrak{h} \beta |\gamma_i| - \sum_{i=1}^\mathfrak{h} \frac{a}{N} b^{i-1} |A_i| \bigg) \prod_{i=1}^\mathfrak{h} \hat{Z}_{\Delta_i} \,,
\end{equation}
where $a = a_\beta > 0$ and $b=e^{4\beta}$ are constants (note $\mathfrak{h} \asymp \log_b N$), $A_i$ is the region outside of $\gamma_i$, and $\hat{Z}_{\Delta_i}$ is the partition function for a certain SOS model with no wall on $\Delta_i := A_{i-1} \setminus A_{i}$.
We briefly sketch the origin of \eqref{areatiltZ}; see \cite[Section~1.3]{caputo2016scaling}, \cite[Section~1.2]{CIW19a} for further details.

The first term is an energy contribution that arises directly from the SOS Hamiltonian, since the surface height jumps by height (at least) $1$ across each level line. The second term is referred to as a \textit{geometric area tilt}: $|A_i|$ is the area of the region outside of the contour $\gamma_i$, and the prefactor $\frac{a}{N}b^{i-1}$ is geometrically increasing with the curve index. This term arises from the hard wall constraint: on a heuristic level, as level lines get closer to the boundary, they encompass more sites, increasing the entropy as these sites can feature pits going down towards the hard wall. Sites encompassed by level lines of higher curve index (corresponding to lower layers of the surface) cannot feature as-large pits, and so such level lines feel a greater reward for the area they encompass, or equivalently, a larger penalty for the area $|A_i|$ they do not encompass. 

After removing these area terms, one is left with the no-wall factors $\hat{Z}_{\Delta_i}$. Cluster expansion can be applied to each of these, resulting in a complicated interaction term $\Phi(\gamma_1, \dots, \gamma_{\mathfrak{h}})$, which itself is made up of terms that cause the level lines to both repel and attract. In principle, the effect of $\Phi$ should be negligible in the low-temperature regime,  though rigorously showing this is highly nontrivial. See \cite{ISV14,caddeo2024level} where this is shown for an SOS model with a single level line.

\subsubsection{A random walk model}\label{sec:rwmodel}
Per the last paragraph, we shall ignore the $\hat{Z}_{\Delta_i}$ terms in \eqref{areatiltZ} to produce our effective random walk model. 
Return to the region of the box  around the center of its bottom boundary, as in \cref{ssub:fluctuations_of_sos_level_lines}, where the level lines form a stack of non-crossing open contours and are asymptotically flat. Rescale the box and shift it upwards so that the bottom boundary in this region is $\llbracket -N, N \rrbracket \times \{0\}$. \cref{areatiltZ} leads to our random walk model as follows.

If one assumes that the contours $\gamma_i$ are single-valued, i.e., are height functions on $\llbracket -N, N\rrbracket$, then the factor $\exp(-\beta|\gamma_i|)$ from \eqref{areatiltZ} is simply the probability that a random walk with transition probability $p(x) \propto \exp(-\beta|x|)$ follows the trajectory $\gamma_i$.
(Indeed, random walk behavior is a feature of low-temperature interfaces \cite{IV18}. In a variety of models featuring a single interface, a coarse-graining strategy known as Ornstein--Zernike theory has been implemented to actually obtain a coupling between the interface and a random walk; this was pioneered by Campanino, Ioffe, and Velenik in \cite{CampaninoIoffe02,CIV03,CIV08} and implemented in the setting of the SOS model with boundary conditions resulting in a single level line in \cite{IV08,IST,caddeo2024level}.) The term $|A_i|$ is then the area enclosed between the random walk and the floor at $0$. 

This leads to the following random walk model for the SOS level lines; see Figure \ref{fig:level-lines} for an illustration. Fix $a>0$ and $b>1$. Let $\mathbb{P}_{n,N}^{\mathbf{u},\mathbf{v}}$ denote the law of $n = n(N) \to \infty$ (for instance $n = H(N)$) many independent mean zero random walk bridges $X_1, \dots, X_n$ on $\llbracket -N, N\rrbracket$ with a fixed common transition probability and suitable boundary conditions $\mathbf{u} = (u_1 \geq \dots \geq u_n \geq 0),\mathbf{v} = (v_1 \geq \dots \geq v_n \geq 0) \in \R^n$ at $-N$ and $N$ (for instance $\max(u_1,v_1) \leq N^{1/3+\varepsilon}$, in accordance with the aforementioned upper bound on SOS level line fluctuations). We then define a new non-crossing area-tilted line ensemble measure $\mathbb{P}_{n,N;0}^{a,b;\mathbf{u},\mathbf{v}}$ by the Radon--Nikodym derivative
\begin{equation}\label{areatiltRW}
	\frac{\mathrm{d}\mathbb{P}_{n,N;0}^{a,b;\mathbf{u},\mathbf{v}}}{\mathrm{d}\mathbb{P}_{n,N}^{\mathbf{u},\mathbf{v}}} \propto \exp\left(-\frac{a}{N}\sum_{i=1}^n b^{i-1} \mathcal{A}(X_i)\right) \one_{X_1(j) \geq \cdots \geq X_n(j) \geq 0 \;\; \forall j\in \llbracket -N, N \rrbracket},
\end{equation}
where $\mathcal{A}(X_i) = \sum_{j=-N}^{N} X_i(j)$ is the (discrete) area between $X_i$ and the floor at zero.

\subsubsection{Candidate for the scaling limit}
\label{ssub:candidate_for_the_scaling_limit}
Though establishing an edge scaling limit (or even tightness around $N^{1/3}$) for the law of the SOS level lines remains a challenging open problem, one can construct a candidate measure for the scaling limit. A natural candidate is the putative diffusive scaling limit of \eqref{areatiltRW} on compact windows as $N\to\infty$.
The definition \eqref{areatiltRW} can be naturally viewed as the prescription of a Gibbs measure. The Gibbs property is simply that for an interval $\intint{j,k}$ of curve indices on a spatial interval $\intint{\ell,r}\subseteq \intint{-N,N}$, the conditional distribution of $(X_{j}, \ldots, X_k)|_{\intint{\ell,r}}$, given $\{X_i(x): (i,x)\not\in \intint{j,k}\times \intint{\ell,r}\}$, is that of $k-j+1$ independent random walk bridges with endpoints prescribed by the conditioning, reweighted by the area tilt factor $\exp(-\frac{a}{N}\sum_{i=j}^k b^{i-1} \sum_{j=\ell}^r X_i(j))$ and conditioned to not cross each other, $X_{j-1}$, or $X_{k+1}$. Therefore, one would expect the scaling limit of \eqref{areatiltRW} to be an infinite-volume Gibbs measure involving Brownian paths.
 
Such an infinite-volume Gibbs measure $\mu_{a,b}$ was constructed, and later characterized, across a series of works \cite{CIW19a,CIW19b,DLZ,CG23} via a continuum Brownian analogue of the model \eqref{areatiltRW} with $n$ lines on a compact interval $[-T,T]$, sending $n,T\to\infty$. The measure $\mu_{a,b}$ thus has a Gibbs property in terms of area-tilted, non-intersecting Brownian bridges; see \cref{def:gibbs}.

\subsection{Main results}\label{sub:main_result}

The goal of the present paper is to establish convergence to the measure $\mu_{a,b}$, described in \cref{ssub:candidate_for_the_scaling_limit}, under 1:2:3 scaling for the area-tilted random walk ensemble with a wide range of boundary conditions and number of curves. This in particular resolves a question posed in \cite[Section 3.5.3]{CIW19b}. The discrete, non-Gaussian nature of the model poses several challenges compared to the Brownian model treated in past works; {we expand on this and emphasize several significant features of our results in \cref{rk:parabolic-decay,rk:generalizations-of-theorems}}.

We give a slightly simplified version of our main result here, with a more general statement given in Theorem \ref{thm:main} below. 
{The result applies to a broad class of both lattice and nonlattice random walks satisfying certain technical conditions, detailed in Assumptions~\ref{a.convex} and~\ref{a.misc}.
In short, we assume the increment distribution has a log-concave probability mass/density function,
with mean $0$ and finite exponential moments (in particular finite variance $\sigma^2$). The log-concavity
condition (Assumption~\ref{a.convex}) is needed for the sole purpose that it implies a stochastic monotonicity
result (Lemma~\ref{l.monotonicity}); that is, we could have assumed the conclusion of Lemma~\ref{l.monotonicity} and removed
Assumption~\ref{a.convex}. For instance, our results also apply to the simple random walk, which satisfies
monotonicity as shown in the proof of Lemma~\ref{l.monotonicity}.}

Here we let $W_0^n := \{\mathbf{x}\in\mathbb{R}^n : x_1 \geq \cdots \geq x_n \geq 0\}$ denote the nonnegative Weyl chamber in $\mathbb{R}^n$.

\begin{maintheorem}\label{thm:simplifiedmain}
	Fix $a>0$, $b>1$, and $\ep\in(0,1)$. For $N\in\mathbb{N}$, consider the line ensemble of $n$ area-tilted random walk bridges $(X_1, \dots, X_n)$ on $\llbracket -N,N\rrbracket$ with law $\mathbb{P}_{n,N;0}^{a,b;\mathbf{u},\mathbf{v}}$ as given by \eqref{areatiltRW}, satisfying Assumptions \ref{a.convex} and \ref{a.misc}, and with boundary conditions $\u,\v \in W_0^n$ satisfying $\max(u_1, v_1) \leq N^{1-\varepsilon}$. There exists $\delta =\delta(\ep)>0$ such that for any sequence $n = n(N) \in \N$ satisfying $n\to\infty$ as $N\to\infty$ and $n\leq N^{\delta}$, the law of the rescaled line ensemble $(x_1^N,\dots,x_n^N)$ given by $x_i^N(t) := \sigma^{-2/3} N^{-1/3} X_i(t\sigma^{-2/3}N^{2/3})$ converges weakly as $N\to\infty$  to the infinite-volume Gibbs state $\mu_{a,b}$, in the topology of uniform convergence on compact sets.
\end{maintheorem}

Above, we make the random walk trajectories of the line ensemble continuous functions by linear interpolation. The topology of uniform convergence on compact sets means we require convergence of the law of any finite number of curves from the top on any compact interval.

A crucial input for the proof of \cref{thm:simplifiedmain} is a quantitative estimate: an upper tail bound on the top curves of the line ensemble away from $N^{1/3}$ on intervals of size $N^{2/3}$. This constitutes our second main result, \cref{thm:simplifiedmain2}, which features the Tracy--Widom right tail exponent $3/2$. A more general statement is given ahead as Theorem~\ref{thm:max}.

\begin{maintheorem}\label{thm:simplifiedmain2}
Suppose the random walk increment distribution satisfies Assumptions~\ref{a.convex} and \ref{a.misc}. Fix $a_0>0$, $b_0 > 1$, $\ep \in(0,1)$, and $t\in \mathbb{R}$.  There exist positive constants $\delta = \delta(\ep)$, $K_0 = K_0(a_0,b_0, \ep)$, and $c = c(a_0,b_0)$ such that the following holds. For any $a\geq a_0$, $b\geq b_0$, there exists $N_0 = N_0(a,b,\varepsilon,t)\in\mathbb{N}$ so that for all $N\geq N_0$, $n\leq N^{\delta}$, $\u,\v\in W_0^n$ with $\max(u_1,v_1) \leq N^{1-\ep}$, $j \leq \min(n, \frac{\ep}{10}\log_b N)$, and $K \in [K_0, N^{2/3-\varepsilon}]$,
\begin{align*}
\mathbb{P}_{n,N;0}^{a,b;\mathbf{u},\mathbf{v}}\left(X_j(tN^{2/3}) > K a^{-1/3}b^{-(j-1)/3} N^{1/3}\right) \leq e^{-cK^{3/2}}.
\end{align*}
\end{maintheorem}

The proofs of Theorems \ref{thm:simplifiedmain} and \ref{thm:simplifiedmain2} are given following the statement of Theorem \ref{thm:max}.

\begin{remark}[Parabolic decay of the random walks]\label{rk:parabolic-decay}
\cref{thm:main,thm:max} are the expanded  versions of Theorems \ref{thm:simplifiedmain} and \ref{thm:simplifiedmain2}. There, we state the results for line ensembles on intervals with length as small as $LN^{2/3}$, with $L \geq (\log N)^{1/3+\gamma}$, and boundary conditions that may all be as large as $L^{2-\kappa}N^{1/3}$, for arbitrarily small $\gamma,\kappa>0$. Modulo the $\kappa$ factor, this implies that after 1:2:3 scaling, all curves of the line ensemble decay \emph{parabolically} (on average) from high points before reaching equilibrium (i.e., descend from $L^{2-\kappa}$ to $O(1)$ in time $L$). This is expected to be the truth: for an area-titled Brownian excursion, such parabolic decay follows quickly from a Girsanov transformation (see \cref{rk:generalizations-of-theorems}). Because we do not have any analogue of Girsanov for our model, a crucial step is to develop a more robust argument to establish this roughly parabolic decay. 
\end{remark}

\begin{remark}[Range of parameters, relevance to SOS level lines]\label{rk:generalizations-of-theorems}
Here, we describe how the ranges of our parameters in Theorems \ref{thm:main} and \ref{thm:max}, discussed in \cref{rk:parabolic-decay},
were chosen sufficiently wide so as to model a variety of features of the SOS level lines.

First, we treat boundary conditions much larger than the typical scale $N^{1/3}$ due to the aforementioned  existing upper bound on the fluctuation of the top SOS level line being $N^{1/3+\ep}$, for any $\ep >0$ (see \cref{ssub:fluctuations_of_sos_level_lines}, \cite[Theorem~3]{caputo2016scaling}). Moreover, near the corners of $[-N,N]^2$, the SOS level lines lie at distance $O(N)$ from the boundary of the box. Because our results treat boundary conditions up to $N^{1-\ep}$ on intervals of size $N$, our random walk model describes the behavior of the level lines as soon as one looks away from the corners (i.e., as soon as the level lines become $\ll N$).
	
This range of boundary conditions is nearly optimal for  establishing convergence to the \textit{time-stationary} Gibbs state $\mu_{a,b}$, which is the result of  \cref{thm:simplifiedmain}.
For the Brownian model on an interval $[-T,T]$, the recent work \cite{CCG23} shows that if one pushes to the parabolic scale, there is a two-parameter family of limiting Gibbs states $\mu_{a,b}^{L,R}$ corresponding to boundary conditions $T^2 - LT$ at $-T$ and $T^2 - RT$ at $T$, where $L+R > 0$. In our setup on the interval $[-N,N]$ this corresponds to boundary conditions $N - LN^{2/3}$ and $N - RN^{2/3}$. Our bound of $N^{1-\ep}$ shows that essentially any boundary conditions below this scale will result in the stationary Gibbs state.
	
Second, the significance of considering the random walk model on intervals as small as scale\footnote{To be precise, we can handle  intervals as small as $N^{2/3}(\log N)^{1/3+\gamma}$, for any $\gamma >0$, as laid out in \eqref{def:boundary-conditions-thm:max}.} $N^{2/3}$ (as opposed to $N$, which is the size of the flat facets of the SOS level lines as mentioned in \cref{ssub:entropic_repulsion_in_the_sos_model_with_hard_wall}, and would in fact be easier to analyze due to the much slower rate of descent required of the walks) is as follows. With existing technology, a representation such as \eqref{areatiltZ} for the SOS level lines as area-tilted polymers is only valid in boxes of side length at most $N^{2/3+\eta}$, where $\eta \in (0,1/10)$ (see \cite[Proposition~A.1]{caputo2016scaling}). Therefore, currently, any hope of obtaining a 1:2:3 scaling limit of the SOS level lines by coupling with area-tilted random walks (e.g., via the Ornstein--Zernike machinery) must occur on intervals of size $\leq N^{2/3+\eta}$. It is therefore of interest to determine whether or not a 1:2:3 scaling limit (which we recall refers to vertical scaling by $N^{1/3}$ and horizontal scaling by $N^{2/3}$) for the random walk model could be obtained on such small intervals, even with boundary conditions reflecting the ``worst-case scenario'' of the known upper bound on the maximum deviation of the top SOS level line ($N^{1/3+\ep}$, far above the size of the vertical scaling), and a diverging number of walks (exceeding $\mathfrak{h} \asymp \log_b N$). \cref{thm:main,thm:max} answer this in the affirmative, {suggesting that it may be a possible and fruitful strategy to attempt to couple the SOS level lines with our random walk model.\footnote{Indeed, some time after this paper was first posted, \cite{chen2025limiting} established a scaling limit in the setting of the $|\nabla \varphi|^p$ model, $p>1$, with arguments that are, at a very high level, the strategy outlined here. 
% ($p=1$ is the SOS model; for $p>1$, the level lines do not interact, so the analysis boils down to analyzing a single interface/random walk with high boundary conditions). 
See Section~\ref{s.related} for more.}}

{Lastly, we mention that our result addresses up to $n = N^\delta$-many walks, for suffciently small $\delta>0$, which is well beyond the $\mathfrak{h} \asymp \log_b N$-many SOS level lines. The $N^\delta$ upper bound is a result of our proof strategy, and we did not attempt to loosen this condition.}
\end{remark}

\begin{remark}\label{rk:tracy-widom}
Our arguments show that the same one-point tail bound in \cref{thm:simplifiedmain2} holds for $X_j(x)$ at any $x \in \llbracket -(1-\eta)N, (1-\eta)N\rrbracket$, for any $\eta \in (0,1)$. The bound for the top curve $X_1$ is optimal up to the constant $c$\hspace{.08em}: the matching  lower bound follows from an estimate on a single area-tilted walk, discussed in Remark~\ref{rmk:tail-lbd}. Furthermore, for each $j$, the fluctuation scale and $3/2$ exponent in the tail decay match those obtained for the Brownian model in \cite[Corollary 3.2]{CG23} (they also obtain an explicit value for $c$ in the exponent). Lastly, for each $j$, the fluctuation scale $a^{-1/3} b^{-(j-1)/3}$ is also the optimal fluctuation scale for a \emph{single} area-tilted random walk with parameter $ab^{j-1}$, the same area tilt experienced by $X_j$; see \cite[Theorem~1.2]{HV04} or \cref{oneptbd}.
\end{remark}

\subsection{Related work}\label{s.related}
A number of works have considered idealized models of interfaces in terms of non-crossing random walks with area tilts. The case of a single random walk conditioned to stay above a floor at zero subject to an area tilt was considered in \cite{HV04,ISV14} (including generalized nonlinear area tilts). The work \cite{HV04} established estimates on the upper tail, enclosed area, decorrelation, and mixing. Under 1:2:3 scaling as above, \cite{ISV14} showed weak convergence as $N\to\infty$ to a stationary diffusion on $\R$ called the \emph{Ferrari--Spohn} (FS) \emph{diffusion}, first constructed in \cite{ferrari2005constrained} as the scaling limit of the gap between a Brownian bridge and a parabolic floor. This single area-tilted walk can be understood as a proxy for the interface in the 2D Ising model at critical prewetting or in the SOS model above a hard wall with one level line. The Ising interface was coupled with an area-tilted random walk and shown to converge under 1:2:3 scaling to the FS diffusion in \cite{IOSV}; analogous results for the single SOS level line were shown in \cite{caddeo2024level}.

A model of multiple interacting area-tilted walks as a proxy for SOS level lines was first analyzed in \cite{IVW}. The model considered there consists of a fixed number $n$ of non-intersecting random walks subject to a \emph{common} area tilt, i.e., with $b=1$ in \eqref{areatiltRW}; we refer to \cite{IV18,FS23} for interpretations of this model in terms of Ising-type interfaces without a hard wall. Weak convergence was shown in \cite{IVW} to a determinantal process consisting of $n$ non-intersecting copies of the Ferrari--Spohn diffusion, known as the \emph{Dyson Ferrari--Spohn diffusion}. It was recently shown in \cite{ds24} that as the number of curves $n$ tends to $\infty$, the Dyson FS diffusion converges after a vertical shift to the \emph{Airy line ensemble}, a fundamental object in the KPZ universality class first constructed in \cite{corwin2014brownian}. A significant technical simplification arises from the adoption of a common area-tilt parameter (taking $b=1$) in these works, as the marginal distribution of the curves (before imposing non-intersection) is the same, which allows access to integrable tools such as the Karlin-McGregor formula. This determinantal structure is not present in the case $b>1$ considered here.

In \cite{CIW19a,CIW19b} was proposed a model which considers geometrically growing area tilts, albeit with Brownian motions instead of random walks, in a further idealization of the discrete SOS level lines.  Let us describe the model in more detail. Let $n\in\N$, $x_1 \geq \cdots \geq x_n \geq 0$, and $y_1 \geq \cdots \geq y_n \geq 0$, and fix parameters $a>0$ and $b>1$. Consider the probability measure given by $n$ independent Brownian bridges $B_1, \ldots, B_n$ on an interval $[-T,T]$, with $B_i$ having boundary values $x_i$ and $y_i$, reweighed by the Radon--Nikodym derivative proportional to
\begin{align}\label{eqn:BrownianLE-informal}
\exp\bigg(-a\sum_{i=1}^n b^{i-1} \mc A(B_i)\bigg)\one_{B_1(t) > \cdots > B_n(t) > 0 \;\; \forall t\in(-T,T)},
\end{align}
where $\mc A(B_i) = \int_{-T}^T X_i(t)\,\mrm{d}t$ is the area enclosed between $B_i$ and the floor at 0. These objects are referred to in short as ``Brownian polymers.'' In
\cite{CIW19a}, weak convergence of the Brownian polymer law to a unique limit, namely the measure $\mu_{a,b}$ discussed above, was established as $T,n\to\infty$ for the special case of zero boundary conditions. For the case of ``free'' boundary conditions, tightness in $T,n$ was established for the model in \cite{CIW19b}, and \cite{DLZ} established convergence to the same law $\mu_{a,b}$, albeit with the requirement that $T\to\infty$ before $n\to\infty$ (unlike in the SOS model where $n$ diverges with $T$). The recent work \cite{CG23} strengthened these results, in particular proving convergence to $\mu_{a,b}$ for free boundary conditions as $T,n\to\infty$ in any order. 
Various important properties of the limit law $\mu_{a,b}$ were additionally proven in \cite{CG23}, including Tracy--Widom-type upper tails with exponent $3/2$, as well as strong mixing and ergodicity. That work also proved a strong characterization of $\mu_{a,b}$, showing in particular that it is the unique such infinite-volume Gibbs state which is stationary in time (up to deterministic vertical shifts). We also mention again the recent work \cite{CCG23}, which establishes a complete characterization of the class of area-tilted Brownian Gibbs states; see~\cref{rk:generalizations-of-theorems} above. {We note that each of these works rely in a significant way on Brownian tools, such as scaling invariance, exact Gaussian formulas, and Girsanov transformations, which are not available in the random walk case.}

Returning to the random walk setting, the work \cite{ser23} made progress on the scaling limit for area-tilted walks, but like \cite{DLZ}, required $N\to\infty$ before $n\to\infty$. This allowed the use of rather strong mixing estimates adapted from \cite{IVW}, which have complicated dependence on the number of curves and are only effective for $n$ fixed. By contrast, in Theorem~\ref{thm:simplifiedmain} we allow $n$ to diverge simultaneously with $N$, as is expected in the SOS model, requiring a substantially more delicate analysis. In particular, the mixing estimates used in \cite{ser23} bypass the need for any other quantitative estimates on the line ensemble with $n$ diverging, such as our \cref{thm:simplifiedmain2} here.

{
Finally, we mention that after this paper was originally posted, \cite{chen2025limiting} established a 1:2:3 edge scaling limit of the level lines of the low temperature $|\nabla \varphi|^p$ model, $p>1$. In stark contrast with the SOS model (which is the case $p=1$), they show that the level lines  of the $|\nabla\varphi|^p$ model separate macroscopically when $p>1$ and therefore do not interact. Thus, their analysis reduces to considering a \emph{single} level line, and showing coupling with a \emph{single} area-tilted random walk, as opposed to an ensemble of many interfaces (more precisely, they separately show stochastic domination of and by an area-tilted random walk). Establishing such a coupling for a single interface is already a challenging task. But, since the arguments in their work can be regarded as a version of the strategy laid out in Remark~\ref{rk:generalizations-of-theorems} (as mentioned in a footnote there) in the setting of a single interface, their success hints at the potential for our results on line ensembles of random walks to be transferred to the SOS level lines.
}

\subsection{Proof ideas}\label{sub:proof_ideas}

In this section we give a brief overview of the ideas underlying the proof of Theorem \ref{thm:simplifiedmain}. Ultimately, we must establish the existence of weak subsequential limits of the ensemble and that all such subsequential limits are $\mu_{a,b}$. For the latter, we make use of \cite[Theorem 3.7]{CG23} (Theorem~\ref{muexist} here), which roughly says that $\mu_{a,b}$ is the unique law on line ensembles of $\N$-indexed continuous curves with the above Brownian Gibbs property for which the $j$\th curve converges to $0$ as $j\to\infty$ and for which there is uniform one-point upper tail control on the top curve. 

Thus the main points to establish are (i) that the prelimiting rescaled line ensemble is tight, and (ii) upper tail bounds on the $j$\th curve. Of course, (ii) is needed to establish (i), so we begin there and outline the proof of \cref{thm:simplifiedmain2}.

\subsubsection{Reduction to single-curve estimates via stochastic monotonicity}
\label{ssub:reduction_to_single_curve_estimates_via_stochastic_monotonicity}
The basic strategy is as follows: construct a sequence of ``ceiling'' functions $\Cl_j$ such that, with high probability (w.h.p.), $X_j(x) \leq \Cl_j(x)$ for all $j$ and $x$. We do this recursively, via a sequence of estimates for a \emph{single} area-tilted  walk. 
Suppose we have already shown $X_{j+1} \leq \Cl_{j+1}$ w.h.p. To show $X_{j}(x) \leq \Cl_{j}(x)$ for all $x$, we apply the Gibbs property and a stochastic monotonicity statement for area-tilted line ensembles (\cref{l.monotonicity,rmk:remove-top}) to replace $X_{j+1}$ with a floor at $\Cl_{j+1}$ and also remove the top $j-1$ curves, resulting in a single area-tilted walk. Thus, it suffices to prove the following:
letting $\lambda_j$ denote the area-tilt parameter of $X_j$, a walk with area tilt $\lambda_{j}$ conditioned to stay above $\Cl_{j+1}$ is pointwise dominated by $\Cl_{j}$ w.h.p.\ summable in $j$.  This is \cref{prop:recursive-bound}. This stochastic monotonicity property plays a central role in our arguments and has also been a common feature of the majority of models of line ensembles considered in the past. The condition on the random walk increments that we impose to ensure this property is Assumption~\ref{a.convex}.

The construction of the ceiling $\Cl_{j}$ on a random walk with area tilt $\lambda_{j}$ given a floor at $\Cl_{j+1}$ is based on the following idea: at a typical point $x$ in the ``bulk'' of the interval (away from the high boundary conditions), the walk should stay at height $\approx H_{j} + \Cl_{j+1}(x)$, where $H_{j} := \lambda_j^{-1/3}N^{1/3}$ is the fluctuation scale for a random walk with area tilt $\lambda_{j}$ and floor at $0$. Thanks to the geometric relation $\sum_{i \geq j} H_i \leq CH_j$, one can then hope to recursively construct each $\Cl_j$ on scale $H_j$. As the interval is much longer than $H_j^2$, atypically high fluctuations at random locations force an extra logarithmic factor, and ultimately in the bulk we obtain $\Cl_j(x) \approx H_j \log (|x|/H_j^2)^{2/3}$. In particular, we choose our ceilings to be non-decreasing in $|x|$. We now explain how we construct these ceilings.

\subsubsection{High boundary conditions and dropping estimates} \label{ssub:high_boundary_conditions_and_dropping_estimates}
It is a key result requiring several inputs that, as a consequence of the area tilt, the walk $X_j(x)$ drops as $|x|\to0$ just as $\Cl_{j+1}(x)$ does. This is far from immediate due to the lack of a Girsanov transformation, which for Brownian models converts the area tilt into a parabolic drift like $-x^2$. 
Moreover, complications arise because we consider very high boundary conditions (far above the usual fluctuation scale $H_{j}$) on relatively short intervals, i.e., we need the dropping to occur at a certain speed. Allowing such boundary conditions is motivated  in \cref{rk:parabolic-decay}; however, it is in fact \emph{necessary} to consider boundary conditions at scale $H_j(\log N)^{2/3} \gg H_{j}$, roughly because we must perform Gibbs resampling at many random points, which may have atypically high fluctuations   as mentioned above. Treating such high boundary conditions precludes us from directly applying existing results on area-tilted random walks. 

In Lemma~\ref{lem:dropping-lemma}, the ``dropping lemma,'' we manifest the area tilt in a dropping effect for random walk bridges on intervals $I$ above a \emph{floor at zero}. The lemma constructs a ``drop point'' $x$ such that $X_j(x)$ lies at scale $CH_{j}$ above the flat floor; however, the location of this point $x$ is random inside an interval of length $C^{-1}|I|$, and also in our application the floor is not flat (it is $\Cl_{j+1}$). Thus, there is some delicacy in the application of the lemma. 
Using another recursive scheme, we produce a mesh of appropriately spaced deterministic points along which the walk drops successively from its high boundary conditions. At mesh points $x(k)$ in the bulk, we show $X_j(x(k)) \approx H_{j} + \Cl_{j+1}(x(k))$. A detailed outline of this procedure is given in \cref{sub:outline_of_the_proof_of_cref_thm_max_general}.

We fill in the mesh using a bound on the maximum fluctuations of an area-tilted random walk bridge (\cref{maxbd}). \cref{maxbd} follows from a one-point upper tail bound of Tracy--Widom exponent 3/2 (\cref{oneptbd}). Such an estimate appeared in \cite{HV04} for a more general class of area-tilted random walks; we adapt their proof strategy to handle the high boundary values and short interval sizes we consider. Simplifications can be made thanks to stochastic monotonicity.

\subsubsection{Partition function lower bound and ballot theorem}
The key input for the proof of the dropping lemma (\cref{lem:dropping-lemma}) is a suitably precise lower bound on the partition function of a single area-tilted random walk on an interval $I\subseteq \intint{-N,N}$ with tilt parameter $\lambda>0$ and boundary conditions $u,v \geq 0$ (Proposition~\ref{p.partition function lower bound}). Here we allow $u$ and $v$ to be as large as $H_\lambda^3 := \lambda^{-1} N$ (the start of the large deviations regime). Estimates for such partition functions appeared in \cite{HV04,ISV14}, though with boundary conditions $O(H_\lambda)$.

The proof relies on controlling the random walk bridge on carefully chosen subdiffusive scales so as to force a parabolic descent. This in turn makes use of ballot theorems for random walk bridges, i.e., upper and lower bounds on the probability that a bridge on $\intint{-N,N}$ with boundary values $x$ and $y$ stays above $0$ for its entire lifetime. In the case of the simple random walk, it is well known that this probability behaves like $xy/N$. While such ballot theorems have been studied for a much more general class of walks a number of times previously (e.g., see \cite{addarioberryreed} and references therein), they have focused on the case of diffusive boundary conditions, i.e.,  $x,y = O(N^{1/2})$. 
To handle our high boundary conditions, we require a ballot theorem for general random walk bridges that also covers the case $\max(x,y) \gg N^{1/2}$ but $xy \ll N$. We establish such a ballot theorem in Theorem~\ref{t.ballot theorem} using monotonicity and again by analyzing the random walk bridge on a subdiffusive scale.

\subsubsection{Tightness, limiting Gibbs property, and convergence}

Once \cref{thm:simplifiedmain2} has been proven, the final ingredient to establish Theorem \ref{thm:simplifiedmain} is to show tightness of the rescaled line ensemble and that all subsequential limits have the Brownian Gibbs property. Indeed, if we do so, then a result of \cite{CG23} (included in Theorem~\ref{muexist} here) characterizing $\mu_{a,b}$ guarantees that all weak subsequential limits are $\mu_{a,b}$, as the above upper tail control will imply the conditions of the characterization. 

 For tightness, given the upper tail control outlined above and the fact that the curves are deterministically lower bounded by $0$, we only need to show a uniform modulus of continuity bound. Given the prelimiting Gibbs property, the idea (as has been used in many previous works on tightness of line ensembles) is to transfer the modulus of continuity estimates known for the base measure of independent random walk bridges by controlling the Radon--Nikodym derivative appearing in the Gibbs property. Essentially, this comes down to controlling a partition function of the form implicitly appearing in \eqref{areatiltRW} (though on an interval on scale $N^{2/3}$ rather than $N$).

There are by now many arguments establishing such statements for Brownian models (e.g. \cite{CIW19b,corwin2014brownian}) and also discrete ones \cite{dff,dimitrov2021tightness,ser1}. The existing arguments in discrete models generally have made use of a strong KMT coupling result for random walk bridges \cite{dimitrov2021kmt} (i.e., a quantitative coupling with Brownian bridge), which would require us to impose stronger assumptions on the increment distribution (and would exclude distributions of the form $\exp(-\beta|x|)$ as one might heuristically expect to arise in approximations to the SOS level lines). 

Here we give a shorter, more streamlined proof of tightness that relies only on monotonicity and an invariance principle for random walk bridges, requiring very minimal assumptions on the increment distribution. Our argument is a simplification of an argument in \cite{ach24}, which is possible due to the full monotonicity property we have at our disposal (in contrast to \cite{ach24}).

Finally, showing that the prelimiting Gibbs property becomes the Brownian Gibbs property in the limit is a standard application of an invariance principle for finitely many area-tilted random walks on the diffusive scale, which we adapt from \cite{ser23}.\\

We now outline the remainder of the paper. In \cref{s.RWest}, we define the area-tilted line ensembles we consider, specify the full range of parameters we address, and then state \cref{thm:main,thm:max}, which are the expanded versions of \cref{thm:simplifiedmain,thm:simplifiedmain2}, respectively. We then prove \cref{thm:main} given \cref{thm:max}, tightness, and the Gibbs property for subsequential limits.
In  \cref{s.ballot}, we develop basic results on random walks and bridges, stochastic monotonicity (\cref{l.monotonicity}), and the ballot theorem for arbitrarily high boundary conditions (\cref{t.ballot theorem}). \cref{sec:Zlbd} is devoted to the partition function lower bound (\cref{p.partition function lower bound}), from which the dropping lemma (\cref{lem:dropping-lemma}) follows. In \cref{sec:tail}, we establish the one-point (\cref{oneptbd}) and maximum (\cref{maxbd}) upper tail bounds on the single area-tilted random walk bridge. In \cref{sec:recursion,sec:recursion-proofs}, we prove \cref{thm:max} using the recursive scheme outlined above. In \cref{sec:tight}, we establish tightness and the Gibbs property for subsequential limits. Lastly, in \cref{s.monotonicity} we prove the stochastic monotonicity statement \cref{l.monotonicity}.

\subsection*{Notational conventions}

We lay out some notation that will be used throughout the paper. We implicitly work on a fixed probability space $(\Omega,\mathcal{F},\mathbb{P})$. We write $\mathbb{E}$ for expectation with respect to $\mathbb{P}$ and $\one_{\mathsf{E}}$ for the indicator of an event $\mathsf{E}\in\mathcal{F}$. The complement of $\mathsf{E}$ is denoted $\neg\mathsf{E}$. For a random variable $X$ and a probability measure $\mathbb{Q}$, we write $X\sim \mathbb{Q}$ to mean that $X$ has law $\mathbb{Q}$.
For a topological space $Y$, we let $C(Y)$ denote the space of continuous functions $Y\to\mathbb{R}$, which we endow with the topology of uniform convergence on compact sets and the induced Borel $\sigma$-algebra. By $\mathbb{N}$ we mean the set of positive integers. For two real numbers $\ell < r$, we write $\llbracket \ell, r \rrbracket := \{j \in \mathbb{Z} : \lfloor \ell \rfloor \leq j \leq \lceil r \rceil\}$. 
We denote the length of an interval $I$ by $|I|$. We  usually denote vectors by boldface and their components by italics, e.g., $\mathbf{x} = (x_1,\dots,x_n)$. Lastly, we use $C$ and $c$ to denote positive constants that may change from line to line and are universal in the following sense: neither $C$ nor $c$ depend on $N$, the curve index $j$, $a\geq a_0$, or $b\geq b_0$ (though they may depend on $a_0$, $b_0$, and other parameters).

\subsection*{Acknowledgments}
We thank Yvan Velenik for a useful explanation of aspects of his related work. C.S.\ thanks Amir Dembo for several helpful conversations about this project. Y.H.K.\ thanks Joseph Chen for a helpful discussion of the paper \cite{chen2025limiting}. 
M.H.\ was partially supported by the NSF through grants DMS-1937254 and DMS-2348156.
Y.H.K.\ acknowledges the support of a Junior Fellowship at Institut Mittag-Leffler in Djursholm, Sweden, during the Fall semester of 2024, funded by the Swedish Research Council under grant no.\ 2021-06594.
C.S.\ acknowledges fellowship support from the Northern California Chapter of the ARCS Foundation.
Part of this work was completed during the workshop ``Universality and Integrability in KPZ'' at Columbia University, March 2024, supported by NSF grants DMS-2400990 and DMS-1664650.

%!TEX root = ./rwareatilt.tex

\section{Setup, Gibbs property, and main results}\label{s.RWest}

Here we define the line ensembles we consider and state detailed versions of \cref{thm:simplifiedmain,thm:simplifiedmain2}.

\subsection{Random walk setup}\label{s:assumptions}

We work with random walk bridges $X$ on an interval $\llbracket \ell, r\rrbracket$ with boundary conditions $u, v \in \mathbb{R}$. It will be convenient to describe the random walk increments as an exponential of a \textit{random walk Hamiltonian}. Fix a function $H_{\mathrm{RW}} : \mathbb{R} \to \mathbb{R} \cup \{+\infty\}$. For a possible random walk bridge trajectory $X = (X(\ell), \dots, X(r)) \in \mathbb{R}^{r-\ell+1}$, define
\begin{equation}\label{HRW}
	p_{\ell,r}(X) = \exp\bigg(-\sum_{j=\ell}^{r-1} H_{\mathrm{RW}}(X(j+1) - X(j))\bigg) \,.
\end{equation}
We consider random walk bridges of two types:
\begin{enumerate}
	\item (Lattice) {Fix $h>0$ and $u,v\in q\mathbb{Z}$}. The random walk bridge is supported on {$q\mathbb{Z}$-valued trajectories $X \in (q\mathbb{Z})^{r-\ell+1}$} with $X(\ell)=u$, $X(r)=v$, and probability mass proportional to $p_{\ell,r}(X)$. 
	
	\item (Nonlattice) The random walk bridge $X$ has density proportional to $\delta_u(X(\ell)) \cdot p_{\ell,r}(X) \cdot \delta_v(X(r))$ with respect to Lebesgue measure on $\mathbb{R}^{r-\ell+1}$. 
\end{enumerate}
In the lattice case, $H_{\mathrm{RW}}$ need only be defined on {$q\mathbb{Z}$}. We make the following additional assumptions:

\begin{assumption}\label{a.convex}
	The function $H_{\mathrm{RW}} : \mathbb{R} \to \mathbb{R} \cup \{+\infty\}$ is convex. $($In the lattice case, this means that the linear interpolation of $H_{\mathrm{RW}}$ is convex.$)$
\end{assumption}

\begin{assumption}\label{a.misc}
	The random walk increment distribution has zero mean, finite variance $\sigma^2 > 0$, and finite moment generating function near zero. That is,
	
	\noindent
	\emph{(Lattice)}  Let $Z_{\mathrm{lat}} = \sum_{k\in{q\mathbb{Z}}} e^{-H_{\mathrm{RW}}(k)}$. Assume $Z_{\mathrm{lat}}\in (0,\infty)$,  and for $t$ in a neighborhood of $0$,
	\[
	\sum_{k\in{q\mathbb{Z}}} ke^{-H_{\mathrm{RW}}(k)} = 0 \,,\, \quad  Z_{\mathrm{lat}}^{-1} \ \sum_{k\in{q\mathbb{Z}}} k^2e^{-H_{\mathrm{RW}}(k)}  = \sigma^2 \in (0,\infty) \,,\, \quad \text{and} \quad \sum_{k\in{q\mathbb{Z}}} e^{tk}e^{-H_{\mathrm{RW}}(k)} < \infty \,.
	% \sum_{k\in\mathbb{Z}} ke^{-H_{\mathrm{RW}}(k)} = 0$,\\  $Z_{\mathrm{lat}}^{-1}\sum_{k\in\mathbb{Z}} k^2e^{-H_{\mathrm{RW}}(k)}  = \sigma^2 \in (0,\infty)$, and $\sum_{k\in\mathbb{Z}} e^{tk}e^{-H_{\mathrm{RW}}(k)} < \infty $ for $t$ in a neighborhood of $0$.
	\]

	\noindent
	\emph{(Nonlattice)} Let $Z_{\mathrm{nl}} = \int_{\mathbb{R}} e^{-H_{\mathrm{RW}}(x)}\,\mathrm{d}x$. Assume $Z_{\mathrm{nl}}\in(0,\infty)$, and for $t$ in a neighborhood of $0$,
	\[
		\int_{\mathbb{R}} xe^{-H_{\mathrm{RW}}(x)}\,\mathrm{d}x = 0 \,,\, \quad Z_{\mathrm{nl}}^{-1}\int_{\mathbb{R}} x^2e^{-H_{\mathrm{RW}}(x)}\,\mathrm{d}x = \sigma^2 \in (0,\infty)\,,\, \quad \text{and} \quad \int_{\mathbb{R}} e^{tx}e^{-H_{\mathrm{RW}}(x)}\,\mathrm{d}x < \infty\,.
	\]
\end{assumption}

\begin{remark}
The log-concavity condition Assumption \ref{a.convex} is imposed in Theorem \ref{thm:simplifiedmain} for the sole purpose that it implies a stochastic monotonicity result (Lemma \ref{l.monotonicity}); that is, we could have assumed the conclusion of \cref{l.monotonicity} and removed Assumption \ref{a.convex}. For instance, our results also apply to the simple random walk, which satisfies monotonicity as shown in the proof of Lemma \ref{l.monotonicity}. 
\end{remark}

\begin{remark}\label{rmk:bridgebc}
	Assumption \ref{a.convex} along with the zero mean condition in Assumption \ref{a.misc} force the support of the increment distribution to be a contiguous interval (in ${q\mathbb{Z}}$ or $\mathbb{R}$) containing 0. In particular the walks are aperiodic, so there will always exist $C>0$ depending only on $H_{\mathrm{RW}}$ so that the bridge from $(\ell,u)$ to $(r,v)$ is well-defined as long as $|u-v| \leq C(r-\ell)$. Likewise, non-crossing bridges with boundary conditions $\mathbf{u}, \mathbf{v}$ are well-defined as long as $|u_i-v_i|$ is not too large for each $i$. We will always implicitly assume boundary conditions for which the bridges are well-defined.
\end{remark}

\subsection{Line ensembles and known results}

Fix $N\in\mathbb{N}$ and an interval $I = \llbracket \ell, r\rrbracket \subseteq \llbracket-N, N\rrbracket$, where $\ell,r\in\mathbb{Z}$ and $\ell < r$. By a \textit{discrete line ensemble}, we mean a collection $\mathbf{X} = (X_1,\dots,X_n)$ of paths $X_i : I \to \mathbb{R}$. For $j\in I$ we write $\mathbf{X}(j) = (X_1(j),\dots,X_n(j))$. These can be identified with elements of $\mathbb{R}^{n|I|}$ (or {$(q\mathbb{Z})^{n|I|}$} in the lattice case).

For $n\in\mathbb{N}$ and $\mathbf{u},\mathbf{v}\in\mathbb{R}^n$, we let $\mathbb{P}_{n,I}^{\mathbf{u},\mathbf{v}}$ denote the law of $\mathbf{X} = (X_1,\dots,X_n)$, where $X_i$ are independent random walk bridges on $I$ with increments satisfying Assumptions \ref{a.convex} and \ref{a.misc}, $X_i(\ell) = u_i$, and $X_i(r) = v_i$. We now define the area-tilted line ensembles that are the focus of this paper. For $y\in\mathbb{R}$, define the closed Weyl chamber $W_y^n$ by 
$$W_y^n := \left\{\mathbf{x}\in\mathbb{R}^n : x_1 \geq x_2 \geq \cdots \geq x_n \geq y\right\}.$$ 

\begin{definition}\label{def:areatilt}
	Fix two \textit{area tilt parameters} $a>0$ and $b>1$ and a \textit{floor} $h : I\to\mathbb{R}$. Let $\mathbf{u} \in W^n_{h(\ell)}$ and $\mathbf{v} \in W^n_{h(r)}$, recalling $I = \llbracket \ell, r \rrbracket$. Define the \textit{area under} $X_i$ by
	\begin{equation}\label{A(X)}
		\mathcal{A}(X_i) := \sum_{j=\ell}^{r-1} [X_i(j) - h(j)],
	\end{equation}
	and define the non-crossing event $\mathsf{NC} := \{\mathbf{X}(j)\in W_{h(j)}^n \mbox{ for all } j\in I\}$. We then define the measure $\mathbb{P}_{n,I;h}^{a,b;\mathbf{u},\mathbf{v}}$ on discrete line ensembles by the Radon--Nikodym derivative
	\[
	\frac{\mathrm{d}\mathbb{P}_{n,I;h}^{a,b;\mathbf{u},\mathbf{v}}}{\mathrm{d}\mathbb{P}_{n,I}^{\mathbf{u},\mathbf{v}}} (\mathbf{X}) = \frac{1}{Z_{n,I;h}^{a,b;\mathbf{u},\mathbf{v}}} \exp\left(-\frac{a}{N}\sum_{i=1}^n b^{i-1}\mathcal{A}(X_i)\right) \one_{\mathsf{NC}}(\mathbf{X}),
	\]
	where the \textit{partition function} $Z_{n,I;h}^{a,b;\mathbf{u},\mathbf{v}}$ is the normalizing constant necessary to make $\mathbb{P}_{n,I;h}^{a,b;\mathbf{u},\mathbf{v}}$ a probability measure. We will use $\mathbb{E}_{n,I;h}^{a,b;\mathbf{u},\mathbf{v}}$ to denote expectation with respect to this measure.
	
\end{definition}

Note the implicit dependence of these measures on the parameter $N$ in the area tilt, which we omit from the notation for brevity. As above, we implicitly assume that $\mathbf{u},\mathbf{v},h$ are such that $Z_{n,I;h}^{a,b;\mathbf{u},\mathbf{v}} > 0$, for which it suffices to assume $\mathbb{P}_{n,I}^{\mathbf{u},\mathbf{v}}(\mathsf{NC}) > 0$. (Note that trivially $Z_{n,I;h}^{a,b;\mathbf{u},\mathbf{v}} \leq 1$.)

\begin{remark}[Notation for measures]\label{rmk:notation}
	
	We will simplify the above notation in various situations by removing superfluous indices. When the interval $I$ is symmetric, $I = \llbracket -M,M\rrbracket$, we often write $M$ instead of $I$ in the subscript, e.g., $\mathbb{P}_{n,M;h}^{a,b;\mathbf{u},\mathbf{v}}$. When $n=1$, we drop the subscript of $n$. The parameter $b$ plays no role when $n=1$, so we omit it from the superscript. We usually write $\lambda$ instead of $a$ for the area tilt coefficient in this case. When $\lambda=0$, i.e., a bridge with floor $h$ and no area tilt, we omit $\lambda$ from the superscript. Finally, we sometimes consider random walks with no conditioning at the endpoint; we let $\mathbb{P}_{I}^u$ denote the law of a random walk on $I$ started at $u$. (Strictly speaking only the left endpoint of $I$ is relevant, but as we often convert between walks and bridges we will write the whole interval for clarity.) We summarize the notation in the table below.

	\begin{figure*}[h!]
		\begin{center}
			\renewcommand{\arraystretch}{1.35}
			\begin{tabular}{|l|l|}
				\hline
				Notation & Meaning\\
				\hline
				$\mathbb{P}_{I}^u$ & random walk on interval $I$ started at $u$\\
				\hline
				$\mathbb{P}_{I}^{u,v}$ & random walk bridge (RWB) on $I$ from $u$ to $v$\\
				\hline
				$\mathbb{P}_{I;h}^{u,v}$ & RWB on $I$ from $u$ to $v$ with floor $h$\\
				\hline
				$\mathbb{P}_{I;h}^{\lambda;u,v}$ & RWB $X$ on $I$ from $u$ to $v$ with floor $h$ and area tilt $\exp(-\frac{\lambda}{N}\mathcal{A}(X))$\\
				\hline
				$\mathbb{P}_{n,I;h}^{a,b;\mathbf{u},\mathbf{v}}$ & $n$ non-crossing RWBs $X_i$ on $I$ from $u_i$ to $v_i$ with floor $h$ and area tilt\\[-4pt]
				& $\smash{\exp(-\frac{a}{N}\sum_{i=1}^n b^{i-1} \mathcal{A}(X_i))}$\\
				\hline
			\end{tabular}
		\end{center}
	\end{figure*}
	
	We will often abuse notation slightly by writing for instance $\mathbb{P}_{[\ell, r]}^{u,v}$ to mean $\mathbb{P}_{\llbracket \ell, r\rrbracket}^{u,v}$ in order to slightly declutter notation. In all of these cases, we replace $\mathbb{P}$ with $\mathbb{E}$ to denote expectation and with $Z$ to denote the partition function.
	
\end{remark}

\begin{remark}\label{rmk:floor}
	In previous works, the area in \eqref{A(X)} has usually been taken with respect to 0 instead of the floor $h$, i.e., $\sum_j X_i(j)$ instead of $\sum_j [X_i(j)-h(j)]$. These two choices are of course equivalent, because the exponent is a linear function and so the term $\sum_j h(j)$ cancels in the normalization. We choose to include this term as we will sometimes consider negative floors $h$ in our arguments, and it is convenient to keep the area term always positive so that the partition function remains bounded above by 1 and can thus be ignored when lower bounding various probabilities.
\end{remark}

\begin{remark}\label{rmk:shift}
	We record here a simple observation about invariance of the line ensemble under constant vertical shifts. If $\mathbf{X} \sim \P_{n,I;h}^{a,b;\u,\v}$ and $\zeta$ is a constant, then the law of $\mathbf{X} + \zeta := (X_1+\zeta,\dots,X_n+\zeta)$ is $\P_{n,I;h+\zeta}^{a,b;\u+\zeta,\v+\zeta}$, where $\u+\zeta := (u_1+\zeta, \dots, u_n+\zeta)$ and $\v+\zeta := (v_1+\zeta,\dots,v_n+\zeta)$. This property is immediate for non-crossing random walk bridges (i.e., $a=b=0$). The identity then follows by noting that $\mathcal{A}(X_i+\zeta) = \mathcal{A}(X_i) + \zeta|I|$, and the constant $\zeta|I|$ cancels in the normalization. 
\end{remark}

We will now discuss line ensembles on a continuum domain. By a \textit{line ensemble}, we will now mean a random variable $\mathbf{x}$ taking values in $C(\Sigma \times \Lambda)$, where $\Sigma = \llbracket 1,n\rrbracket$ or $\Sigma = \mathbb{N}$ and $\Lambda \subset\mathbb{R}$ is an interval or $\Lambda =\mathbb{R}$. Equivalently we can write $\mathbf{x} = (x_i)_{i\in\Sigma}$ where $x_i \in C(\Lambda)$. Note that any discrete line ensemble can naturally be viewed in this way by linearly interpolating between integer points so as to create continuous paths, and in the remainder of the paper we will make this identification.

As discussed in Section \ref{s.background}, 
 we aim to prove weak convergence as $N\to\infty$ of the rescaling $\mathbf{x}^N$ of $\mathbf{X}\sim \mathbb{P}_{n,N;0}^{a,b;\mathbf{u},\mathbf{v}}$ defined via
\begin{equation}\label{x^N}
	x_i^N(t) := \sigma^{-2/3}N^{-1/3} X_i(t\sigma^{-2/3}N^{2/3}), \qquad 1\leq i\leq n, \qquad t\in [-\sigma^{2/3}N^{1/3}, \sigma^{2/3}N^{1/3}].
\end{equation}
We define the following continuum analog of the measures in \cref{def:areatilt}, formalizing \eqref{eqn:BrownianLE-informal}.

\begin{definition}\label{def:brownian}
	For $n\in\mathbb{N}$, an interval $I = [\ell, r]\subset\mathbb{R}$, and $\mathbf{u},\mathbf{v}\in\mathbb{R}^n$, let $\mathbf{P}_{n,I}^{\mathbf{u},\mathbf{v}}$ denote the law of $\mathbf{x} = (x_1,\dots,x_n)$, where $x_i$ are independent standard Brownian bridges on $I$ with $x_i(\ell) = u_i$, $x_i(r) = v_i$. For $h:I\to\mathbb{R}$, $\mathbf{u}\in W_{h(\ell)}^n$, and $\mathbf{v}\in W_{h(r)}^n$, define the non-intersection event $\mathsf{NI} := \{\mathbf{x}(t) \in W_{h(t)}^n \mbox{ for all } t\in I\}$, and let $\mathbf{P}_{n,I;h}^{a,b;\mathbf{u},\mathbf{v}}$ be the measure on line ensembles $\mathbf{x}$ specified by
	\[
	\frac{\mathrm{d}\mathbf{P}_{n,I;h}^{a,b;\mathbf{u},\mathbf{v}}}{\mathrm{d}\mathbf{P}_{n,I}^{\mathbf{u},\mathbf{v}}}(\mathbf{x}) \propto \exp\bigg(-a\sum_{i=1}^n b^{i-1} \int_I [x_i(t)-h(t)]\,\mathrm{d}t\bigg)\, \one_{\mathsf{NI}}(\mathbf{x}) \,.
	\]
	We denote expectation with respect to this measure by $\mathbf{E}_{n,I;h}^{a,b;\mathbf{u},\mathbf{v}}$.
\end{definition}

As above, when $n=1$ we omit it from the subscript, and when $I=[-T,T]$ we write $T$ in place of $I$ in the subscript. See Remark \ref{rmk:floor} above in relation to the $h$ term in the integral.

The measures in Definition \ref{def:areatilt} possess an important Gibbs property, described as follows. 

\begin{definition}\label{def:gibbs}
	The $\textit{Gibbs property}$ of the area-tilted line ensembles is the following. Consider $\mathbf{X} \sim \mathbb{P}_{n,I;h}^{a,b;\mathbf{u},\mathbf{v}}$, and fix a positive integer $k\leq n$ and a subinterval $J = \llbracket \ell', r'\rrbracket \subseteq I$. Let $\mathcal{F}_{k,J}$ denote the $\sigma$-algebra generated by $\{X_i(j) : i > k \mbox{ or } j \in I \setminus \llbracket \ell' + 1, r'-1\rrbracket\}$. Then the conditional law of $\mathbf{X}$ given $\mathcal{F}_{k,J}$ is $\mathbb{P}_{k,J;g}^{a,b;\mathbf{w},\mathbf{z}}$, where $\mathbf{w} = (X_1(\ell'),\dots,X_k(\ell'))$, $\mathbf{z} = (X_1(r'),\dots,X_k(r'))$, and $g = X_{k+1} |_J$. This is an immediate consequence of Definition \ref{def:areatilt}, in particular, the linearity of the area-tilt functional.
	
	More generally, we have the following \textit{strong Gibbs property}. For two integer-valued random variables $\sigma,\tau$ with $\sigma \leq \tau$, we say the random interval $J = \llbracket \sigma,\tau\rrbracket$ is a \textit{stopping domain} for $(X_1,\dots,X_k)$ if for all $\ell,r\in\mathbb{Z}$, the event $\{\sigma\leq \ell, \,\tau\geq r\}$ lies in $\mathcal{F}_{k,\llbracket \ell, r\rrbracket}$. The strong Gibbs property says that conditional on the $\sigma$-algebra generated by $\sigma$, $\tau$, and $\mathcal{F}_{k,J}$, the law of $\mathbf{X}$ is $\smash{\mathbb{P}_{k,J;g}^{a,b;\mathbf{w},\mathbf{z}}}$ where $\mathbf{w}=(X_1(\sigma),\dots,X_k(\sigma))$, $\mathbf{z}=(X_1(\tau),\dots,X_k(\tau))$, and $g = X_{k+1}|_J$. 
\end{definition}

The proof of the fact that $\mathbf{X}\sim \mathbb{P}_{n,N;0}^{a,b;\mathbf{u},\mathbf{v}}$ satisfies the strong Gibbs property is a straightforward consequence of the usual Gibbs property; 
	 see \cite[Lemma 2.5]{corwin2014brownian}.

As discussed in Section \ref{s.related}, we should expect any limit point in $C(\mathbb{N}\times\mathbb{R})$ of $\mathbf{x}^N$ in \eqref{x^N} to possess the following continuum Gibbs property.

\begin{definition}\label{def:BGP}
	Let $\mathbf{x}$ be a line ensemble on $\mathbb{N}\times\mathbb{R}$. For intervals $I = [\ell,r] \subset \mathbb{R}$ and $\Sigma = \llbracket 1,k\rrbracket \subset \mathbb{N}$,
	 let $\mathcal{F}_{k,I}$ denote the $\sigma$-algebra generated by $\{x_i(j) : i \notin \Sigma \mbox{ or } j\notin (\ell,r)\}$. We say that $\mathbf{x}$ satisfies the \textit{Brownian Gibbs property with respect to} $(a,b)$-\textit{area tilts}, or more briefly that $\mathbf{x}$ is an $(a,b)$-\textit{tilted line ensemble} if for any bounded Borel-measurable functional $F : C(\Sigma \times I) \to \mathbb{R}$, we have $\mathbb{P}$-a.s.
	\begin{equation}\label{eqn:bgp}
		\mathbb{E}[F(\mathbf{x}|_{\Sigma\times I}) \mid \mathcal{F}_{k,I}] = \mathbf{E}_{k,I;h}^{a,b;\mathbf{u},\mathbf{v}} [F],
	\end{equation}
	where $\mathbf{u} = (x_1(\ell),\dots,x_k(\ell))$, $\mathbf{v} = (x_1(r),\dots,x_k(r))$, and $h = x_{k+1}|_I$.
\end{definition}

The existence of an $(a,b)$-tilted line ensemble on all of $\mathbb{N}\times\mathbb{R}$ was established in \cite{CIW19b}, and its uniqueness in law was shown in \cite{CG23} under natural conditions which we now define. 

\begin{definition}\label{def:utight}
	A line ensemble $\mathbf{x}$ on $\mathbb{N}\times\mathbb{R}$ is said to be \textit{uniformly tight} if for any $\varepsilon>0$ there exists $M>0$ such that
	\begin{equation}\label{utight}
		\sup_{t\in\mathbb{R}} \mathbb{P}(x_1(t) > M) < \varepsilon.
	\end{equation}
	We say $\mathbf{x}$ is \textit{asymptotically pinned to zero} if for all $\varepsilon, T>0$ there exists $j = j(T,\varepsilon)\in\mathbb{N}$ such that
	\begin{equation}\label{pinned}
		\mathbb{P}\bigg(\sup_{t\in[-T,T]} x_j(t) < \varepsilon\bigg) > 1-\varepsilon.
	\end{equation}
\end{definition}

\begin{theorem}\cite[Theorems 1.4 and 1.5]{CIW19b}$;$ \cite[Theorem 3.7]{CG23}. \label{muexist}
	Fix $a>0$, $b>1$. Then there exists a unique probability measure $\mu_{a,b}$ on $C(\mathbb{N}\times\mathbb{R})$ that is the law of a uniformly tight, asymptotically pinned-to-zero, $(a,b)$-tilted line ensemble.
\end{theorem}

\subsection{Main results}

We now give a more precise statement of our main results, Theorems \ref{thm:simplifiedmain} and \ref{thm:simplifiedmain2} from the Introduction. We establish the convergence of the rescaled line ensemble defined in \eqref{x^N} to the infinite-volume Gibbs measure $\mu_{a,b}$ of Theorem \ref{muexist}. 

The theorems will be stated for line ensembles on an interval $\mathcal{I}$ with boundary conditions $\mathbf{u},\mathbf{v} \in W_0^n$, satisfying various conditions. For parameters $\gamma,\kappa >0$, we will require 
\begin{equation}\label{def:boundary-conditions-thm:max}
	\mathcal{I} = [-LN^{2/3},LN^{2/3}], \quad \mathrm{where} \quad L \in [(\log N)^{1/3+\gamma}, N^{1/3}],
\end{equation}
and
\begin{equation}\label{def:BL}
	\max(u_1,v_1) \leq L^{2-\kappa} N^{1/3}.
	\end{equation}
Define $\kappa_0 := 6\gamma/(1+3\gamma)$ (so  that $(1/3+\gamma)(2-\kappa_0) = 2/3$).

\begin{theorem}\label{thm:main}
	Fix $a>0$, $b>1$, $\gamma>0$, and $\kappa \in(0, \kappa_0)$. Let $\mathbf{X}^N \sim\mathbb{P}_{n,\mathcal{I};0}^{a,b;\mathbf{u},\mathbf{v}}$ with $\mathcal{I}$ as in \eqref{def:boundary-conditions-thm:max}, $n=n(N)$ a sequence tending to $\infty$ as $N\to\infty$, and $\smash{\mathbf{u},\mathbf{v}\in W_0^n}$. Define the line ensemble $\smash{\mathbf{x}^N}$ on $\mathbb{N}\times\mathbb{R}$ by 
	\begin{align}\label{def:x_i}
	x_i^N(t) = \sigma^{-2/3}N^{-1/3} X_i^N (t\sigma^{-2/3}N^{2/3})\,, \qquad 1\leq i\leq n \,, \quad t\sigma^{-2/3}N^{2/3} \in \mathcal{I} \,,
	\end{align}
	and extend to all $t\in\mathbb{R}$ and $i \in\mathbb{N}$ as continuous functions in an arbitrary fashion. Let $\mu_N$ denote the law of $\mathbf{x}^N$ on $C(\mathbb{N}\times\mathbb{R})$. There exist $\delta> 0$ so that if $n\leq N^{\delta}$ and $\u,\v$ satisfy \eqref{def:BL}, then $\mu_N$ converges weakly as $N\to\infty$ to $\mu_{a,b}$, in the topology of uniform convergence on compact sets. 
	\end{theorem}

Most of the work required to prove Theorem \ref{thm:main} lies in verifying tightness. A key input is \cref{thm:max} below, which estimates the maximum of the $j^{\mathrm{th}}$ curve of the line ensemble on intervals of size $N^{2/3}$.  The proof occupies the largest portion of this paper: Sections \ref{sec:recursion} and \ref{sec:recursion-proofs}. We note that the restriction on $\kappa<\kappa_0$ is simply to make the upper bound on $K$ in \cref{thm:max}  diverging.

\begin{theorem}\label{thm:max}
	Fix $a_0>0$, $b_0>1$, $\gamma>0$, $\kappa \in (0, \kappa_0)$, $t\in\mathbb{R}$, and $T>0$. There exist positive constants $\delta = \delta(\kappa)$, $K_0 = K_0(a_0,b_0,\kappa)$, $N_0 = N_0(a_0,b_0,\gamma,\kappa,t,T)\in\mathbb{N}$, and $c=c(a_0,b_0)$ such that the following holds. 
		For  $N\geq N_0$, $n\leq N^{\delta}$, 
	$a \in [a_0, N^{\kappa/30}]$, $ b\in [b_0, \exp(\frac{\kappa}{6}(\log N)^{1/3})]$,
	$j \leq \min(n,\frac{\kappa}{20} \log_b N)$, $\mathcal{I}$ and $L$ satisfying \eqref{def:boundary-conditions-thm:max}, $\u$ and $\v$ satisfying \eqref{def:BL}, and $K\in[K_0, L^{2-\kappa}(\log N)^{-2/3}]$,
\begin{equation*} 
	\mathbb{P}_{n,\mathcal{I};0}^{a,b;\mathbf{u},\mathbf{v}} \bigg( \exists s\in[-T,T] : x_j^N(s+t) > K a^{-1/3}b^{-(j-1)/3}\bigg[\log \bigg(2 + \frac{a^{2/3}b^{2(j-1)/3}}{cK^{1/2}} |s|  \bigg) \bigg]^{2/3} \bigg)
		\leq e^{-cK^{3/2}}.
	\end{equation*}

\end{theorem}

{The $2/3$ exponent on the logarithm comes from a $3/2$ exponent in a corresponding one-point tail bound, which, as discussed in \cref{rk:tracy-widom}, is optimal.} The upper bounds on $a$ and $b$ in the hypotheses are not optimal and are only used in technical estimates in the proof of \cref{thm:max-general}; note in particular that $a$ and $b$ may diverge with $N$.

\begin{proof}[Proof of \cref{thm:simplifiedmain,thm:simplifiedmain2}]
Take $L = N^{1/3}$ and $\kappa < 3\ep$. Then $\mathcal{I} = [-N,N]$ in \eqref{def:boundary-conditions-thm:max}, and the condition $\max(u_1,v_1) \leq N^{1-\ep}$ implies \eqref{def:BL} (the parameter $\gamma$, and thus $\kappa_0$, can be taken arbitrarily large with this choice of $L$), so Theorem \ref{thm:simplifiedmain} follows immediately from Theorem \ref{thm:main}. Moreover, the condition $K \leq N^{2/3-\ep}$ in Theorem \ref{thm:simplifiedmain2} implies $K \leq L^{2-\kappa}(\log N)^{-2/3}$ for $N$ large, and as long as $\kappa\geq 2\ep$, the condition $j \leq \frac{\ep}{10} \log_b N$ implies $j\leq \frac{\kappa}{20}\log_b N$. Taking $N$ large enough so that $a\leq N^{\kappa/30}$ and $b\leq \exp(\frac{\kappa}{6}(\log N)^{1/3}))$, Theorem \ref{thm:max} applies to our situation, and taking $s=0$ yields the result.
\end{proof}

\cref{thm:max} will follow from the more technical \cref{thm:max-general}, showing that with high probability the top curve is dominated by a deterministic ``ceiling'' function on the entire interval which is roughly logarithmic in the bulk (a monotonicity argument immediately yields a similar ceiling on the $j\th$ curve of the line ensemble, see \cref{rk:thm6.1-j}).

The proof of our main result Theorem \ref{thm:main} will follow quickly from the above along with the following two propositions, which we prove in Section \ref{sec:tight}.

\begin{proposition}\label{prop:tight}
	The sequence $\{\mu_N\}_{N\geq 1}$ of laws on $C(\mathbb{N}\times\mathbb{R})$ is tight.
\end{proposition}

\begin{proposition}\label{prop:gibbs}
	Any subsequential limit of $\{\mu_N\}_{N\geq 1}$ is the law of a line ensemble satisfying the Brownian Gibbs property with respect to $(a,b)$-area tilts.
\end{proposition}

\begin{proof}[Proof of Theorem \ref{thm:main}]
	It suffices to verify that the sequence $\{\mu_N\}_{N\geq 1}$ is tight with respect to the topology of uniform convergence on compact sets on $C(\mathbb{N}\times\mathbb{R})$, and that any weak subsequential limit must be the measure $\mu_{a,b}$ of Theorem \ref{muexist}. Tightness is provided by Proposition \ref{prop:tight}.
	
	Next we verify that any subsequential limit satisfies the two conditions of uniform tightness and asymptotic pinning to zero in Definition \ref{def:utight}. First, note from \eqref{def:BL} that $L\to\infty$ as $N\to\infty$, so that $tN^{2/3}\in\mc I$ for any $t\in\R$ for all large enough $N$. Next, note that due to the conditions $L\geq (\log N)^{1/3+\gamma}$ in \eqref{def:boundary-conditions-thm:max} and $\kappa < \kappa_0(\gamma)$, the upper bound of $L^{2-\kappa}(\log N)^{-2/3}$ on $K$ in \cref{thm:max} diverges as $N\to\infty$. Thus for any fixed $M>0$, we can choose $N$ large enough so that \cref{thm:max} with $j=1$ (and $s=0$) yields

	\[
	\lim_{M\to\infty}\sup_{t\in\mathbb{R}} \limsup_{N\to\infty} \mathbb{P}_{n,\mathcal I;0}^{a,b;\mathbf{u},\mathbf{v}} \Big(x_1^N(t) > M\Big) = 0\,,
	\] 
	which implies \eqref{utight}. To verify \eqref{pinned} we first take $K\geq K_0$ {in} Theorem \ref{thm:max} large enough depending on $\ep$ so that $e^{-cK^{3/2}} < \ep$, and then take $j$ large enough depending on $T,\varepsilon$ so that $K a^{-1/3}b^{-(j-1)/3}[\log(2+\frac{Ta^{2/3}b^{2(j-1)/3}}{cK^{1/2}})]^{2/3} < \varepsilon$.
	
	Proposition \ref{prop:gibbs} now yields that any subsequential limit is the law of a uniformly tight, asymptotically pinned-to-zero, $(a,b)$-tilted line ensemble. It follows from Theorem \ref{muexist} that $\mu_{a,b}$ is the only possible subsequential limit of $\{\mu_N\}_{N\geq 1}$, and so in view of tightness we in fact have $\mu_N\to\mu_{a,b}$.
\end{proof}

%!TEX root = ./rwareatilt.tex

\section{Random walk bridge tools, monotonicity, and ballot theorems}\label{s.ballot}

In this section we state various basic results on weak convergence (Section~\ref{s.invariance}), stochastic monotonicity (Section~\ref{s.monotonicity statement}), and tail bounds for random walk bridges (Section~\ref{s.rw rail bounds}) that will be used throughout the paper.  Then, in Section~\ref{s.ballot theorems}, we state and prove a new ballot theorem estimating the probability that a random walk bridge remains nonnegative, which allows for boundary conditions far beyond the diffusive scale.

\subsection{Invariance principle}\label{s.invariance}

We first state a classical version of Donsker's invariance principle that applies to the random walk bridges that we consider in this paper. We will use this fact throughout the paper to obtain various estimates on random walk bridges. The result is essentially proven in \cite{liggett} in the lattice case and in \cite{borisov} in the nonlattice case.

\begin{lemma}\label{l.invar}
	Fix $\ell < r$, $u,v\in\mathbb{R}$, and $\alpha\in\mathbb{R}$. Assume $u,v\in\mathbb{R}$ and $\{u^N\}_{N\geq 1}$, $\{v^N\}_{N\geq 1}$ are sequences such that $\sigma^{-(1+\alpha/2)} N^{-1/2}u^N \to u$ and $\sigma^{-(1+\alpha/2)} N^{-1/2}v^N \to v$ as $N\to\infty$. For $N\geq 1$, let $X^N$ denote a random walk bridge from $(\lfloor \ell N\rfloor, u^N)$ to $(\lceil rN\rceil, v^N)$, i.e., $X^N \sim \mathbb{P}^{u^N,v^N}_{[\ell N, rN]}$, satisfying Assumption \ref{a.misc}. Define $x^N(t) = \sigma^{-(1+\alpha/2)} N^{-1/2} X^N(\sigma^\alpha Nt)$ for $t\in[\ell,r]\cap \sigma^{-\alpha}N^{-1}\mathbb{Z}$, and extend to $t\in[\ell,r]$ by linear interpolation. Then the process $\{x^N(t) : t\in[\ell,r]\}$ converges weakly as $N\to\infty$ to a standard Brownian bridge from $(\ell,u)$ to $(r,v)$ in the uniform topology on $C([\ell,r])$.
\end{lemma}

The exponent in the prefactor $\sigma^{-(1+\alpha/2)}$ ensures the correct covariance, as follows from a simple calculation. For the lattice case, weak convergence of the process follows from Theorem 1 of \cite{liggett}, while the limit is identified as the Brownian bridge by the conditional density specified in Theorem 2. For the nonlattice case, the result follows from Theorem 1 of \cite{borisov}. This result is only stated in the case $[\ell,r]=[0,1]$ and $u=v=0$, but straightforward modifications of the proof extend the result to our setup. We refer the reader to these two papers for the details of the argument.

\subsection{Stochastic monotonicity}\label{s.monotonicity statement}

We next state a crucial monotonicity result for area-tilted random walk bridges that we will use throughout. It will be relevant here to consider line ensembles with ceilings (in addition to the floors already discussed), so in this section only, we use the following modified notation. With notation as in Definition \ref{def:areatilt}, for a function $g : I\to\mathbb{R}$ satisfying $g(\ell) \geq u_1$, $g(r) \geq v_1$, and $g(j)\geq h(j)$ for all $j\in I$, we use $\mathbb{P}^{a,b;\u,\v}_{n,I;g,h}$ to denote the measure $\mathbb{P}^{a,b;\u,\v}_{n,I;h}$ further conditioned on the event $\{X_1(j) \leq g(j) \mbox{ for all } j\in I\}$. Thus $\mathbb{P}^{a,b;\u,\v}_{n,I;+\infty,h} = \mathbb{P}^{a,b;\u,\v}_{n,I;h}$. The proof of Lemma~\ref{l.monotonicity} is a standard Markov chain coupling argument and is deferred to Appendix~\ref{s.monotonicity}.

\begin{lemma}\label{l.monotonicity}
	Let $n, N\in\N$ and $I\subseteq \llbracket -N,N\rrbracket$ an interval. For $*\in\{\shortuparrow,\shortdownarrow\}$, let $a^*, b^*\in\R$, $\mathbf{u}^*,\mathbf{v}^* \in\R^n_{\geq}$ and $g^*, h^*:I\to\R$ be such that $u^{\shortuparrow}_i\geq u^{\shortdownarrow}_i$, $v^{\shortuparrow}_i\geq v^{\shortdownarrow}_i$ for $i\in\intint{1,n}$, $g^{\shortuparrow}(x) \geq g^{\shortdownarrow}(x)$, $h^{\shortuparrow}(x) \geq h^{\shortdownarrow}(x)$ for all $x\in I$, and $a^{\shortuparrow}\leq a^{\shortdownarrow}$, $b^{\shortuparrow} \leq b^{\shortdownarrow}$. %$($Note the reversed inequality with respect to the superscript$)$. 
{Suppose the underlying increment distribution satisfies Assumption~\ref{a.convex}.}
Then there exists a coupling $Q$ of the random variables $\mathbf{X}^* = ((X_i^*)_{i=1}^n) \sim \P^{a^*,b^*; \mathbf{u}^*, \mathbf{v}^*}_{n,I; g^*, h^*}$ for $*\in\{\shortuparrow,\shortdownarrow\}$ such that, if $\smash{(\mathbf{X}^{\shortuparrow}, \mathbf{X}^{\shortdownarrow})} \sim Q$, then $Q$-a.s.,	
	\begin{align*}
		X^{\shortuparrow}_i(x) \geq X^{\shortdownarrow}_i(x) \quad \mbox{for all} \quad i\in\intint{1,n}, \, x\in I.
	\end{align*}
\end{lemma}

When such a coupling exists, we say that $\P^{a^\downarrow,b^\downarrow; \mathbf{u}^\downarrow, \mathbf{v}^\downarrow}_{n,I; g^\downarrow, h^\downarrow}$ is stochastically dominated by $\P^{a^\uparrow,b^\uparrow; \mathbf{u}^\uparrow, \mathbf{v}^\uparrow}_{n,I; g^\uparrow, h^\uparrow}$.

\begin{remark}\label{rmk:remove-top}
We record here a simple but very useful consequence of Lemma \ref{l.monotonicity} which will underlie the argument in Section \ref{sec:recursion}. For $\mathbf{X} \sim \mathbb{P}^{a,b;\u,\v}_{n,I;h}$ and any $j\in\llbracket 1, n\rrbracket$, consider the conditional law of $X_j$ given all curves $X_i$ with $i < j$. By the Gibbs property (a straightforward extension of Definition \ref{def:gibbs}), this conditional law is the measure $\mathbb{P}^{\lambda_j, b; \u^{\geq j}, \v^{\geq j}}_{n-j+1,I; X_{j-1},0}$, where $\lambda_j := ab^{j-1}$, $\u^{\geq j} := (u_j,\dots,u_n)$, $\v^{\geq j} := (v_j,\dots,v_n)$ (and $X_0 := +\infty$). By Lemma \ref{l.monotonicity}, this (random) measure is stochastically dominated by the no-ceiling measure $\mathbb{P}^{\lambda_j, b; \u^{\geq j}, \v^{\geq j}}_{n-j+1,I; 0}$. Thus, in order to upper bound the probability of any increasing event for $X_j$, for instance the event in Theorem \ref{thm:max}, we can simply remove all curves above level $j$. Moreover, on the event that $X_{j+1}$ is dominated by a deterministic function $f$ on $I$, we can further condition on all curves $X_i$ with $i>j$, and then use monotonicity with respect to the floor to consider instead the simpler non-random measure $\mathbb{P}_{I;f}^{\lambda_j,u_j,v_j}$.
\end{remark}

\subsection{Subexponential tail bounds}\label{s.rw rail bounds}

In this section we establish subexponential tail bounds on the maximum of random walks and random walk bridges. The first lemma is a standard estimate for the maximum of a random walk. Recall that $\mathbb{P}_{[0,N]}^u$ denotes the law of a random walk $X$ on $\llbracket 0,N \rrbracket$ started at $u$.

\begin{lemma}[Doob's maximal inequality]\label{l.doob maximal}
	Under Assumption \ref{a.misc}, there exists $c>0$ independent of $N$ such that for all $x>0$,
	\begin{align}\label{e.doob}
		\P_{[0,N]}^0 \left(\max_{0\leq j \leq N} |X(j)| > xN^{1/2}\right) \leq 2\exp\left(-c(x^2\wedge xN^{1/2})\right).
	\end{align}
	The same inequality without the prefactor of $2$ holds for $\max_{1\leq j\leq N} X(j)$.
\end{lemma}

\begin{proof}
	By Doob's maximal inequality applied to the exponential martingale $e^{\theta X(N)}/\E_{[0,N]}^0[e^{\theta X(N)}]$ (see, e.g., \cite[Theorem 12.2.5]{lawler2010random}),
	\begin{align*}
		\P_{[0,N]}^0 \left(\max_{1\leq j\leq N} X(j) > xN^{1/2}\right) \leq \exp(-\theta xN^{1/2})\E_{[0,N]}^0 [\exp(\theta X(N))].
	\end{align*}
	Note $\E_{[0,N]}^0 [\exp(\theta X(N))] = \E[\exp(\theta Y)]^N$, where $Y$ has the law of the increments of $\P_{[0,N]}^0$. Under the finite moment generating function assumption, there exist $\theta_0,K>0$ such that for all $|\theta|<\theta_0$, $\E[\exp(\theta Y)] \leq \exp(K\theta^2)$ (see, e.g., \cite[Theorem 2.13]{wainwright}). Taking $\theta = \theta_0\wedge(2K)^{-1}xN^{-1/2}$ and adjusting the constants gives the inequality for $\max X$. For $\max |X|$, we do the same argument for $-X$ and apply a union bound, giving the prefactor of 2 on the right-hand side of \eqref{e.doob}. 
\end{proof}

The next lemma is an analogous estimate for random walk bridges. Recall that $\mathbb{P}_{[0,N]}^{u,v}$ denotes the law of a random walk bridge on $\llbracket 0,N\rrbracket $ from $u$ to $v$.

\begin{lemma}\label{l.bridge maximal}
	Under Assumptions \ref{a.convex} and \ref{a.misc}, there exist $c>0$ and $N_0\in\mathbb{N}$ such that for all $u>0$ and $N\geq N_0$,
	\begin{equation*}
		\mathbb{P}_{[0,N]}^{u,u}\left(\max_{1\leq j \leq N-1} X(j) > 2u\right) \leq 4\exp\bigg(-c\bigg(\frac{u^2}{N} \wedge u\bigg)\bigg).
	\end{equation*}
\end{lemma}

\begin{proof}
	Note that the event that $\max X > 2u$ is increasing with respect to $X$. 
	Then, by stochastic monotonicity with respect to the boundary conditions (Lemma \ref{l.monotonicity}),
	\begin{align*}
		\mathbb{P}_{[0,N]}^{u,u}\left(\max_{0\leq j \leq N} X(j) > 2u\right)  &\leq \mathbb{P}_{[0,N]}^u \left(\max_{0\leq j\leq N} X(j) > 2u \ \bigg|\  X(N) \geq u\right) \\
		&= \mathbb{P}_{[0,N]}^0\left(\max_{0\leq j\leq N} X(j) > u \ \bigg|\ X(N) \geq 0\right) \leq \frac{\mathbb{P}_{[0,N]}^0(\max_{0\leq j\leq N} X(j) > u)}{\mathbb{P}_{[0,N]}^0 (X(N)\geq 0)}.
	\end{align*}
	By Lemma \ref{l.doob maximal}, the numerator in the last expression is bounded by $\exp(-c(u^2 N^{-1} \wedge u))$. By the central limit theorem, $\mathbb{P}_{[0,N]}^0 (X(N)\geq 0) \to 1/2$ as $N\to\infty$, so in particular we can choose $N_0\in\mathbb{N}$ so that the denominator in the last expression is at least $1/4$ for all $N\geq N_0$. Combining these two estimates proves the desired bound.
\end{proof}

\subsection{Ballot theorems}\label{s.ballot theorems}

In this section we state a ballot theorem for the class of random walk bridges under consideration in this article, namely those satisfying Assumptions \ref{a.convex} and \ref{a.misc}. The lower and upper bounds will be proven separately in Sections~\ref{s.ballot theorem upper bound} and \ref{s.ballot theorem lower bound}, respectively.

\begin{theorem}[Ballot theorem]\label{t.ballot theorem}
There exist $c,C>0$ such that, for all $x,y>0$ and $N\in\N$,
\begin{align*}
		c\cdot\min\left(1,\frac{xy}{N}\right) \leq \P^{x,y}_{[0,N]}\bigl(X(k) \geq 0,  k=1, \ldots, N-1\bigr) \leq {C\cdot\frac{(x+1)(y+1)}{N}}.
	\end{align*}
\end{theorem}

The proofs of both inequalities of the ballot theorem will use the following simple change of measure lemma.

\begin{lemma}\label{l.change of measure}
	Let $N\in\N$, $x>0$, and $f:\R\to[0,\infty)$ be continuous with $\E^{x}_{[0,N]}[f(X(N))]<\infty$. Define $\tilde \P^{ x}_{[0,N]}$ and $ \tilde \P^{x,y}_{[0,N]} $ for $y\in\mathbb{R}$ by 
	\begin{equation}\label{e.tilted RN}
		\frac{\mathrm d\tilde \P^{ x}_{[0,N]}}{\mathrm d \P^{ x}_{[0,N]}}(X(1), \ldots, X(N)) \propto f(X(N)), \qquad  \tilde \P^{x,y}_{[0,N]} := \tilde \P^{ x}_{[0,N]}(\cdot\mid X(N)=y).
	\end{equation}
	Then for any $y\in\R$ lying in the support of the law of $X(N)$ under $\P^{x}_{[0,N]}$, it holds that $\P^{x,y}_{[0,N]} = \tilde \P^{x,y}_{[0,N]}$.
\end{lemma}

\begin{proof}
	We may write $\tilde \P^{x,y}_{[0,N]}(\cdot) = \lim_{\varepsilon\downarrow 0}\tilde \P^{x}_{[0,N]}(\cdot\mid X(N)\in [y-\varepsilon,y+\varepsilon])$ as a weak limit. Then the conclusion follows from the definition of $\tilde \P^{x}_{[0,N]}$ and the assumed continuity of $f$.
\end{proof}

\begin{corollary}\label{c.tilted random walk bridge}
	Let $y\in\R$ lie in the support of the law of $X(N)$ under $\P^{x}_{[0,N]}$. Then $\tilde \P^{x, y}_{[0,N]} = \P^{x, y}_{[0,N]}$, where $\tilde \P^{x,y}_{[0,N]}$ is defined as in \eqref{e.tilted RN} with $f(x) = \exp(\theta^*_yx)$, where $\theta^*_y = \theta^*_y(N)$ is the unique real number such that $\E^{0}_{[0,N]}[X(1)e^{\theta^*_yX(1)}]/\E^0_{[0,N]}[e^{\theta^*_yX(1)}] = y/N$. 
\end{corollary}

\begin{proof}
	Given Lemma~\ref{l.change of measure}, this follows immediately from the existence and uniqueness of $\theta^*_y$; the latter is standard and can be found, for instance, in \cite[Section 2.7]{durrett2019probability}.
\end{proof}

In other words, Corollary~\ref{c.tilted random walk bridge} says that the random walk bridge of length $N$ associated to the random walk tilted to have drift $y/N$ (and thus mean $y$ after $N$ steps) is the same as the original random walk bridge. This is useful because, for instance, under the tilted random walk measure, the probability of lying above $y$ after $N$ steps is of constant order even if $y\gg N^{1/2}$. 

\begin{remark}
	Observe that the tilted random walk measure has increment distribution which still satisfies Assumption \ref{a.convex} if the original increment distribution did. This is simply because the density at $z\in\R$ of the tilted increment is the same as that of the original increment at $z$ with an additional factor of $\exp(\theta^*_yz)$, and a convex function plus a linear function remains convex.
	In particular, the tilted random walk bridge measure also satisfies monotonicity by Lemma~\ref{l.monotonicity}.
\end{remark}

The argument for Theorem~\ref{t.ballot theorem} will reduce the desired estimates to ballot theorems for random walks, i.e., without the endpoint fixed. We record below some of the estimates needed for the latter from \cite{pemantle1995critical}:

\begin{lemma}[Ballot theorem for random walk, {\cite[Lemma 3.3]{pemantle1995critical}}]\label{l.ballot theorem for random walk}
	Let the increment distribution of the random walk under $\P^{x}_{[0,N]}$ have mean zero and finite variance. Fix $K>0$. Then there exist constants $C,c>0$ such that, for all $N\in\N$,
	\begin{enumerate}
		\item $\P^{x}_{[0,N]}(X(k)\geq 0, k=1, \ldots, N) \leq CxN^{-1/2}$ for all $x\ge 1$;
		
		\item $\P^{x}_{[0,N]}(X(k)\geq 0, k=1, \ldots, N) \geq cxN^{-1/2}$ for all $x\in[0,N^{1/2}]$; and
		
		\item $\E^{x}_{[0,N]}[X(N)^2\mid X(k)\geq 0, k=1, \ldots, N] \leq CN$ for all $x\in[0,N^{1/2}]$.
	\end{enumerate}
	
\end{lemma}

The third estimate is proven in \cite{pemantle1995critical} only for $x=0$, but the proof easily extends to $x>0$, as also noted in \cite[Lemma 3]{addarioberryreed}.

\subsection{Upper bound}\label{s.ballot theorem upper bound}

The basic idea to establish the upper bound of the ballot theorem is to look at a scale $z=N^2/y^2$ (where we assume without loss of generality that $y\gg x$), which is the one where we expect the value of the random walk to be $x+O(z^{1/2})$; one can check that $z$ is chosen such that the movement of the random walk due to the drift induced by the endpoint values (i.e., $(y-x)/N$) is of the same order as the diffusive fluctuations. 

For the upper bound, one can ignore the probability contribution coming from the event that the random walk bridge stays above zero after $z$ (which we expect to be of constant order), and consider only the event that it stays above $0$ on $[0,z]$. On the latter interval the endpoint values are $x$ and of order $z^{1/2}$. Since both values are at most of order the diffusive scale on the interval (which is of size $z$), one can use the ballot theorem estimates for random walks from Lemma~\ref{l.ballot theorem for random walk} to obtain a bound of order $xz^{-1/2} = xz^{1/2}/z = xy/N$.

\begin{proof}[Proof of Theorem~\ref{t.ballot theorem}, upper bound]
	First note by stochastic monotonicity that if in the boundary values we replace $x,y$ with $(x+1),(y+1)$ then the probability in Theorem \ref{t.ballot theorem} increases. Therefore it suffices to prove an upper bound of $Cxy/N$ for $x,y\ge 1$. We may assume that $xy\leq N$ as otherwise the statement is trivial. We assume without loss of generality that $x\leq y$, so that $x\leq N^{1/2}$. Let $z = \min(\frac{1}{2}N, N^2/y^2)$. Note that then $x\leq 2z^{1/2}$. {We observe that, by Bayes' theorem, for any $M>0$,}
	\begin{align}\label{e.ballot theorem upper bound first step}
		\P^{x,y}_{[0,N]}(X(k) \geq 0, k=1, \ldots, N-1) \leq \frac{\P^{x,y}_{[0,N]}(X(k) \geq 0, k=1, \ldots, N-1  \mid X(z) \le Mz^{1/2})}{\P^{x,y}_{[0,N]}(X(z) \le Mz^{1/2} \mid X(k) \geq 0, k=1, \ldots, N-1)}.
	\end{align}
	We need to lower bound the denominator by a constant order quantity. Equivalently, we want to show that {$\P^{x,y}_{[0,N]}(X(z) > Mz^{1/2} \mid X(k) \geq 0, k=1, \ldots, N-1)$} is bounded away from $1$ uniformly in $N$. Recall from Corollary~\ref{c.tilted random walk bridge} that $\P^{x,y}_{[0,N]} = \tilde\P^{x,y}_{[0,N]}$, where the latter corresponds to the underlying random walk measure being tilted by $\exp(\theta^*_yX(N))$, where $\theta^*_y = \theta^*_y(N)$ is chosen such that $X(N)$ has mean $y$ (when started at zero). Let $Y(k) = X(k)-ky/N$ for all $k$, {which under $\tilde \P^{x}_{[0,N]}$ is a mean zero random walk started from $x$}. Note also that $z^{1/2} \geq zy/N$, so $Mz^{1/2} -zy/N \geq \frac{1}{2}Mz^{1/2}$ if $M\geq 2$. We then bound as follows:
	\begin{align*}
		\MoveEqLeft[6]
		\P^{x,y}_{[0,N]}(X(z) > Mz^{1/2} \mid X(k) \geq 0, k=1, \ldots, N-1)\\
		&= \tilde\P^{x}_{[0,N]}(X(z) > Mz^{1/2} \mid X(k) \geq 0, k=1, \ldots, N-1, X(N)=y)\\
		&\leq \tilde\P^{x}_{[0,N]}\left(Y(z) > \tfrac{1}{2}Mz^{1/2} \mid Y(k) > -ky/N, k=1, \ldots, N-1, Y(N) = 0\right)\\
		&\leq \tilde\P^{x}_{[0,N]}\left(Y(z) > \tfrac{1}{2}Mz^{1/2} \mid Y(k) \geq 0, k=1, \ldots, N-1\right),
	\end{align*}
	the final inequality using monotonicity (Lemma~\ref{l.monotonicity}).  Then, using the Gibbs property and monotonicity again, the previous line can written as
	\begin{align}
		\MoveEqLeft[1]
		\sum_{t>\frac{1}{2}Mz^{1/2}}\tilde\P^{x}_{[0,N]}\left(Y(z) \in[t, t+1) \mid Y(k)\geq 0, k=1, \ldots, N-1\right)\nonumber\\
		&=\sum_{t>\frac{1}{2}Mz^{1/2}}\frac{\tilde\P^{x}_{[0,N]}\left(Y(z) \in[t, t+1), Y(k) \geq 0, k=1, \ldots, N-1\right)}{\tilde\P^{x}_{[0,N]}\left(Y(k)\geq 0, k=1, \ldots, N-1\right)}\label{e.rw meander tail bound decomposition}\\
		&\leq\sum_{t>\frac{1}{2}Mz^{1/2}}\frac{\tilde\P^{x}_{[0,z]}\left(Y(z) \in[t, t+1), Y(k) \geq 0, k=1, \ldots, z\right)\cdot \tilde\P^{t+1}_{[z, N]}(Y(k) \geq 0, k= z+1\ldots, N-1)}{\tilde\P^{x}_{[0,N]}\left(Y(k)\geq 0, k=1, \ldots, N-1\right)}.\nonumber
	\end{align}
	Consider the summand with index $t$ in the last line. By Lemma~\ref{l.ballot theorem for random walk} (1), the second factor in the numerator is upper bounded by $C(t+1)(N-z)^{-1/2}\leq CtN^{-1/2}$ since $z\leq N/2$. Since $x\leq N^{1/2}$, by Lemma~\ref{l.ballot theorem for random walk} (2), the denominator is lower bounded by $cxN^{-1/2}$. The first factor in the numerator can be written as
	\begin{align*}
		\tilde\P^{x}_{[0,z]}\left(Y(z) \in[t, t+1)\mid Y(k) \geq 0, k=1, \ldots, z\right)\cdot \tilde\P^{x}_{[0,z]}\left(Y(k) \geq 0, k=1, \ldots, z\right).
	\end{align*}
	The second factor here is upper bounded by $Cxz^{-1/2}$ again by Lemma~\ref{l.ballot theorem for random walk} (1). Thus we can bound the last line of \eqref{e.rw meander tail bound decomposition} above by
	\begin{align*}
		\MoveEqLeft[6]
		C\sum_{t>\frac{1}{2}Mz^{1/2}}\frac{t}{z^{1/2}}\tilde\P^{x}_{[0,z]}\left(Y(z) \in[t, t+1)\mid Y(k) \geq 0, k=1, \ldots, z\right)\\
		&\leq C\cdot\tilde\E^{x}_{[0,z]}\left[\frac{Y(z)}{z^{1/2}}\one_{Y(z)>\frac{1}{2}Mz^{1/2}}\mid Y(k) \geq 0, k=1, \ldots, z\right].
	\end{align*}
	Using that $\E^{x}_{[0,z]}[Y(z)^2\mid Y(k)\geq 0, k=1, \ldots, z] \leq Cz$ by Lemma~\ref{l.ballot theorem for random walk} (3), it follows that there exists $M$ large enough, independent of $x,y,z$ and $N$, such that the previous display is at most $1/2$. Thus the denominator of \eqref{e.ballot theorem upper bound first step} is lower bounded by $1/2$, so the right-hand side of \eqref{e.ballot theorem upper bound first step} is at most
	\begin{align*}
		2\cdot \P^{x,y}_{[0,N]}(X(k) \geq 0, k=1, \ldots, N-1  \mid X(z) < Mz^{1/2}).
	\end{align*}
	By the Gibbs property and monotonicity, this is upper bounded by
	\begin{align}
		\MoveEqLeft[24]
		2\cdot \E^{x,y}_{[0,N]}\left[\P^{x, X(z)}_{[0,z]}(X(k) \geq 0, k=1, \ldots, z)\cdot\P^{X(z), y}_{[z,N]}(X(k) \geq 0, k=z, \ldots, N-1) \midd X(z) < Mz^{1/2}\right]\label{twoprob}\\
		&\leq 2\cdot\P^{x, Mz^{1/2}}_{[0,z]}(X(k) \geq 0, k=1, \ldots, z).\nonumber
	\end{align}
	Again by monotonicity, this is upper bounded by
	\begin{align*}
		2\cdot \P^{x}_{[0,z]}(X(k) \geq 0, k=1, \ldots, z  \mid X(z) > Mz^{1/2}) \leq 2\cdot\frac{\P^{x}_{[0,z]}(X(k) \geq 0, k=1, \ldots, z)}{\P^{x}_{[0,z]}(X(z) > Mz^{1/2})}.
	\end{align*}
	The denominator is lower bounded by $c=c(M)$ uniformly in $z$ by the central limit theorem and the Portmanteau theorem. {In the case $z=N^2/y^2$, the numerator is upper bounded by $Cxz^{-1/2} = Cxy/N$ by Lemma~\ref{l.ballot theorem for random walk} (1). In the other case $z=N/2$, in \eqref{twoprob} we instead bound both probabilities inside the expectation in the same manner, which implies the same upper bound. This completes the proof.}
\end{proof}

\subsection{Lower bound}\label{s.ballot theorem lower bound}

Here we prove the lower bound half of Theorem~\ref{t.ballot theorem}. The basic idea underlying the argument is again to consider the scale $z = N^2/y^2$. However, unlike in the argument for the upper bound, here we must  lower bound the probability of the random walk bridge staying positive on $[0,z]$ as well as $[z,N]$, instead of just the former. We expect the random walk bridge to be at height of order $z^{1/2}$ at $z$, which means the endpoint values of $Cz^{1/2} = CN/y$ and $y$ on $[z,N]$ satisfy the property that their product is equal to the interval length up to a constant. The ballot theorem lower bound in this case is handled in Lemma~\ref{l.random walk bridge persistence extreme case}, while the case that the endpoint values are at most diffusive in the interval length (which is needed to control the probability on $[0,z]$) is addressed in Lemma~\ref{l.ballot theorem lower bound diffusive case}. Their proofs will be given after that of the lower bound of Theorem~\ref{t.ballot theorem}.

\begin{lemma}\label{l.random walk bridge persistence extreme case}
	Fix $\delta > 0$. There exists $\varepsilon>0$ such that for all $N\in\N$ and $x,y >0$ such that $xy = \delta N$,
	\begin{align}\label{e.persistence prob extreme case lower bound}
		\P^{ x,y}_{[0,N]}\left(X(k) \geq 0, \, k=1, \ldots, N-1\right) \geq \varepsilon.
	\end{align}
	
\end{lemma}

\begin{lemma}\label{l.ballot theorem lower bound diffusive case}
	Let $M>0$. There exists $c>0$ such that for all $N\in\N$ and $x,y\in [0, MN^{1/2}]$,
	\begin{align*}
		\P^{x,y}_{[0,N]}(X(k) \geq 0, k=1, \ldots, N-1) \geq c\cdot\min\left(1,\frac{xy}{N}\right).
	\end{align*}
	
\end{lemma}

\begin{proof}[Proof of Theorem~\ref{t.ballot theorem}, lower bound]
	It suffices to assume that $N\geq N_0$ for any fixed $N_0$. Indeed, for $1\leq N\leq N_0$, the probability on the left-hand side is lower bounded by the case of $x=y=0$ (again by monotonicity) and the minimum over $1\leq N\leq N_0$ of the same probability. {Now, it is easy to see that this probability is positive: the probability or density (depending on whether we are in the discrete or continuous case) of a random walk starting at zero ending at zero is positive by Remark~\ref{rmk:bridgebc}, and therefore, by considering a permutation of the increments (all are equally likely since the increments are i.i.d.) which also ensures that the random walk trajectory remains above zero, the probability/density of the random walk going from zero to zero while remaining non-negative is also positive.}
	Thus the estimate is proven in this range of $N$ by adjusting $c$ appropriately. We will take $N_0$ to be a large constant whose value will be specified later.
	
	We may further assume that $x,y\geq 1$. If not, there exist constants $N'$ and $\delta>0$ (both depending on the law of the underlying random walk's increments) such that, with probability at least $\delta>0$, $X(1), \ldots, X(N')$ and $X(N-N'), \ldots, X(N)$ all remain positive and $X(N'),X(N-N')\geq 1$ as long as say $N_0>3N'$ (which we may assume as observed above). By the Gibbs property we can then apply the below argument on the interval $[N',N-N']$ which now has boundary conditions lower bounded by $1$.
	
	By monotonicity (Lemma~\ref{l.monotonicity}), we may assume $y\leq N$.
	Next, let $z = N^2/y^2$; note that since $y\leq N$, $z\geq 1$. It holds by the Gibbs property that, for any $\delta>0$,
	\begin{align*}
		\MoveEqLeft
		\P^{x,y}_{[0,N]}\left(X(k) \geq 0, k=1, \ldots, N-1\right)\\%
		&\geq \E^{x,y}_{[0,N]}\left[\P^{x,X(z)}_{[0,z]}\left(X(k) \geq 0, k=1, \ldots, z\right)\cdot \P^{X(z),y}_{[z,N]}\left(X(k) \geq 0, k=z+1, \ldots, N-1\right)\cdot \one_{X(z) > \delta z^{1/2}}\right].
	\end{align*}
	By monotonicity (Lemma~\ref{l.monotonicity}), this is lower bounded by
	\begin{align}\label{e.ballot theorem lower bound general case breakup}
		\P^{x,y}_{[0,N]}\left(X(z) > \delta z^{1/2}\right)\cdot\P^{x,\delta z^{1/2}}_{[0,z]}\left(X(k) \geq 0, k=1, \ldots, z\right)\cdot \P^{\delta z^{1/2},y}_{[z,N]}\left(X(k) \geq 0, k=z, \ldots, N\right).
	\end{align}
	Now, again by monotonicity the first term is lower bounded by the same with $x=y=0$, i.e., $\P^{0,0}_{[0,N]}(X(z) > \delta z^{1/2} )$. We wish to show that this is bounded away from zero uniformly. We upper bound the complementary probability as follows, using monotonicity in the first inequality:
	\begin{align*}
		\P^{0,0}_{[0,N]}\left(X(z) \leq \delta z^{1/2}\right)
		&\leq \P^{0}_{[0,N]}\left(X(z) \leq \delta z^{1/2}\mid X(N)\leq 0\right)
		\leq \frac{\P^{0}_{[0,N]}\left(X(z) \leq \delta z^{1/2}\right)}{\P^{0}_{[0,N]}\left(X(N)\leq 0\right)}.
	\end{align*}
	By the central limit theorem, the denominator is lower bounded by $1/4$ for all large enough $N$. Also by the central limit theorem, there exist $\delta>0$ and $z_0$ such that the numerator is at most $1/8$ (recalling that $z\geq z_0$). With this choice of $\delta$, $z_0$, and $N_0$, it holds that the first factor in \eqref{e.ballot theorem lower bound general case breakup} is lower bounded by $1/2$ as long as $N\geq N_0$.
	
	In the second factor of \eqref{e.ballot theorem lower bound general case breakup}, we may replace $x$ by $\min(x,z^{1/2})$ by monotonicity. Then by Lemma~\ref{l.ballot theorem lower bound diffusive case}, the second factor of \eqref{e.ballot theorem lower bound general case breakup} is lower bounded by $c\min(x,z^{1/2})z^{-1/2}$. By monotonicity and since $z^{1/2}= N/y$, the third factor of \eqref{e.ballot theorem lower bound general case breakup} is lower bounded by the same with $\delta z^{1/2}$ replaced by $\delta N/y$. Thus, by Lemma~\ref{l.random walk bridge persistence extreme case} that factor is lower bounded by some constant $c>0$ since {$\delta N/y\cdot y  = \delta N \geq \delta(N-z)$}. Overall, \eqref{e.ballot theorem lower bound general case breakup} is lower bounded by $c\min(1, xz^{-1/2}) = c\min(1, xy/N)$, using the definition of $z$. This completes the proof.
\end{proof}

Next we give the proof of Lemma~\ref{l.random walk bridge persistence extreme case} before turning to that of Lemma~\ref{l.ballot theorem lower bound diffusive case}.

\begin{proof}[Proof of Lemma~\ref{l.random walk bridge persistence extreme case}]
	We assume without loss of generality that $x\leq y$. By monotonicity (Lemma~\ref{l.monotonicity}), it also suffices to assume that $\delta\in(0,2]$ (where recall $xy/N=\delta$).
	
	Recall from Corollary~\ref{c.tilted random walk bridge} that $\P^{x,y}_{[0,N]} = \tilde\P^{x,y}_{[0,N]}$, where the latter corresponds to the underlying random walk measure being tilted by $\exp(\theta^*_yX(N))$, where $\theta^*_y = \theta^*_y(N)$ is chosen such that $X(N)$ has mean $y$.
	In particular, under $\tilde\P^{x}_{[0,N]}$, $(X(1), \ldots, X(N))$ is a random walk with drift $y/N$, so $(X(i)-iy/N)_{i=1}^N$ is a \iid random walk with zero drift. Further, $X(i)-X(i-1)$ still has a finite exponential moment under $\tilde\P^{x}_{[0,N]}$.
	
	Now, first by definition and then invoking Corollary~\ref{c.tilted random walk bridge}, the probability from \eqref{e.persistence prob extreme case lower bound} that we must lower bound equals
	\begin{align*}
		\P^{x}_{[0,N]}\left(X(k) \geq 0,  \, k=1, \ldots, N-1 \mid X(N)=y\right)
		&= \tilde\P^{x}_{[0,N]}\left(X(k) \geq 0,  \, k=1, \ldots, N-1 \mid X(N)=y\right).
	\end{align*}
	By rearranging some terms inside the probability and invoking monotonicity (Lemma~\ref{l.monotonicity}), this is lower bounded by
	\begin{align}
		\MoveEqLeft[6]
		\tilde\P^{x}_{[0,N]}\left(X(k) -ky/N >-ky/N,  \, k=1, \ldots, N-1 \mid X(N) -y\leq 0\right)\nonumber\\
		&\geq \tilde\P^{x}_{[0,N]}\left(X(k) -ky/N >-ky/N,  \, k=1, \ldots, N-1, X(N) -y\leq 0\right).\label{e.ballot theorem extreme case first lower bound}
	\end{align}
	Let us adopt the notation $Y(k) = X(k)-ky/N$, so that, under $\tilde\P^{x}_{[0,N]},$ $k\mapsto Y(k)$ is a drift zero random walk started at $x$, and let $Y_{k,\ell} = Y(\ell) -Y(k)$ denote its increment and $\overline Y_{k,\ell} := \max_{k\leq j\leq \ell} |Y_{k, j}|$ denote its maximum increment between $k$ and $\ell$. Let $J = \log_2(N^{1/2}x^{-1} + 1)-1$. Then the previous display equals $\tilde\P^{x}_{[0,N]}(Y(k) >-ky/N,  \, k=1, \ldots, N-1, Y(N)\leq 0)$, which we claim is lower bounded by
	\begin{align}\label{e.ballot theorem lower bound extreme case event}
		\tilde\P^{x}_{[0,N]}\left(\parbox{3.8in}{\centering $\overline Y_{(2^j-1)x^2, (2^{j+1}-1)x^2} < \rho\delta 2^{3j/4} x, \, j=0, \ldots, J$  and\\  $Y_{xN^{1/2}, {N}} < -4N^{1/2} \text{ and } \max_{xN^{1/2}\leq j\leq N} Y_{xN^{1/2}, j} \leq 8N^{1/2}$}\right),
	\end{align}
	where $\rho = \frac{1}{2}(2^{3/4}-1)$. 
	To see that this lower bound holds, observe that on the above event, for any $j=1, \ldots, J$ and $k\in[(2^j-1)x^2, (2^{j+1}-1)x^2]$,
	\begin{align*}
		Y(k) \geq x-\sum_{i=0}^{j} \overline Y_{(2^i-1)x^2, (2^{i+1}-1)x^2} \geq x- \rho\delta\sum_{i=0}^{j} 2^{3i/4} x \geq (-2^{3j/4}\delta+\tfrac{1}{2}\delta+1)x,
	\end{align*}
	while $ky/N \geq (2^j-1)x^2y/N = (2^j-1)\delta\cdot x$, i.e., $-ky/N\leq (-2^j\delta+\delta)x$. Similarly $Y_{xN^{1/2}} \leq x+\tfrac{1}{2}\delta(2^{J+1}-1)x = x+\frac{1}{2}\delta N^{1/2} \leq (\delta^{1/2}+\frac{1}{2}\delta)N^{1/2}$ (the last inequality using that $xy=\delta N$ and $x\leq y$ implies $x\leq \delta^{1/2}N^{1/2}$). Since $\delta\leq 2$ and $2^{1/2}+1\leq 4$ it follows that the event in \eqref{e.ballot theorem lower bound extreme case event} implies that in \eqref{e.ballot theorem extreme case first lower bound}.
	
	Now, $\overline Y_{(2^j-1)x^2, (2^{j+1}-1)x^2}$ are independent across $j$ under $\tilde\P^{x}_{[0,N]}$. By Lemma~\ref{l.doob maximal} (Doob's maximal inequality), for all $j$,
	\begin{align}\label{e.doob for Y_k}
		\P\left(\overline Y_{(2^j-1)x^2, (2^{j+1}-1)x^2} > \rho\delta 2^{3j/4} x\right) \leq 2\exp(-c\delta^2 2^{j/2}).
	\end{align}
	There is some absolute constant $j_0$ below which the right-hand side in the previous display is larger than $1$. However, it follows from Donsker's invariance principle that $\inf_{N\geq 1}\P(\max_{k\leq N}|Y(k)| < \rho N^{1/2}) > 0$: this uses the fact that since $Y(k)-Y(k-1)$ are mean zero and have contiguous support by Remark \ref{rmk:bridgebc}, they each have positive probability of lying in a small neighborhood around zero. Combining this fact with \eqref{e.doob for Y_k} yields that \eqref{e.ballot theorem lower bound extreme case event} and thus \eqref{e.ballot theorem extreme case first lower bound} is lower bounded by
	\begin{align*}
		c\cdot \prod_{j=0}^{J}\left(1-\min(c,\exp(-c\delta^2 2^{j/2})\right),
	\end{align*}
	where $c>0$. This expression is lower bounded by an absolute constant, {thus lower bounding the probability of the event in the first line of \eqref{e.ballot theorem lower bound extreme case event}. The two events in the second line of \eqref{e.ballot theorem lower bound extreme case event} are independent from the former, and their probability is bounded below by a positive constant by the invariance principle and standard Brownian estimates, since $N - xN^{1/2} \in [0, \delta^{1/2} N]$. This completes the proof.}
\end{proof}

\begin{proof}[Proof of Lemma~\ref{l.ballot theorem lower bound diffusive case}]
		
	As in the proof of Proposition~\ref{p.partition function lower bound}, we may assume $N\geq N_0$ for a constant $N_0$ that will be set later and that $x,y\geq 1$. We may also assume $x,y\leq N^{1/2}$ using monotonicity (Lemma~\ref{l.monotonicity}) to reduce them if they are higher and adjusting the constant $c$ at the end.
	
	Next, let $K>0$ be a large constant whose value will be set later, and assume for notational convenience that $N$ is even so that $N/2$ is an integer. We observe from the Gibbs property (in the first line) and Lemma~\ref{l.monotonicity} (monotonicity, in the second line) that
	\begin{align*}
		\MoveEqLeft[0]
		\P^{ x,y}_{[0,N]}\left(X(k) \geq 0, \, k=1, \ldots, N-1\right)\\
		&= \E^{ x,y}_{[0,N]}\left[\P^{ x,X(N/2)}_{[0,N/2]}\left(X(k) \geq 0, \, k=1, \ldots, N/2\right)\cdot \P^{ X(N/2), y}_{[0,N/2]}\left(X(k) \geq 0, \, k=1, \ldots, N/2\right)\right]\\
		&\geq \E^{ x,y}_{[0,N]}\left[\P^{ x, KN^{1/2}}_{[0,N/2]}\left(X(k) \geq 0,\, k=1, \ldots, N/2\right)\cdot \P^{ KN^{1/2}, y}_{[0,N/2]}\left(X(k) \geq 0,\, k=1, \ldots, N/2\right)\one_{X(N/2) \geq KN^{1/2}}\right]\\
		&= \P^{ x, KN^{1/2}}_{[0,N/2]}\left(X(k) \geq 0, \,k=1, \ldots, N/2\right)\cdot \P^{ KN^{1/2}, y}_{[0,N/2]}\left(X(k) \geq 0,\, k=1, \ldots, N/2\right)\\
		&\qquad\times\P^{ x,y}_{[0,N]}\left(X(N/2) \geq KN^{1/2}\right).
	\end{align*}
	By monotonicity, the last line is lower bounded by the same with $x$ and $y$ set to zero, and, by the invariance principle Lemma \ref{l.invar}, the latter is lower bounded by some $\eta = \eta(K)>0$ for all $N>N_0$, for some $N_0= N_0(K)$. We will show that for an appropriate choice of $K$ (which will be done independent of $x$ and $y$), the first factor in the second to last line of the last display is lower bounded by $cxN^{-1/2}$ for some absolute constant $c>0$; the same argument will show that the second factor is lower bounded by $cyN^{-1/2}$. This will complete the proof.
	
	By monotonicity again,
	\begin{align*}
		\MoveEqLeft[2]
		\P^{ x, KN^{1/2}}_{[0,N/2]}\left(X(k) \geq 0, \, k=1, \ldots, N/2\right)\\
		&\geq \P^{ x}_{[0,N/2]}\left(X(k) \geq 0, \, k=1, \ldots, N/2 \mid X(N/2) \leq KN^{1/2}\right)\\
		&= \frac{\P^{ x}_{[0,N/2]}\left(X(N/2) \leq KN^{1/2} \mid X(k) \geq 0, \, k=1, \ldots, N/2\right)\cdot\P^{ x}_{[0,N/2]}\left(X(k) \geq 0, \, k=1, \ldots, N/2\right)}{\P^{ x}_{[0,N/2]}(X(N/2) \leq KN^{1/2})}.
	\end{align*}
	The denominator of the last line is trivially upper bounded by $1$. By Lemma~\ref{l.ballot theorem for random walk} (2), the second factor in the numerator is lower bounded by $cxN^{-1/2}$. By considering the complement and applying the Markov inequality using the bound on the second moment from Lemma~\ref{l.ballot theorem for random walk} (3), the first factor in the numerator is lower bounded by $1/2$ for large enough $K$ (chosen independently of $x$, since $C$ in Lemma~\ref{l.ballot theorem for random walk} is independent of $x$). Thus the previous display is lower bounded by $cxN^{-1/2}$.
	\end{proof}

%!TEX root = ./rwareatilt.tex

\section{Partition function lower bound and dropping lemma}\label{sec:Zlbd}

In this section, we prove a lower bound on the partition function for a single random walk bridge above a wall with area tilt. This will be used throughout the paper to upper bound various probabilities by removing the area tilt. Recall that we use $\lambda$ to denote the area tilt coefficient. Define the scale parameter
\begin{equation}\label{eqn:H}
H_\lambda := \lambda^{-1/3}N^{1/3}.
\end{equation}
The role of this parameter is that a $\lambda$-area-tilted walk $X$ on an interval with size of order $H_\lambda^2$ has typical height of order $H_\lambda$, i.e., behaves diffusively. This can be seen heuristically by noting that the area tilt $\frac{\lambda}{N}\mathcal{A}(X)$ becomes order 1 on this scale, as $H_\lambda \cdot H_\lambda^2 = \lambda^{-1}N$.

Recall the notation laid out in Remark \ref{rmk:notation}. In particular, $\mathbb{P}_{I;0}^{u,v}$ denotes the untilted (i.e., $\lambda=0$) random walk bridge on an interval $I$ from $u$ to $v$ with floor at 0. When the floor is removed we omit the 0 subscript, and when the right boundary condition is removed (i.e., a random walk) we omit the second superscript.

\begin{proposition}\label{p.partition function lower bound}
	There exist $c^*>0$ and $N_0\in\mathbb{N}$ such that the following holds for all $\lambda\in(0,N)$ and $N\geq N_0$. For any $A,B>0$ such that $\max(A,B) \leq H_\lambda^2$ and $|I| \geq 2\max(A^{1/2}, B^{1/2})H_\lambda^2$,
	\begin{align*}
		Z_{I;0}^{\lambda;AH_\lambda,BH_\lambda} := \E^{AH_\lambda, BH_\lambda}_{I;0} \big[e^{-\frac{\lambda}{N}\cA(X)} \big] \geq (c^*)^{-1}\exp\bigg(-c^* \bigg(A^{3/2}+B^{3/2}+ \frac{|I|}{H_\lambda^{2}} \bigg)\bigg).
	\end{align*}
\end{proposition}

This is a generalization of a similar lower bound obtained in \cite[(3.3)]{ISV14}, which only considered a walk with boundary condition fixed on the $H_\lambda$ scale. It is essential here that we may take $A$ and $B$ in principle growing with $N$, as we will apply this bound on intervals with random boundary conditions that may be much larger than $H_\lambda^2$. In particular, when $I = \llbracket -N,N\rrbracket$ and $\lambda = O(1)$, we may take boundary conditions as high as $O(N)$, much larger than the typical $N^{1/3}$ scale. The ballot theorems in Section \ref{s.ballot} are needed to prove the proposition in this full range of parameters.

Before giving the proof of \cref{p.partition function lower bound}, we state the ``dropping lemma,'' which follows as an immediate consequence and is an essential tool in the proof of \cref{thm:max}. The lemma states that an area tilt of coefficient $\lambda$ quickly pulls the random walk down to a typical height $H_{\lambda}$.
\begin{lemma}[Dropping lemma]
	\label{lem:dropping-lemma}
	Recall $c^*>0$ and $N_0\in\mathbb{N}$ from \cref{p.partition function lower bound}.
	For any $\ep>0$, $\lambda \in (0,N)$, $N\geq N_0$, $A >0$, $u,v\geq 0$, and interval $I$ satisfying $\max(u,v)\leq H_\lambda^3$ and $|I| \geq 2\max(u^{1/2},v^{1/2})H_\lambda^{3/2}$, and for any $I_{\ep} \subset I$ such that $|I_{\ep}| \geq \ep |I|$, 
	\begin{align}\label{eqn:dropping-lemma} 
		\P^{\lambda;u,v}_{I;0}\big(\exists x\in I_{\ep} : X(x) \leq A H_\lambda \big) > 1- (c^*)^{-1} \exp\bigg(-(A\ep-c^*)\frac{|I|}{H_\lambda^{2}} 
		+ c^*\frac{u^{3/2} + v^{3/2}}{H_{\lambda}^{3/2}} \bigg)\,.
	\end{align}
\end{lemma}

	\begin{proof}
		If $X(x) > AH_\lambda$ for all $x\in I_\ep$, then the area $\mathcal{A}(X)$ exceeds $A\ep|I|H_\lambda$. By \cref{p.partition function lower bound},
		\begin{align*}
			\P^{\lambda;u,v}_{I;0}\Bigl(\cA(X) >A\ep |I|H_\lambda \Bigr) &\leq (c^*)^{-1}\exp\bigg(c^*\frac{|I|}{H_\lambda^{2}}+c^*\frac{u^{3/2} + v^{3/2}}{H_{\lambda}^{3/2}}\bigg) \E^{u,v}_{I;0} \left[e^{-\frac{\lambda}{N} \cA(X)} \one_{\cA(X) >A\ep  |I| H_\lambda} \right]  \\
			&\leq (c^*)^{-1} \exp\bigg(-(A\ep-c^*)\frac{|I|}{H_\lambda^2} + c^*\frac{u^{3/2} + v^{3/2}}{H_{\lambda}^{3/2}}\bigg).
		\end{align*}
		Here we used $\lambda/N = H_\lambda^{-3}$.
	\end{proof}

\begin{proof}[Proof of \cref{p.partition function lower bound}]
	We may assume without loss of generality that the interval $I$ is of the form $[ -J, J ]$ for some $J>0$. We can also assume $A=B$, by raising the smaller one to equal the other and using stochastic monotonicity, Lemma \ref{l.monotonicity}. By increasing $c^*$ appropriately and using monotonicity again, we may assume $A\geq 10$. Let $J_A = A^{1/2}H_\lambda^2$, and define the events
	\begin{align*}
		\mathsf{Left} &:= \left\{\max_{i\in[-J,-J+J_A]} X(i) \leq 2AH_\lambda, \, X(-J+J_{A})\in \left[0, H_\lambda\right]\right\}\\
		\mathsf{Mid} &:= \left\{\max_{i\in [-J + J_A, J - J_A]} X(i) \leq 2H_\lambda\right\}\\
		\mathsf{Right} &:= \left\{\max_{i \in [J - J_A, J]} X(i) \leq 2AH_\lambda, \, X(J-J_A)\in \left[0, H_\lambda\right]\right\}.
	\end{align*}
	Finally let $\mathsf{Drop} = \mathsf{Left} \cap \mathsf{Mid} \cap \mathsf{Right}$. It is immediate that on the event $\mathsf{Drop}$, 
	$$\cA(X) \leq \left(4A^{3/2} + (JH_\lambda^{-2}-2A^{1/2})\right)H_\lambda^3.$$
	Recall $H_\lambda^3 = \lambda^{-1}N$. With this we see that
	\begin{align*}
		\E^{AH_\lambda, AH_\lambda}_{I;0}\big[e^{-\frac{\lambda}{N}\cA(X)}\big] \geq \exp\left(-\left[4A^{3/2} + (JH_\lambda^{-2}-2A^{1/2})\right]\right)\cdot \P^{AH_\lambda, AH_\lambda}_{I;0}(\mathsf{Drop}).
	\end{align*}
	So it remains to lower bound $\P^{AH_\lambda, AH_\lambda}_{I;0}(\mathsf{Drop})$, in particular to show that there exists $c>0$ such that $\P^{AH_\lambda, AH_\lambda}_{I;0}(\mathsf{Drop}) \geq \exp(-c(A^{3/2}+ JH_\lambda^{-2}))$. First we note that $\mathsf{Left}$, $\mathsf{Mid}$, and $\mathsf{Right}$ are decreasing events since, $X(i) \geq 0$ due to the floor. Let $\cF_A$ be the $\sigma$-algebra generated by $X(-J+J_A), X(J-J_A)$. Then by the Gibbs property along with monotonicity,
	\begin{align}
		\P^{AH_\lambda, AH_\lambda}_{I;0}(&\mathsf{Drop}) = \E^{AH_\lambda, AH_\lambda}_{I;0}\left[\P^{AH_\lambda, AH_\lambda}_{I;0}\Bigl(\mathsf{Left}\cap \mathsf{Mid} \cap \mathsf{Right} \mid \cF_A \Bigr)\right] \notag \\
		&= \E^{AH_\lambda, AH_\lambda}_{I;0}\left[\P^{AH_\lambda, AH_\lambda}_{I;0}\Bigl(\mathsf{Left}\mid \cF_A \Bigr)\cdot \P^{AH_\lambda, AH_\lambda}_{I;0} \Bigl(\mathsf{Mid}\mid \cF_A \Bigr) \cdot \P^{AH_\lambda, AH_\lambda}_{I;0}\Bigl( \mathsf{Right}\mid \cF_A\Bigr)\right] \notag \\
		&= \E^{AH_\lambda, AH_\lambda}_{I;0}\biggl[\P^{AH_\lambda,X(-J+J_A)}_{[-J,-J+J_A];0}\left(\max_{i\in[-J,-J+J_A]} X(i)\leq 2AH_\lambda\right)\cdot \P^{X(-J+J_A),X(J-J_A)}_{[-J+J_A,J-J_A];0}(\mathsf{Mid}) \notag \\
		&\qquad\qquad\qquad \times \P^{X(J-J_A),AH_\lambda}_{[J-J_A,J];0}\left(\max_{i\in[J-J_A,J]} X(i)\leq 2AH_\lambda\right) \one_{X(-J+J_A), X(J-J_A)\leq H_\lambda}\biggr] \notag \\
		&\geq \E^{AH_\lambda, AH_\lambda}_{I;0}\biggl[\P^{AH_\lambda,H_\lambda}_{[-J,-J+J_A];0}\left(\max_{i\in[-J,-J+J_A]} X(i)\leq 2AH_\lambda\right)\cdot \P^{H_\lambda, H_\lambda}_{[-J+J_A,J-J_A];0}(\mathsf{Mid}) \notag \\
		&\qquad\qquad\qquad \times \P^{H_\lambda,AH_\lambda}_{[J-J_A, J];0}\left(\max_{i\in[J-J_A,J]} X(i)\leq 2AH_\lambda\right) \one_{X(-J+J_A), X(J-J_A)\leq H_\lambda}\biggr], \label{DGibbs}
	\end{align}
	Now, observe that the three probabilities inside the expectation are deterministic. Taking them out of the expectation shows that the remaining term is $\P^{AH_\lambda, AH_\lambda}_{I;0}(X(-J+J_A) \leq H_\lambda, X(J-J_A)\leq H_\lambda)$, which, by the Gibbs property, monotonicity, and the Harris inequality (which together can be interpreted as a special case of the FKG inequality, though we do not establish this in full), is lower bounded by the product of $\P^{AH_\lambda, AH_\lambda}_{I;0}(X(-J+J_A) \leq H_\lambda)$ and  $\P^{AH_\lambda, AH_\lambda}_{I;0}(X(J-J_A)\leq H_\lambda)$:
	\begin{align*}
		\MoveEqLeft[6]
		\P^{AH_\lambda, AH_\lambda}_{I;0}(X(-J+J_A) \leq H_\lambda, X(J-J_A)\leq H_\lambda)\\
		&= \E^{AH_\lambda, AH_\lambda}_{I;0}\left[\P^{AH_\lambda, AH_\lambda}_{I;0}\left(X(J-J_A)\leq H_\lambda \mid X(-J+J_A)\right)\one_{X(-J+J_A) \leq H_\lambda}\right]\\
		&= \E^{AH_\lambda, AH_\lambda}_{I;0}\left[\P^{X(-J+J_A), AH_\lambda}_{[-J+J_A,J];0}\left(X(J-J_A)\leq H_\lambda\right)\one_{X(-J+J_A) \leq H_\lambda}\right]\\
		&\geq \E^{AH_\lambda, AH_\lambda}_{I;0}\left[\P^{X(-J+J_A), AH_\lambda}_{[-J+J_A,J];0}\left(X(J-J_A)\leq H_\lambda\right)\right] \P^{AH_\lambda, AH_\lambda}_{I;0}\left(X(-J+J_A) \leq H_\lambda\right)\\
		&= \P^{AH_\lambda, AH_\lambda}_{I;0}(X(-J+J_A) \leq H_\lambda)\cdot \P^{AH_\lambda, AH_\lambda}_{I;0}(X(J-J_A)\leq H_\lambda),
	\end{align*}
	where the last line follows again from the Gibbs property.
	Thus we have so far shown that
	\begin{align}
			\P^{AH_\lambda, AH_\lambda}_{I;0}(\mathsf{Drop}) &\geq \P^{AH_\lambda,H_\lambda}_{[-J,-J+J_A];0}\left(\max_{i\in[-J,-J+J_A]} X(i)\leq 2AH_\lambda\right)\cdot \P^{H_\lambda, H_\lambda}_{[-J+J_A,J-J_A];0}(\mathsf{Mid}) \nonumber \\
			&\qquad\times  \P^{H_\lambda,AH_\lambda}_{[J-J_A,J];0}\left(\max_{i\in[J-J_A, J]} X(i)\leq 2AH_\lambda\right)\cdot \P^{AH_\lambda, AH_\lambda}_{I;0}\bigl(X(-J+J_A) \leq H_\lambda\bigr) \nonumber \\
			&\qquad\times \P^{AH_\lambda, AH_\lambda}_{I;0}\bigl(X(J-J_A)\leq H_\lambda\bigr) \label{e.partition function prob lower bound}
	\end{align}
	The first and third probabilities on the right-hand side of \eqref{e.partition function prob lower bound} are lower bounded by some absolute constant $\delta>0$ by monotonicity and an invariance principle for random walk bridges conditioned to stay positive \cite{caravenna2013invariance} (which assumes no more than what we have assumed already on our random walk bridge distributions). In more detail, the first probability is lower bounded by the same with the second endpoint raised to $AH_\lambda$ and the floor raised to $AH_\lambda - \smash{J_A^{1/2}} = (A-A^{1/4})H_\lambda$, which is in turn lower bounded by a constant using the mentioned invariance principle (the floor was raised so that the endpoints were above the floor by an amount on the diffusive scale). An analogous argument applies to the third probability in \eqref{e.partition function prob lower bound}. 
	
	\bigskip

	\noindent\emph{Lower bounding $\P^{H_\lambda, H_\lambda}_{[-J+J_A,J-J_A];0}(\mathsf{Mid})$}: The second probability in \eqref{e.partition function prob lower bound} is lower bounded by $\exp(-cJH_\lambda^{-2})$. To see this,
	first let $\mc F_A'$ be the $\sigma$-algebra generated by $\{X(-J+J_A+jH_\lambda^2):1 \leq j \leq 2H_\lambda^{-2}(J-J_A)-1\}$. By the Gibbs property and monotonicity, 
	\begin{align}
	\MoveEqLeft[0.5]
		\P^{H_\lambda, H_\lambda}_{[-J+J_A,J-J_A];0}(\mathsf{Mid})\nonumber\\
		&\geq \E^{H_\lambda, H_\lambda}_{[-J+J_A,J-J_A];0}\left[\P^{H_\lambda, H_\lambda}_{[-J+J_A,J-J_A];0}(\mathsf{Mid} \mid \mc F_u') \cdot \one \left\{\max_{1 \leq j \leq 2H_\lambda^{-2}(J-J_A)-1} X(-J+J_A+jH_\lambda^2) \leq H_\lambda\right\}\right]\nonumber\\
		&\geq \P^{ H_\lambda, H_\lambda}_{[0,H_\lambda^2]; 0}\left(\max_{1\leq i \leq H_\lambda^2} X(i) \leq 2AH_\lambda\right)^{2H_\lambda^{-2}(J-J_A)}\nonumber\\
		&\qquad\times \P^{H_\lambda, H_\lambda}_{[-J+J_A,J-J_A];0}\left(\max_{1 \leq j \leq 2H_\lambda^{-2}(J-J_A)-1} X(-J+J_A+jH_\lambda^2) \leq H_\lambda\right). \label{e.mid lwoer bound first break up}
	\end{align}
	The first factor is lower bounded by $\exp(-cJH_\lambda^{-2})$ using an invariance principle for random walk bridges conditioned to remain positive \cite{caravenna2013invariance} to obtain that $\P^{ H_\lambda, H_\lambda}_{H_\lambda^2; 0}(\max_{1 \leq i \leq H_\lambda^2} X(i) \leq 2AH_\lambda)$ is lower bounded by an absolute constant (here we use that $H_\lambda\geq 1$ since $\lambda<N$ and $H_\lambda^3 = \lambda^{-1}N$). Next we turn to the second factor. Using the Gibbs property and monotonicity again, we see that
	\begin{align}
	\MoveEqLeft[6]
	\P^{H_\lambda, H_\lambda}_{[-J+J_A,J-J_A];0}\left(\max_{1 \leq j \leq 2H_\lambda^{-2}(J-J_A)-1} X(-J+J_A+jH_\lambda^2) \leq H_\lambda\right)\nonumber\\
	&\geq \prod_{j=1}^{2H_\lambda^{-2}(J-J_A)-1}\P^{H_\lambda, H_\lambda}_{[-J+J_A+(j-1)H_\lambda^2,J-J_A];0}\left(X(-J+J_A+jH_\lambda^2) \leq H_\lambda\right). \label{e.mid lower bound max into product}
	\end{align}
	For notational simplicity we write $I_j = [-J+J_A+(j-1)H_\lambda^2,J-J_A]$. Now, again by monotonicity followed by the definition of conditional probability,
	\begin{align}
	\P^{H_\lambda, H_\lambda}_{I_j;0}\left(X(-J+J_A+jH_\lambda^2) \leq H_\lambda\right)\nonumber
	&\geq \P^{H_\lambda}_{I_j;0}\left(X(-J+J_A+jH_\lambda^2) \leq H_\lambda\midd X(J-J_A) \geq H_\lambda\right)\nonumber\\
	&\geq \P^{H_\lambda}_{I_j;0}\left(X(J-J_A) \geq H_\lambda\midd X(-J+J_A+jH_\lambda^2) \leq H_\lambda\right) \label{e.break up for mid bound}\\
	&\qquad\quad \times \P^{H_\lambda}_{I_j;0}\left(X(-J+J_A+jH_\lambda^2) \leq H_\lambda\right).\nonumber
	\end{align}
	Using monotonicity, we can lower bound the first factor of \eqref{e.break up for mid bound} by the same without a floor at 0, i.e., $X$ becomes a standard random walk. Then the Markov property and the central limit theorem guarantee that the first factor of \eqref{e.break up for mid bound} is lower bounded by an absolute constant. Now we turn to the second factor of \eqref{e.break up for mid bound}. This is lower bounded by
	\begin{align}
	\MoveEqLeft[6]
	\P^{H_\lambda}_{I_j;0}\left(X(-J+J_A+jH_\lambda^2) \in [\tfrac{1}{2}H_\lambda, H_\lambda]\right)\nonumber\\
	&= \frac{\P^{H_\lambda}_{I_j}\left(X(-J+J_A+jH_\lambda^2) \in [\tfrac{1}{2}H_\lambda,H_\lambda], X(i) \geq 0\ \forall i\in I_j \right)}{\P^{H_\lambda}_{I_j}\left(X(i) \geq 0\ \forall i\in I_j \right)}. \label{e.break up for mid bound second}
	\end{align}
	Using Lemma~\ref{l.ballot theorem for random walk}, the denominator is upper bounded by $CH_\lambda/|I_j|^{1/2}$ Conditioning on $X(-J+J_A+jH_\lambda^2)$ and using the Gibbs property yields that the numerator equals
	\begin{align*}
	\MoveEqLeft[3]
	\E^{H_\lambda}_{I_j}\biggl[\P^{H_\lambda, X(-J+J_A+jH_\lambda^2)}_{[-J+J_A+(j-1)H_\lambda^2, -J+J_A+jH_\lambda^2]}\left(X(i) \geq 0\ \forall i\in [-J+J_A+(j-1)H_\lambda^2, -J+J_A+jH_\lambda^2]\right)\\
	&\times \P^{X(-J+J_A+jH_\lambda^2)}_{I_{j+1}}(X_i\geq 0\  \forall i\in I_{j+1}) \one_{\frac{1}{2}H_\lambda \leq X(-J+J_A+jH_\lambda^2) \leq H_\lambda}\biggr].
	\end{align*}
	Using Lemma~\ref{l.ballot theorem for random walk} again, on the event that $X(-J+J_A+jH_\lambda^2) \geq \frac{1}{2}H_\lambda$, the first factor is lower bounded by an absolute constant and the second factor by $cH_\lambda/|I_{j+1}|^{1/2}$. The probability that $\frac{1}{2}H_\lambda \leq X(-J+J_A+jH_\lambda^2) \leq H_\lambda$ is also lower bounded by an absolute constant using the central limit theorem along with the portmanteau theorem. Overall, since $|I_{j+1}|/|I_j|$ is bounded away from zero, we obtain that the left-hand side of \eqref{e.break up for mid bound second}, and in turn the left-hand side of \eqref{e.break up for mid bound}, is lower bounded by an absolute constant. 

	Returning to \eqref{e.mid lower bound max into product} and \eqref{e.mid lwoer bound first break up} yields that
	\begin{align*}
	\P^{H_\lambda,H_\lambda}_{[-J+J_A, J-J_A]; 0}\left(\mathsf{Mid}\right) \geq \exp(-cJH_\lambda^{-2})
	\end{align*}
	for some absolute constant $c>0$.
	
	\bigskip

	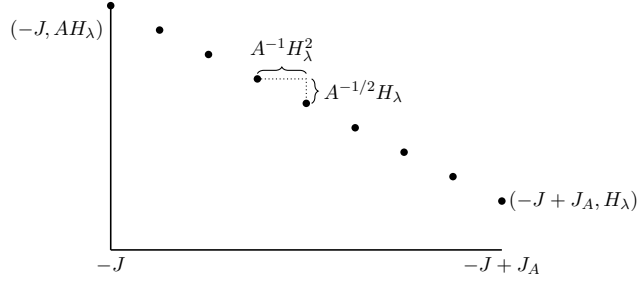
\begin{figure}
	\begin{tikzpicture}[scale=1.3]
	\draw[semithick] (0,0) -- ++(4,0);
	\draw[semithick] (0,0) -- ++(0,2.5);

	\foreach \i in {0,...,8}
		\node[circle, fill, black, inner sep = 1pt] at (0.5*\i, 2.5-0.25*\i) {};

	\node[anchor=east, scale=0.7] at (0,2.25) {$(-J, AH_\lambda)$};
	\node[anchor=west, scale=0.7] at (4,0.5) {$(-J+J_A, H_\lambda)$};
	\node[anchor=north, scale=0.7] at (0,0) {$-J$};
	\node[anchor=north, scale=0.7] at (4,0) {$-J+J_A$};

	\draw[densely dotted] (1.5, 1.75) -- ++(0.5,0) -- ++(0,-0.25);

	\draw[decoration={brace,raise=2pt},decorate] (1.5,1.75) -- node[above=4pt, scale=0.7] {$A^{-1}H_\lambda^2$} ++(0.5,0);
	\draw[decoration={brace,raise=2pt},decorate] (2,1.75) -- node[right=4pt, scale=0.7] {$A^{-1/2}H_\lambda$} ++(0,-0.25);
	\end{tikzpicture}
	\vspace{-0.1in}
	\caption{To lower bound $\P^{AH_\lambda, AH_\lambda}_{I;0}(X(J-J_A)\leq H_\lambda)$, we force the path to fall linearly as depicted above. Note that across the interval of size $J_A = A^{1/2}H_\lambda^2$, the path then falls by an amount $AH_\lambda$, which is not diffusive with respect to the interval length. However, the path does fall by an on-scale amount on the smaller scale subintervals of length $A^{-1/2}H_\lambda$, which allows us to make use of invariance principles. We expect that in the true behavior, the path would fall according to a parabola with curvature $\lambda$, but we do not care about getting the sharp coefficient in the exponent in the lower bound for the overall partition function $Z_{I;0}^{\lambda;AH_\lambda,BH_\lambda}$, and the above prescription is simpler to analyze and achieves the correct order.}\label{fig:Zlbd}
	\end{figure}

	\noindent\emph{Lower bounding $\P^{AH_\lambda, AH_\lambda}_{I;0}(X(J-J_A)\leq H_\lambda)$}: To establish the desired lower bound on $\P^{AH_\lambda,AH_\lambda}_{I;0}(\mathsf{Drop})$, it remains to lower bound the last two probabilities in \eqref{e.partition function prob lower bound}. We will show that the fourth probability is lower bounded by $\exp(-cA^{3/2})$; a symmetric argument applied to the fifth factor will then yield the same lower bound and complete the proof.
	
	To prove the lower bound on $\P^{AH_\lambda, AH_\lambda}_{I;0}\bigl(X(-J+J_A) \leq H_\lambda\bigr)$, we will consider the event that on each subinterval of $[-J,-J+J_A]$ of length $A^{-1}H_\lambda^2$, $X$ falls by an amount of order $A^{-1/2}H_\lambda$. Note that $A^{-1/2}H_\lambda$ is the fluctuation scale of a random walk bridge on an interval of size $A^{-1}H_\lambda^2$, so this event can be expected to have uniformly positive probability. Recall also that by our assumptions on $A$, $A^{-1}H_\lambda^2 \geq 9$. Since there are $A^{3/2}$ many subintervals of size $A^{-1}H_\lambda^{2}$ within the interval $[-J,-J+J_A]$ (which has length $J_A = A^{1/2}H_\lambda^2$), and since the total amount fallen will be $AH_\lambda$ (after adjusting the last fall appropriately), this will complete the proof. 
	
	Now we proceed to give the details. Let $\sigma_A = A^{-3/2}(A-1)$, so that $A^{3/2}\sigma_A = A-1$, and let $A_j = A- (j-1) \sigma_A$ and $x_j = -J+(j-1)A^{-1}H_\lambda^2$. Then we have that
	\begin{equation}\label{u32Cap}
		\P^{AH_\lambda, AH_\lambda}_{I;0}\bigl(X(-J+J_A) \leq H_\lambda\bigr) \geq \P^{AH_\lambda, AH_\lambda}_{I;0}\left(\bigcap_{j=1}^{A^{3/2}} \left\{X(x_{j+1}) \leq A_{j+1}H_\lambda\right\}\right).
	\end{equation}
	%
	% It will be convenient to extend our notation slightly and write $\P^{x, y}_{1,[a,b]}$ for a probability measure under which $X$ is distributed as a random walk bridge from $(a,x)$ to $(b,y)$ conditioned to stay above the floor at zero, so that $\P^{x, y}_{I;0} = \P^{x, y}_{1,[-J,J]}$. 
	Now by a similar argument using the Gibbs property and monotonicity as in \eqref{DGibbs}, the probability on the right side of \eqref{u32Cap} is bounded below by
	\begin{equation}\label{u32prod}
		\prod_{j=1}^{A^{3/2}} \P^{A_jH_\lambda, AH_\lambda}_{[x_j, J];0}\left(X(x_{j+1}) \leq A_{j+1}H_\lambda\right).
	\end{equation}
	The $j$\textsuperscript{th} factor in \eqref{u32prod} is lower bounded by
	\begin{align}
		\MoveEqLeft[2]
		\P^{A_jH_\lambda, AH_\lambda}_{[x_j, J]}\left(X(x_{j+1}) \in [A_{j+3/2}H_\lambda, A_{j+1}H_\lambda] \midd \min_{i\in[x_j, J]} X(i) \geq 0\right)\nonumber\\
		&= \P^{A_jH_\lambda, AH_\lambda}_{[x_j, J]}\left(X(x_{j+1}) \in [A_{j+3/2}H_\lambda, A_{j+1}H_\lambda]\right)\nonumber\\
		&\qquad\qquad\times\frac{ \P^{A_jH_\lambda, AH_\lambda}_{[x_j, J]}\left(\min_{i\in[x_j, J]} X(i) \geq 0 \midd X(x_{j+1})\in [A_{j+3/2}H_\lambda, A_{j+1}H_\lambda]\right)}{\P^{A_jH_\lambda, AH_\lambda}_{[x_j, J]}\left(\min_{i\in[x_j, J]} X(i) \geq 0\right)}\label{e.single factor breakup}
	\end{align}
	We first lower bound the last line of \eqref{e.single factor breakup} by an absolute positive constant first before returning to the second line. First we observe that by the Gibbs property and monotonicity, the numerator of the last line is lower bounded by
	\begin{align}\label{e.stay above zero ratio numerator}
		\MoveEqLeft[4]
		\P^{A_jH_\lambda, A_{j+3/2}H_\lambda}_{[x_j, x_{j+1}]}\left(\min_{i\in[x_j, x_{j+1}]} X(i) \geq 0\right)\cdot \P^{A_{j+3/2}H_\lambda, AH_\lambda}_{[x_{j+1}, J]}\left(\min_{i\in[x_{j+1}, J]} X(i) \geq 0\right)
	\end{align}
	By the invariance principle Lemma \ref{l.invar}, the first factor of \eqref{e.stay above zero ratio numerator} is lower bounded by an absolute positive constant independent of $j$. By the lower bound of the ballot theorem (Theorem~\ref{t.ballot theorem}), the second factor in \eqref{e.stay above zero ratio numerator} is lower bounded by
	\begin{align*}
		c\min\left(1, \frac{A_{j+3/2}AH_\lambda^2}{J-x_{j+1}}\right).
	\end{align*}
	By Theorem~\ref{t.ballot theorem}, the denominator of the last line of \eqref{e.single factor breakup} is bounded above by
	\begin{align*}
		C\min\left(1, \frac{A_{j}AH_\lambda^2}{J-x_{j}}\right).
	\end{align*}
	Under the assumptions on $I$ that $2J=|I|\geq 2(A^{1/2}+B^{1/2})H_\lambda^2$ and since $x_{j}\leq A^{1/2}H_\lambda^2$  and $A_j\geq A_{j+3/2}$ for every $j$, we see that the last line of \eqref{e.single factor breakup} is lower bounded by an absolute constant. 
	
	Now we turn to the second line of \eqref{e.single factor breakup}. This requires a more detailed calculation which we isolate as Lemma~\ref{l.random walk bridge local gaussianity} ahead. Assuming that lemma, the proof is complete.
\end{proof}

\begin{lemma}\label{l.random walk bridge local gaussianity}
	There exist $\delta, \rho>0$ such that, for all $u$ and $I = [-J,J]$ such that $0\leq AH_\lambda\leq \rho(J- A^{1/2}H_\lambda^2)$, $A\leq \frac{1}{9}H_\lambda^2$, and $|I|\geq 2A^{1/2}H_\lambda^2$, it holds for all $1\leq j\leq A^{3/2}$ that
	$$\P^{A_jH_\lambda, AH_\lambda}_{[x_j, J]}\left(X(x_{j+1}) \in [A_{j+3/2}H_\lambda, A_{j+1}H_\lambda]\right) > \delta.$$
\end{lemma}

\begin{proof}
	By definition,
	\begin{align*}
		\MoveEqLeft[1]
		\P^{A_jH_\lambda, AH_\lambda}_{[x_j, J]}\left(X(x_{j+1}) \in [A_{j+3/2}H_\lambda, A_{j+1}H_\lambda]\right)\\
		&= \frac{\P^{A_jH_\lambda}_{[x_j, J]}\left(X(x_{j+1}) \in [A_{j+3/2}H_\lambda, A_{j+1}H_\lambda], X(J)=AH_\lambda\right)}{\P^{A_jH_\lambda}_{[x_j, J]}\left(X(J) = AH_\lambda\right)}\\
		&= \P^{A_jH_\lambda}_{[x_j, J]}\left(X(x_{j+1}) \in [A_{j+3/2}H_\lambda, A_{j+1}H_\lambda]\right)\cdot \frac{\P^{A_jH_\lambda}_{[x_j, J]}\left(X(J)=AH_\lambda\mid X(x_{j+1}) \in [A_{j+3/2}H_\lambda, A_{j+1}H_\lambda]\right)}{\P^{A_jH_\lambda}_{[x_j, J]}\left(X(J) = AH_\lambda\right)}.
	\end{align*}
	By Donsker's invariance principle, i.e., weak convergence to Brownian motion, there exists $\delta'>0$ such that the first factor in the last line is lower bounded by $\delta'$ as long as $A^{-1/2}H\geq 3$ (as we have assumed). So we must show that the second factor is also lower bounded by an absolute constant. First, by the Gibbs property we can lower bound it by
	\begin{align*}
		\inf_{w\in [A_{j+3/2}, A_{j+1}]}\frac{\P^{wH_\lambda}_{[x_{j+1}, J]}\left(X(J)=AH_\lambda\right)}{\P^{A_jH_\lambda}_{[x_j, J]}\left(X(J) = AH_\lambda\right)} = \inf_{w\in [A_{j+3/2}, A_{j+1}]}\frac{\P^{0}_{[x_{j+1}, J]}\left(X(J)=(A-w)H_\lambda\right)}{\P^{0}_{[x_j, J]}\left(X(J) = (A-A_j)H_\lambda\right)}.
	\end{align*}
	To lower bound the expression on the right, we make use of a sharp local limit theorem for random walks from \cite{richter1957local}. Let $\varphi(x) = (2\pi)^{-1/2}\exp(-x^2/2)$ be the Gaussian density (recall we are assuming for simplicity that the walk jumps have unit variance), and $w\in[A_{j+3/2}, A_{j+1}]$.  Invoking \cite[Theorems 2 and 3]{richter1957local} respectively for continuous and discrete increment distributions (we are able to do so because of the finite MGF condition in Assumption \ref{a.misc}, along with Assumption \ref{a.convex} which forces a bounded increment density in the nonlattice case) tells us that there exists $\rho>0$ such that the following holds. Under the conditions $0\leq (A-w)H_\lambda \leq \frac{1}{2}\rho (J-x_{j+1})$ and $0 \leq (A-A_j)H_\lambda \leq \frac{1}{2}\rho (J-x_{j})$,\footnote{Applying \cite[Theorems 2 and 3]{richter1957local} as stated would require assuming a lower bound of $(J-x_{j+1})^{1/2}$ and $(J-x_{j})^{1/2}$ respectively on $(A-w)H_\lambda$ and $(A-A_j)H_\lambda$, but an inspection of the proof shows that this can be dropped at the cost of modifying the last factor in \eqref{e.local limit ratio}.} there exist $R_1, R_2\in[-\frac{1}{2}\rho^{-1}, \frac{1}{2}\rho^{-1}]$ (which may depend on $(A-w)H_\lambda$ and $(A-A_j)H_\lambda$) such that
	\begin{align}
			\MoveEqLeft[1]
			\frac{\P^{0}_{[x_{j+1}, J]}\left(X(J)=(A-w)H_\lambda\right)}{\P^{0}_{[x_j, J]}\left(X(J) = (A-A_j)H_\lambda\right)} \label{e.local limit ratio} \\
			&= \frac{(J-x_j)^{1/2}}{(J-x_{j+1})^{1/2}}\cdot\frac{\varphi\left(\frac{(A-w)H_\lambda}{(J-x_{j+1})^{1/2}}\right)}{\varphi\left(\frac{(A-A_j)H_\lambda}{(J-x_{j})^{1/2}}\right)}\cdot \frac{\exp\left(\frac{(A-w)^3H_\lambda^3}{(J-x_{j+1})^2}\cdot\Lambda\left(\frac{(A-w)H_\lambda}{J-x_{j+1}}\right)\right)}{\exp\left(\frac{(A-A_j)^3H_\lambda^3}{(J-x_{j})^2}\cdot\Lambda\left(\frac{(A-A_j)H_\lambda}{J-x_{j}}\right)\right)}\cdot \frac{1+R_1\cdot\frac{(A-w)H_\lambda}{J-x_{j+1}}}{1+R_2\cdot\frac{(A-A_j)H_\lambda}{J-x_{j}}}, \nonumber
	\end{align}
	where $\Lambda$ is a power series with radius of convergence $\rho$. Since $I\geq 4A^{1/2}H_\lambda^2$, $x_j \leq A^{1/2}H_\lambda^2$ for all $j$, and $x_j < x_{j+1}$, the first ratio in the last line is bounded below by 1. Similarly by the assumptions adopted before the previous display, the final factor is lower bounded by a positive constant.
	
	Next we focus on the second ratio in the second line of \eqref{e.local limit ratio}. Write $w = A_j - \tilde w$, so that $\tilde w\in[A^{-1/2}, \frac{3}{2}A^{-1/2}]$. Then, also using that $(1-x)^{-1} = 1+O(x)$, the second factor equals
	\begin{align*}
		\MoveEqLeft
		\exp\left(-\frac{(A-A_j)^2 + \tilde w^2 + 2\tilde w(A-A_j)}{2(J-x_{j} - A^{-1}H_\lambda^2)}H_\lambda^2 + \frac{(A-A_j)^2}{2(J-x_{j})}H_\lambda^2\right)\\
		&= \exp\left(-\frac{(A-A_j)^2 + \tilde w^2 + 2\tilde w(A-A_j)}{2(J-x_{j})}H_\lambda^2\left(1+O\left(\frac{A^{-1}H_\lambda^2}{J-x_j}\right)\right) + \frac{(A-A_j)^2}{2(J-x_{j})}H_\lambda^2\right)\\
		&= \exp\left(-O\left(\frac{(A-A_j)^2A^{-1}}{(J-x_j)^2}H_\lambda^4\right) - O\left(\frac{\tilde w(A-A_j)}{J-x_{j}}H_\lambda^2\right)\right).
	\end{align*}
	Since $J-x_j\geq A^{1/2}H_\lambda^2$ for all $j$, and using that $\tilde w$ is of order $A^{-1/2}$, each term in the exponent of the last line is bounded above by an absolute constant. 
	
	It remains to lower bound the third ratio in the second line of \eqref{e.local limit ratio}. Define $M:=\sup_{|x|\leq \frac{1}{2}\rho}|\Lambda'(x)| < \infty$. We observe by Taylor's theorem, and a calculation as in the previous display that, for some $C<\infty$,
	\begin{align}
		\Lambda\left(\frac{(A-w)H_\lambda}{J-x_{j+1}}\right)
		&\geq \Lambda\left(\frac{(A-A_j)H_\lambda}{J-x_{j}}\right) - CM\left(\frac{A^{-1/2}H_\lambda}{J-x_j} + \frac{A^{-1}(A-A_j)H_\lambda^3}{(J-x_j)^2}\right)\nonumber\\
		&\geq \Lambda\left(\frac{(A-A_j)H_\lambda}{J-x_{j}}\right) - CM A^{-1}H_\lambda^{-1}.\label{e.lambda lower bound}
	\end{align}
	It is also easy to check that
	\begin{align*}
		\frac{(A-w)^3H_\lambda^3}{(J-x_{j+1})^2} \geq \frac{(A-A_j)^3H_\lambda^3}{(J-x_{j})^2} - CA^{1/2}H_\lambda^{-1},
	\end{align*}
	and the last term is lower bounded by a constant since $A\leq H_\lambda^2$. Let us raise $M$ if needed so that also $\sup_{|x|\leq \frac{1}{2}\rho} |\Lambda(x)| \leq M$. Then we see from the previous display and \eqref{e.lambda lower bound} that
	\begin{align*}
		\frac{(A-w)^3H_\lambda^3}{(J-x_{j+1})^2}\cdot \Lambda\left(\frac{(A-w)H_\lambda}{J-x_{j+1}}\right) 
		&\geq \frac{(A-A_j)^3H_\lambda^3}{(J-x_{j})^2}\cdot \Lambda\left(\frac{(A-A_j)H_\lambda}{J-x_{j}}\right)\\
		&\qquad - CM A^{1/2}H_\lambda^{-1} - CMA^{-1}H_\lambda^{-1}\frac{(A-A_j)^3H_\lambda^3}{(J-x_{j})^2}
	\end{align*}
	Since $A\leq H_\lambda^2$, $A^{1/2}H_\lambda^{-1}\leq 1$. Since $J-x_j \geq A^{1/2}H_\lambda^2$ for all $j$, $A^{-1}H_\lambda^{-1}\frac{(A-A_j)^3H_\lambda^3}{(J-x_{j})^2} \leq A^{-2}\cdot A^3 \cdot H_\lambda^{-2} = AH_\lambda^{-2}\leq 1$. Substituting this into the third factor in \eqref{e.local limit ratio} yields that the latter is lower bounded by an absolute constant as well. This completes the proof after relabeling $\rho$.
\end{proof}

%!TEX root = ./rwareatilt.tex

\section{Upper tail bounds for a single curve}\label{sec:tail}

The goal of this section is to prove upper tail bounds for a single random walk above a wall with area tilt. We begin with the following one-point bound with characteristic Tracy--Widom-type exponent. Recall the parameter $H_\lambda$ from \eqref{eqn:H}.

\begin{theorem}\label{oneptbd}
	There exist universal constants $c,C,N_0>0$ such that for any $N\geq N_0$, $ \lambda \in (0, N)$, $R \in [0, H_\lambda^2]$, interval $I \subseteq \llbracket -N,N\rrbracket$ with $|I|\geq H_\lambda^2$, $u,v\leq AH_\lambda$ with $A>0$, and $t\in I$, we have
	\[
	\mathbb{P}_{I;0}^{\lambda;u,v}(X(t) > (R+A)H_\lambda) \leq Ce^{-cR^{3/2}}.
	\]
\end{theorem}

A result of this type originally appeared in \cite[Theorem 1.2]{HV04}, but with less explicit restrictions on the ranges of the parameters $\lambda$ and $t$. We will require the result for essentially the full range of parameters in the statement of the theorem, and the stochastic monotonicity statement Lemma \ref{l.monotonicity} allows us to give a somewhat shorter argument. Moreover, it is noted in \cite{ISV14} that there is an error in the proof of \cite[Theorem 1.2]{HV04}, and although \cite[Lemma 2]{ISV14} for a related ground-state chain is intended as a substitute, it is not immediately suited to our setup. We therefore give a self-contained proof of this result, following the same overall strategy as in \cite{HV04}. A matching lower tail bound is also given in \cite{HV04}, and a straightforward argument yields it here as well (see Remark~\ref{rmk:tail-lbd}), though we do not need it and so we do not give the details here. We also mention here the work \cite{GG}, which proves analogous tail bounds for the 2D Ising interface.

As a corollary, by including an extra logarithmic factor, we obtain an upper tail bound on the maximum of a single curve that crucially decays with the length of the interval beyond scale $R^{1/2}H_\lambda^2$.

\begin{corollary}\label{maxbd}
	Under the same assumptions as Theorem \ref{oneptbd}, if in addition $|I| \geq R^{1/2}H_\lambda^2$, $R \geq (2/c)^{2/3}$, and $R^{1/2} (\log\frac{|I|}{R^{1/2}H_\lambda^2})^{2/3}\leq H_\lambda$, then
	\[
	\mathbb{P}_{I;0}^{\lambda;u,v} \left(\max_{t\in I} X(t) > \bigg(R\bigg(\log\frac{|I|}{R^{1/2}H_\lambda^2}\bigg)^{2/3} + A\bigg)H_\lambda\right) \leq C\exp\left(-cR^{3/2}\log\frac{|I|}{R^{1/2}H_\lambda^2}\right).
	\]
\end{corollary}

\begin{remark}\label{rmk:tail-lbd}
	A simpler argument than the one below would prove matching lower bounds on the right-tail probabilities in Theorem \ref{oneptbd} (with $t$ far enough from the boundary, say $t=0$) and Corollary \ref{maxbd}. The basic observation is that by monotonicity, one can obtain a lower bound by pinning the curve to zero on the boundary of the interval $\mathcal{I} := [-R^{1/2}H_\lambda^2, R^{1/2}H_\lambda^2] \cap I$, and it then suffices to lower bound $\mathbb{P}^{\lambda;0,0}_{R^{1/2}H_\lambda;0}(X(0) > RH_\lambda)$. For this, one can first ignore the partition function (since it is at most 1), and then restrict to configurations with $\max_{t\in \mathcal{I}} X(t) < 2RH_\lambda$, so that the area tilt is bounded below by $e^{-4R^{3/2}}$. By monotonicity one can also remove the floor at zero. It then remains to lower bound the probability of $\{X(0) > RH_\lambda, \, \max_{t\in \mathcal{I}} X(t) < 2RH_\lambda\}$ by $\exp(-cR^{3/2})$, for an untilted random walk bridge $X$ on $\mathcal{I}$. This can be done by the methods of Section 4; see Figure \ref{fig:Zlbd}, where one can instead consider a subdiffusive mesh of points ascending from 0 at the boundary to $RH_\lambda$ at time $0$. To obtain the lower bound in Corollary \ref{maxbd}, one can divide the interval $I$ into segments of length $R^{1/2}H_\lambda^2$, pin the curve to zero at the boundary points, and apply the previous lower bound with independence.
\end{remark}

The idea to prove Theorem \ref{oneptbd} is to consider the largest interval $J$ around the point $t$ on which the curve $X$ remains uniformly above height $\frac{1}{2}RH_\lambda$. Either this interval $J$ is large, which will incur a large penalty from the area tilt, or $J$ is small, in which case the curve has to fluctuate by at least $\frac{1}{2}RH_\lambda$ over a short distance. The critical scale at which these two effects balance one another leads to the $R^{3/2}$ exponent in the bound.

In order to make this argument precise, we need the following generalization of the Gibbs property for conditioning on a random interval. For $X \sim \mathbb{P}_{I;0}^{\lambda;0,0}$ and $R>0$, define the two random times
\begin{equation}\label{sigmatau}
	\begin{split}
		\sigma &:= \max\{s\in I \cap \mathbb{Z}, \, s\leq t : X(s) < \tfrac{1}{2}RH_\lambda \},\\
		\tau &:= \min\{s\in I \cap \mathbb{Z}, \, s\geq t : X(s) < \tfrac{1}{2}RH_\lambda\}.
	\end{split}
\end{equation}
The zero boundary conditions ensure that $\sigma,\tau$ are well-defined. In the following we let $J$ denote the random interval $\llbracket\sigma,\tau\rrbracket$. Let $\mathcal{F}_{\sigma,\tau}$ denote the $\sigma$-algebra generated by events of the form $\mathsf{A}\cap\{\sigma=\ell\}\cap\{\tau=r\}$, where $\ell,r\in I$, $\ell<t<r$, and $\mathsf{A}$ lies in the $\sigma$-algebra $\mathcal{F}_{\ell,r}$ generated by $X(j)$ for $j\notin(\ell,r)$. Lastly, define the domain $\Upsilon := \{(f,\ell,r) : \ell,r\in I, \, \ell\leq t\leq r, \, f : \llbracket \ell, r\rrbracket \to \mathbb{R}\}$. 

\begin{lemma}\label{Jgibbs}
	For any bounded measurable functional $F : \Upsilon \to\mathbb{R}$,
	\[
	\mathbb{E}_{I;0}^{\lambda;0,0} \left[F(X|_J,\sigma,\tau) \mid \mathcal{F}_{\sigma,\tau}\right]  = \mathbb{E}_{J;\frac{1}{2}RH_\lambda}^{\lambda;X(\sigma),X(\tau)}[F(X,\sigma,\tau)].
	\]
\end{lemma}

Here we are abusing notation slightly on the right-hand side, as the boundary conditions $X(\sigma), X(\tau)$ in general lie below the floor at $\frac{1}{2}RH_\lambda$, i.e., the floor constraint is applied only in the interior. It is easy to see that the monotonicity properties of Lemma \ref{l.monotonicity} still hold in this case.

\begin{proof}[Proof of Lemma~\ref{Jgibbs}]
	Let us abbreviate $\mathbb{P}_{I;0}^{\lambda;0,0}$ by $\mathbb{P}_I$ and likewise for expectations. By definition of $\mathcal{F}_{\sigma,\tau}$, we must show that for any $\ell,r\in I$ with $\ell\leq t\leq r$ and any event $\mathsf{A} \in \mathcal{F}_{\ell,r}$,
	\begin{equation}\label{Jcondexp}
		\mathbb{E}_I \left[F(X|_J,\sigma,\tau)\one_{\mathsf{A}\cap\{\sigma=\ell\}\cap\{\tau=r\}}\right] = \mathbb{E}_I\left[\mathbb{E}_{J;\frac{1}{2}RH_\lambda}^{\lambda;X(\sigma),X(\tau)} \left[F(X,\sigma,\tau)\right] \one_{\mathsf{A}\cap\{\sigma=\ell\}\cap\{\tau=r\}}\right].
	\end{equation}
	First note that if either $\ell=t$ or $r=t$, then $X(t) < \frac{1}{2}RH_\lambda$ and in fact $\ell=r=t$. In this case, both sides of \eqref{Jcondexp} are simply equal to $\mathbb{E}_I[F(X(t),t,t)\one_{\mathsf{A}\cap\{X(t) < \frac{1}{2}RH_\lambda\}}]$. It therefore suffices to prove \eqref{Jcondexp} for $\ell<t<r$. Define the events
	\begin{equation}\label{IntBd}
		\mathsf{Int}(\ell,r) := \left\{\min_{s\in\llbracket\ell+1,r-1\rrbracket} X(s) \geq \tfrac{1}{2}RH_\lambda\right\}, \qquad
		\mathsf{Bd}(\ell,r) := \left\{X(\ell),X(r) < \tfrac{1}{2}RH_\lambda\right\}.
	\end{equation}
	Then $\{\sigma=\ell\}\cap\{\tau=r\}
	=\mathsf{Int}(\ell,r)\cap\mathsf{Bd}(\ell,r)$. So in order to prove \eqref{Jcondexp} it suffices to show that for each $\ell,r\in I$ with $\ell<t<r$ and $\mathsf{A}\in\mathcal{F}_{\ell,r}$,
	\begin{equation}\label{Jcondexp2}
		\mathbb{E}_I \left[F(X|_{[\ell,r]},\ell,r)\one_{\mathsf{A}} \one_{\mathsf{Int}(\ell,r)}\one_{\mathsf{Bd}(\ell,r)}\right] = \mathbb{E}_I\left[\mathbb{E}_{[\ell,r];\frac{1}{2}RH_\lambda}^{\lambda;X(\ell),X(r)} [F(X,\ell,r)] \one_{\mathsf{A}} \one_{\mathsf{Int}(\ell,r)}\one_{\mathsf{Bd}(\ell,r)}\right].
	\end{equation}
	Note that $\mathsf{Bd}(\ell,r)\in\mathcal{F}_{\ell,r}$. We condition on $\mathcal{F}_{\ell,r}$ and apply the tower property, followed by the Gibbs property, followed by the tower property in reverse, to see that 
	\begin{align*}
		&\mathbb{E}_I \left[F(X|_{[\ell,r]},\ell,r)\one_{\mathsf{A}} \one_{\mathsf{Int}(\ell,r)}\one_{\mathsf{Bd}(\ell,r)}\right] \\
		&\qquad\qquad =\mathbb{E}_I \left[ \mathbb{E}_I \left[F(X|_{[\ell,r]},\ell,r)\one_{\mathsf{A}} \one_{\mathsf{Int}(\ell,r)} \mid \mathcal{F}_{\ell,r}\right]\one_{\mathsf{Bd}(\ell,r)} \right]\\ 
		&\qquad\qquad = \mathbb{E}_I \left[ \mathbb{E}_{[\ell,r];0}^{\lambda;X(\ell),X(r)}\left[F(X,\ell,r) \one_{\mathsf{Int}(\ell,r)}\right] \one_{\mathsf{A}}\one_{\mathsf{Bd}(\ell,r)}\right]\\
		&\qquad\qquad = \mathbb{E}_I \left[\mathbb{E}_{[\ell,r];0}^{\lambda;X(\ell),X(r)}\left[F(X,\ell,r) \mid \mathsf{Int}(\ell,r)\right] \mathbb{P}_I\left(\mathsf{Int}(\ell,r)\mid\mathcal{F}_{\ell,r}\right) \one_{\mathsf{A}}\one_{\mathsf{Bd}(\ell,r)}\right]\\
		&\qquad\qquad = \mathbb{E}_I\left[\mathbb{E}_{[\ell,r];\frac{1}{2}RH_\lambda}^{\lambda;X(\ell),X(r)}[F(X,\ell,r)] \one_{\mathsf{A}}\one_{\mathsf{Int}(\ell,r)}\one_{\mathsf{Bd}(\ell,r)}\right].
	\end{align*}
	This proves \eqref{Jcondexp2}.
\end{proof}

With this lemma, we can prove the one point bound using the strategy outlined above.

\begin{proof}[Proof of Theorem \ref{oneptbd}]
	First observe that it suffices to prove the theorem and corollary when $u=v=0$, since if $u,v\leq AH_\lambda$ then $X\sim\mathbb{P}_I^{\lambda;u,v}$ is stochastically dominated by $Y + AH_\lambda$ where $Y\sim\mathbb{P}_I^{\lambda;0,0}$. In this case we can take $A=0$ in the desired bound. We can also of course assume $R\geq 1$ by taking $C\geq e$ and $c\leq 1$. Using Lemma \ref{Jgibbs} we write
	\begin{align*}
		\mathbb{P}_{I;0}^{\lambda;0,0}(X(t)>RH_\lambda) &= \mathbb{E}_{I;0}^{\lambda;0,0} \left[\mathbb{E}_{I;0}^{\lambda;0,0} \left[\one_{X(t)>RH_\lambda} \mid \mathcal{F}_{\sigma,\tau}\right] \right] = \mathbb{E}_{I;0}^{\lambda;0,0} \left[\mathbb{P}_{J;\frac{1}{2}RH_\lambda}^{\lambda;X(\sigma),X(\tau)}(X(t)>RH_\lambda)\right].
	\end{align*}
	We now split into cases depending on the length of the interval $J$. Define $\Delta := K\sqrt{R}\,H_\lambda^2$, where $K>0$ is a sufficiently large constant. The reason for this choice of $\Delta$ will become apparent in the argument. We then use the above along with stochastic monotonicity to bound
	\begin{equation}\label{oneptsplit}
		\begin{split}
			\mathbb{P}_{I;0}^{\lambda;0,0}(X(t)>RH_\lambda) &= \mathbb{E}_{I;0}^{\lambda;0,0} \left[\one_{|J| \leq \Delta} \,\mathbb{P}_{J;\frac{1}{2}RH_\lambda}^{\lambda;X(\sigma),X(\tau)}(X(t)>RH_\lambda)\right]\\
			&\qquad + \mathbb{E}_{I;0}^{\lambda;0,0} \left[ \one_{|J| > \Delta}\, \mathbb{P}_{J;\frac{1}{2}RH_\lambda}^{\lambda;X(\sigma),X(\tau)}(X(t)>RH_\lambda)\right]\\
			&\leq \mathbb{E}_{I;0}^{\lambda;0,0} \left[\one_{|J| \leq \Delta} \,\mathbb{P}_{J;\frac{1}{2}RH_\lambda}^{\lambda;\frac{3}{4}RH_\lambda, \frac{3}{4}RH_\lambda}(X(t)>RH_\lambda)\right] + \mathbb{P}_{I;0}^{\lambda;0,0}(|J| > \Delta).
		\end{split}
	\end{equation}
	In the last line we used the fact that $X(\sigma), X(\tau) \leq \frac{1}{2}RH_\lambda \leq \frac{3}{4}RH_\lambda$ by definition of $\sigma$ and $\tau$.
	
	In the remainder of the proof, we will show that there are constants $c,C>0$ and $N_0\in\mathbb{N}$ depending on $K$ so that for all $N\geq N_0$, $0<\lambda<N$, and $0\leq R\leq 2H_\lambda^2$, it holds that
	\begin{equation}\label{Jlengthbd}
		\mathbb{P}_{I;0}^{\lambda;0,0}(|J|>\Delta) \leq Ce^{-cR^{3/2}},
	\end{equation}
	and
	\begin{equation}\label{smallJbd}
		\one_{|J| \leq \Delta} \, \mathbb{P}_{J;\frac{1}{2}RH_\lambda}^{\lambda;\frac{3}{4}RH_\lambda, \frac{3}{4}RH_\lambda}(X(t)>RH_\lambda) \leq Ce^{-cR^{3/2}}.
	\end{equation}
	Applying these bounds in \eqref{oneptsplit} completes the proof. We will prove these estimates in two separate steps.
	
	\subsection*{Step 1: Bounding the size of $J$} We first prove \eqref{Jlengthbd}. We will decompose the probability as follows. Define
	\[
	A_\Delta := \{(\ell, r) \in I\times I : \ell < t < r, \, r-\ell > \Delta\}.
	\]
	Recall the events $\mathsf{Int}(\ell,r)$ and $\mathsf{Bd}(\ell,r)$ from \eqref{IntBd}, and note that $\{|J|>\Delta\}$ is the disjoint union over $(\ell,r)\in A_\Delta$ of $\mathsf{Int}(\ell,r)\cap\mathsf{Bd}(\ell,r)$. Therefore, conditioning on $\mathcal{F}_{\ell,r}$ and applying the Gibbs property followed by monotonicity,
	\begin{equation}\label{Jdecomp}
		\begin{split}
			\mathbb{P}_{I;0}^{\lambda;0,0}(|J| > \Delta) &= \sum_{(\ell,r)\in A_\Delta} \mathbb{E}_{I;0}^{\lambda;0,0}\left[\mathbb{P}_{I;0}^{\lambda;0,0} \left( \mathsf{Int}(\ell,r)\cap\mathsf{Bd}(\ell,r) \mid\mathcal{F}_{\ell,r}\right)\right]\\
			&= \sum_{(\ell,r)\in A_\Delta} \mathbb{E}_{I;0}^{\lambda;0,0} \left[ \one_{\mathsf{Bd}(\ell,r)} \mathbb{P}_{[\ell,r];0}^{\lambda;X(\ell),X(r)}(\mathsf{Int}(\ell,r))\right]\\
			&\leq \sum_{(\ell,r)\in A_\Delta} \mathbb{P}_{[\ell, r];0}^{\lambda;\frac{1}{2}RH_\lambda, \frac{1}{2}RH_\lambda}(\mathsf{Int}(\ell, r)).
		\end{split}
	\end{equation}
	We will now bound the summands in the last line of \eqref{Jdecomp}. Note that on the event $\mathsf{Int}(\ell,r)$, the area tilt is bounded above by $\exp(-\frac{\lambda}{N}\cdot\frac{1}{2}RH_\lambda(r-\ell)) = \exp(-\frac{1}{2}R(r-\ell)H_\lambda^{-2})$. Therefore, bounding above the indicator for the walk remaining nonnegative by 1,
	\begin{equation}\label{HighLRred}
		\begin{split}
			\mathbb{P}_{[\ell, r];0}^{\lambda;\frac{1}{2}RH_\lambda, \frac{1}{2}RH_\lambda}(\mathsf{Int}(\ell, r)) &\leq (Z_{[\ell,r];0}^{\lambda;\frac{1}{2}RH_\lambda, \frac{1}{2}RH_\lambda})^{-1}\\
			&\qquad\qquad \times \exp\left(-\tfrac{1}{2}R(r-\ell)H_\lambda^{-2}\right) \mathbb{P}_{[\ell,r]}^{\frac{1}{2}RH_\lambda,\frac{1}{2}RH_\lambda}(\mathsf{Int}(\ell,r)).
		\end{split}
	\end{equation}
	Here we recall that $\mathbb{P}_{[\ell,r]}^{u,v}$ denotes the law of a random walk bridge with no floor or area tilt. By monotonicity and the ballot theorem (Theorem~\ref{t.ballot theorem}), there is a universal constant $C$ so that 
	\begin{equation}\label{ballotbd}
		\mathbb{P}_{[\ell,r]}^{\frac{1}{2}RH_\lambda,\frac{1}{2}RH_\lambda}(\mathsf{Int}(\ell,r)) \leq \mathbb{P}_{[\ell,r]}^{0,0} \bigg(\min_{j\in(\ell,r)} X(j) > 0\bigg) \leq \frac{C}{r-\ell},
	\end{equation} 
	uniformly in $r-\ell > \Delta$. By Proposition \ref{p.partition function lower bound},
	\begin{equation}\label{oneptZlbd}
		Z_{[\ell,r];0}^{\lambda;\frac{1}{2}RH_\lambda, \frac{1}{2}RH_\lambda} \geq c\exp\left(-C(R^{3/2} + (r-\ell)H_\lambda^{-2})\right).
	\end{equation}
	Indeed, the proposition applies as long as $K$ is large enough since since $r-\ell > \Delta = K\sqrt{R}\,H_\lambda^2$, and $\frac{1}{2}RH_\lambda \leq H_\lambda^2$ since we have assumed $R\leq \frac{1}{2}H_\lambda$. Altogether, choosing $K$ large enough depending on $c,C$, using $r-\ell > K\sqrt{R}\,H_\lambda^2$, and applying \eqref{ballotbd} and \eqref{oneptZlbd} in \eqref{HighLRred} implies that
	\begin{equation*}
		\mathbb{P}_{[\ell, r];0}^{\lambda;\frac{1}{2}RH_\lambda, \frac{1}{2}RH_\lambda}(\mathsf{Int}(\ell, r)) \leq \frac{C}{r-\ell}\, \exp\left(-\tfrac{1}{4}R(r-\ell)H_\lambda^{-2}\right).
	\end{equation*} 
	Now returning to \eqref{Jdecomp}, we change variables via $r = \ell + k$ where $k>\Delta$, to obtain
	\begin{align*}
		\mathbb{P}_{I;0}^{\lambda;0,0}(|J|>\Delta) &\leq \sum_{(\ell,r)\in A_\Delta} \frac{C}{r-\ell}\, \exp\left(-\tfrac{1}{4}R(r-\ell)H_\lambda^{-2}\right) = \sum_{k > \Delta} \sum_{\ell < t < \ell + k} \frac{C}{k}\exp\left(-\tfrac{1}{4}RkH_\lambda^{-2}\right)\\
		&\leq \sum_{k>\Delta} C \exp\left(-\tfrac{1}{4}RkH_\lambda^{-2}\right) \leq C\exp\left(-c R\Delta H_\lambda^{-2}\right) \leq C\exp(-cR^{3/2}).
	\end{align*}
	For the first inequality in the second line, we noted that the number of $\ell$ with $\ell < t < \ell + k$ is at most $k$. For the subsequent inequality we used a straightforward geometric sum estimate, and for the final inequality we recalled $\Delta = K\sqrt{R}\,H_{\lambda}^2$. This proves~\eqref{Jlengthbd}.

	\subsection*{Step 2: Upper tail bound for small $J$} We now prove \eqref{smallJbd}. This is a straightforward application of random walk bridge estimates from Section \ref{s.RWest}. By stochastic monotonicity, we can remove the area tilt and shift vertically down by $\frac{1}{2}RH_\lambda$ to estimate
	\begin{equation}\label{smallJfrac}
		\begin{split}
			\mathbb{P}_{J;\frac{1}{2}RH_\lambda}^{\lambda;\frac{3}{4}RH_\lambda, \frac{3}{4}RH_\lambda}(X(t)>RH_\lambda)  &\leq  \mathbb{P}_{J;0}^{\frac{1}{4}RH_\lambda,\frac{1}{4}RH_\lambda}(X(t)>\tfrac{1}{2}RH_\lambda)\\ 
			&\leq  \frac{\mathbb{P}_J^{\frac{1}{4}RH_\lambda,\frac{1}{4}RH_\lambda}(X(t)>\tfrac{1}{2}RH_\lambda)}{\mathbb{P}_J^{\frac{1}{4}RH_\lambda,\frac{1}{4}RH_\lambda}(\min_{s\in J} X(s) > 0)}.
		\end{split}
	\end{equation}
	For the numerator, Lemma \ref{l.bridge maximal} implies for $|J| \leq \Delta = K\sqrt{R}\,H_\lambda^2$ that
	\begin{equation}\label{subexp}
		\mathbb{P}_J^{\frac{1}{4}RH_\lambda,\frac{1}{4}RH_\lambda}(X(t)>\tfrac{1}{2}RH_\lambda) \leq 4\exp\bigg(-c\bigg(\frac{R^2H_\lambda^2}{t-\sigma} \wedge RH_\lambda\bigg)\bigg) \leq 4e^{-cR^{3/2}}.
	\end{equation}
	In the last inequality, we used $t-\sigma\leq K\sqrt{R}\,H_\lambda^2$ to bound $R^2H_{\lambda}^2/(t-\sigma)\geq K^{-1}R^{3/2}$, as well as the hypothesis $R\leq H_\lambda^2$ to bound $RH_\lambda \geq K^{-1}R^{3/2}$. On the other hand, the ballot theorem (Theorem~\ref{t.ballot theorem}) implies that the denominator in \eqref{smallJfrac} is bounded below by the minimum of $c>0$ and $c(\frac{1}{4}RH_\lambda)^2/|J| \geq \frac{1}{8}cK^{-1}R^{3/2} \geq \frac{1}{8}cK^{-1}>0$. In combination with \eqref{smallJfrac} and \eqref{subexp}, this proves~\eqref{smallJbd}.
\end{proof}

We now give the proof of the maximum bound.

\begin{proof}[Proof of Corollary \ref{maxbd}]
	This is a direct consequence of Theorem \ref{oneptbd} and a union bound. As before we can assume $u=v=0$ and $A=0$. If $I = \llbracket a,b\rrbracket$, let $L = \lfloor \frac{|I|}{R^{1/2} H_\lambda^2} \rfloor $, $t_i = \lfloor a + iR^{1/2}H_\lambda^2\rfloor$ for $0\leq i < L$, and $t_L = b$. Write $M = R(\log\frac{|I|}{R^{1/2}H_\lambda^2})^{2/3} H_\lambda$. Then by a union bound, applying the Gibbs property and monotonicity,
	\begin{equation}\label{maxbdsum}
		\mathbb{P}_{I;0}^{\lambda;0,0} \left(\max_{t\in I} X(t) > M\right) \leq \sum_{i=1}^{L-1} \bigg[ \mathbb{P}_{I;0}^{\lambda;0,0}\left(X(t_i) > \frac{M}{2}\right)  + \mathbb{P}_{[t_i,t_{i+1}];0}^{\lambda;\frac{M}{2}, \frac{M}{2}} \bigg(\max_{t\in[t_i,t_{i+1}]} X(t) > M \bigg) \bigg].
	\end{equation}
	By Theorem \ref{oneptbd}, for each $i$ we have for the first term in the summand
	\begin{equation}\label{maxbd1st}
		\mathbb{P}_{I;0}^{\lambda;0,0}\left(X(t_i) > \frac{M}{2}\right)  \leq C\exp\bigg(-cR^{3/2} \log \frac{|I|}{R^{1/2}H_\lambda^2}\bigg).
	\end{equation} 
	For the second term in the $i^{\mathrm{th}}$ summand in \eqref{maxbdsum}, we use monotonicity to remove the area tilt, followed by Lemma \ref{l.bridge maximal} and the ballot theorem (Theorem~\ref{t.ballot theorem}) with $t_{i+1}-t_i\leq R^{1/2}H_\lambda^2$ to estimate
	\begin{align*}
		&\mathbb{P}_{[t_i,t_{i+1}];0}^{\lambda;\frac{M}{2}, \frac{M}{2}} \bigg(\max_{t\in[t_i,t_{i+1}]} X(t) > M \bigg) \leq \frac{\mathbb{P}_{[t_i,t_{i+1}]}^{\frac{M}{2},\frac{M}{2}}\left(\max_{t\in[t_i,t_{i+1}]} X(t) > M\right)}{\mathbb{P}_{[t_i,t_{i+1}]}^{\frac{M}{2},\frac{M}{2}}\left(\min_{t\in[t_i,t_{i+1}]} X(t) > 0\right)}\\ 
		&\qquad \qquad\qquad\qquad\quad \leq C\exp\bigg(-c\bigg(\frac{M^2}{R^{1/2}H_\lambda^2} \wedge M\bigg)\bigg) =  C\exp\left(-cR^{3/2}\bigg(\log\frac{|I|}{R^{1/2}H_\lambda^2}\bigg)^{4/3}\right).
	\end{align*}
	In the last line we used the assumption $R^{1/2}(\log\frac{|I|}{R^{1/2}H_\lambda^2})^{2/3} \leq H_\lambda$, and the condition $|I| \geq R^{1/2}H_\lambda^2$ implies the last expression is no larger than the right-hand side of \eqref{maxbd1st} (after possibly enlarging $C$). Therefore each summand in \eqref{maxbdsum}  is bounded by the same factor on the right of \eqref{maxbd1st}. The number of terms in the sum is $L-1 \leq |I|/R^{1/2}H_\lambda^2$, and since $R^{3/2} \geq 2/c$, this factor can be absorbed into the exponent by replacing $c$ with $c/2$. This proves the desired bound.
\end{proof}

%!TEX root = ./rwareatilt.tex
\section{Global ceilings for the line ensemble}\label{sec:recursion}

This section is dedicated to the proof of \cref{thm:max}, which follows from \cref{thm:max-general} below. We begin by setting up and stating \cref{thm:max-general}, followed by the proof of \cref{thm:max}. We give an outline of the proof of \cref{thm:max-general} in \cref{sub:outline_of_the_proof_of_cref_thm_max_general}, followed by the proof in \cref{subsec:proof_of_cref_thm_max_general}.

We start with some notation, used throughout this section and \cref{sec:recursion-proofs}. 
Fix $a_0>0$ and $b_0>1$. 
All results in the next two sections are stated for $a\geq a_0$ and $b\geq b_0$, where $(a,b)$ are the area-tilt parameters from \cref{def:areatilt}.
Define
\begin{equation}\label{e.H_j defn}
\lambda_j := ab^{j-1}, \qquad H_j := \lambda_j^{-1/3}N^{1/3}, \qquad j\in\mathbb{N}\,.
\end{equation}
 That is, compared to our notation for single curves in \eqref{eqn:H}, we write $H_j$ instead of $H_{\lambda_j}$ for brevity. This will provide  the fluctuation scale for the $j^{\mathrm{th}}$ curve in the line ensemble.  For real numbers $p < q$, $r \in \R$, and $s>0$, we will sometimes write $r+s[p,q]$ or $r+[p,q]s$ to denote the interval $[r+ps, r+qs]$ for brevity. We sometimes write $x\vee y := \max(x,y)$.

Next, choose a constant $\cC= \cC(b_0)>1$ large enough so that
\begin{align}\label{def:cC}
	\frac{(j+1+\cC)^2 H_{j+1}}{(j+\cC)^2H_j} \leq b^{-1/6}\quad \mbox{for all} \quad j\geq 1\, \text{ and } b\geq b_0.
\end{align}
Note the above is implied by $(2+\cC)^2/(1+\cC)^2 \leq b_0^{1/6}$. 
Define 
\begin{align}\label{def:epj}
	\ep_j := (j+\cC)^{-2}\,.
\end{align}

Recall the choices of parameters from \eqref{def:boundary-conditions-thm:max}--\eqref{def:BL} in \cref{thm:max}. There, the parameter $L$ relating to the interval size was fixed, and restrictions on the maximum size of the boundary conditions $\max(u_1,v_1)$ and the tail-depth parameter $K$ were given in terms of $L$. While this parametrization was natural for the theorem statement, it will be convenient for the proof to reparametrize by imposing restrictions on $L$ in terms of the largest allowed size $B$ of $N^{-1/3}\max(u_1,v_1)$.

We consider $n$ curves on an interval $\mathcal{I}$ defined by 
\begin{equation}\label{def:rec-length-of-interval2}
	\begin{split}
	& n = \lfloor N^\delta \rfloor, \qquad \mathcal{I}= [-LN^{2/3}, LN^{2/3}] \,.
	\end{split}
\end{equation}
Fix $r\in (0,2/3)$ and $\eta >0$. For a constant $K_0>0$  to be chosen sufficiently large, we consider parameters $K,B,L$ satisfying
\begin{equation}\label{def:BLK}
	K \in [K_0, N^r(\log N)^{-2/3}] \,, \quad 
	B \in [K(\log N)^{2/3},N^r]\,, \quad
	L \in [B^{1/2+\eta}, N^{1/3}]\,,  
\end{equation}
and boundary conditions  $\u = (u_j)_{j=1}^n$ and $\v = (v_j)_{j=1}^n$ in the Weyl chamber $W_0^n$ satisfying
\begin{multline}\label{eqn:boundary-conditions-recursion2}
	\max(u_1, v_1) + K \ep_1^{-1} H_{1} \bigg(\log \frac{|\mathcal{I}|}{K^{1/2}H_{1}^2}\bigg)^{2/3} \leq BN^{1/3} \quad \text{and}  \\
	 \min(u_{j-1}, v_{j-1}) \geq \max(u_j, v_j) + K \ep_j^{-1} H_{j} \bigg(\log \frac{|\mathcal{I}|}{K^{1/2}H_{j}^2}\bigg)^{2/3} \quad \text{for all $j\in \llbracket 2, m+1\rrbracket$,}
\end{multline}
where we define the threshold
\begin{align}\label{def:m}
	m= \lfloor (2/3- r)\log_b N - \log_b a\rfloor ,\quad \text{so that} \quad H_{m+1}  \in N^{1/9 + r/3}[1, b^{1/3}]\,.
\end{align} 
The choice of the above parameters, as well as the objects appearing in \eqref{def:Ij}--\eqref{def:ceilingj2} below, are motivated in the proof outline of \cref{thm:max-general} in \cref{sub:outline_of_the_proof_of_cref_thm_max_general}.
Define the intervals 
\begin{align}\label{def:Ij}
	\mathcal{I}_j = [x_j^{\mathrm{L}}, x_j^{\mathrm{R}}] := \Big[-\frac{1}{2}\prod_{i=j}^{m}(1-\ep_i) LN^{2/3} , \, \frac{1}{2} \prod_{i=j}^{m}(1-\ep_i) LN^{2/3}\Big]  
	\qquad
	\text{and}
	\qquad
	\mathcal{I}_{m+1} := \frac{1}{2} \mathcal{I}\,.
\end{align}
Note that $\mathcal{I}_j = (1-\ep_j) \mathcal{I}_{j+1}$ for $j\in\intint{1,m}$. Since
\begin{align}\label{eqn:epj-is-good}
	\prod_{i=1}^m (1-\ep_i) = \prod_{i=1}^{m} \frac{(i+\cC-1)(i+\cC+1)}{(i+\cC)^2} = \frac{\cC}{\cC+1}\left(1+\frac{1}{\cC+m}\right) > 1/2\,,
\end{align}
we have for all $j\in \llbracket1,m+1\rrbracket$ and $C>0$,
\begin{align}\label{eqn:Ij-lengthbound}
	\mathcal{I}_j \supset \tfrac14 \mathcal{I} \,, \qquad \text{ so } \quad |\mathcal{I}_j| \geq \tfrac14 B^{1/2+\eta}  N^{2/3} \geq C K^{1/2}\ep_j^{-1/2}H_j^2, 
\end{align}
where the last inequality holds for all $N$ large enough depending on $a_0$, $b_0$, and $C$. Let $\mathcal T>0$ denote a large constant depending only on $a_0$ and $b_0$, to be specified in the course of the argument. For $j \in \llbracket 1, m+1 \rrbracket$, define the $j^{\mathrm{th}}$ ``ceiling function'' $\Cl_j(x)= \Cl_j(x;K, a,b, \u, \v)$ by 
\begin{align}\label{def:ceilingj2}
	\Cl_j(x) := 
	\begin{cases} 
		K\ep_j^{-1}H_{j} \Big(\log(2\mathcal{T} \ep_j^{-1/2})\Big)^{2/3},  & |x| \leq 2\mathcal TK^{1/2} \ep_j^{-1/2}H_j^2, \\
		K \ep_j^{-1}H_{j} \Big(\log\frac{|x|}{K^{1/2}H_{j}^2}\Big)^{2/3}, &x \in \mathcal{I}_j, \, |x| > 2 \mathcal T K^{1/2}\ep_j^{-1/2}H_j^2, \\
		\max(u_{j},v_{j}) + K\ep_j^{-1}H_{j} \Big(\log \frac{|\mathcal{I}|}{K^{1/2}H_{j}^2}\Big)^{2/3}, & x\in \mathcal{I}\setminus\mathcal{I}_j \,.
	\end{cases}
\end{align} 
Note $\Cl_j$ is continuous on $\cI_j$ and, by \eqref{eqn:boundary-conditions-recursion2}, non-decreasing in $|x|$. \cref{eqn:boundary-conditions-recursion2} also gives
\begin{align} \label{eq:boundaryconditions-above-floor}
	\min(u_{j}, v_{j}) \geq \max_{x \in \mathcal{I}} \Cl_{j+1}(x) = \max(u_{j+1}, v_{j+1}) + K \ep_j^{-1} H_j \bigg(\log \frac{|\mathcal{I}|}{K^{1/2}H_{j}^2}\bigg)^{2/3} \quad \forall j\in \llbracket 1,m\rrbracket\,.
\end{align}
The above is convenient because it will allow us to use monotonicity in the floor (\cref{l.monotonicity}) to stochastically dominate a walk with boundary conditions $u_j$ and $v_j$ and floor at $\Cl_{j+1}$ by a walk with the same boundary conditions but floor at $\max_{x\in \cI} \Cl_{j+1}(x)$ without having to worry about issues such as this raised floor being higher than the boundary conditions.

 We  derive \cref{thm:max} from \cref{thm:max-general} below, whose proof is in Section \ref{subsec:proof_of_cref_thm_max_general}. 

\begin{theorem}\label{thm:max-general}
Fix $a_0>0$, $b_0>1$, $r \in (0,2/3)$, and $\eta>0$. There exist positive constants $K_0(a_0,b_0,\eta)$, $\delta = \delta(r)$, $c=c(a_0,b_0)$, and $C = C(a_0,b_0,\eta)$ such that the following holds. There exists $N_0= N_0(a_0,b_0,r,\eta)\in\mathbb{N}$ so that for all $N\geq N_0$, $a \in [a_0, N^{1/3-r/2}]$, $b\in[b_0, \exp((\frac23-r)(\log N)^{1/3})]$, $n$ and $\mathcal{I}$ as in \eqref{def:rec-length-of-interval2}, $(K,B,L)$ as in \eqref{def:BLK}, and $\u,\v \in W_0^n$ satisfying \eqref{eqn:boundary-conditions-recursion2}, 
\begin{equation}\label{eqn:max-general}
	\P_{n,\mathcal{I};0}^{a,b;\u,\v}\Big(\exists x \in \mathcal{I} : X_1(x) >2\Cl_1(x) \Big) \leq Ce^{-cK^{3/2}}.
\end{equation}
\end{theorem}

Although \cref{thm:max-general} is only stated for the top curve $X_1$, in the proof of \cref{thm:max} below, we show that allowing $a$ to be as large as a power of $N$ allows us to get bounds on the $j\th$ curve $X_j$ of the ensemble. In \cref{rk:thm6.1-j}, we explain that the same idea gives a global ceiling for $X_j$ analogous to \eqref{eqn:max-general} for $X_1$. 

\begin{proof}[Proof of \cref{thm:max}]
We apply \cref{thm:max-general} to prove \cref{thm:max} for $t=0$, as our ceiling functions $\Cl_j$ are defined to be non-decreasing in $|x|$. In light of \eqref{eqn:Ij-lengthbound}, which implies $\cap_j \cI_j \supset \frac14 \mathcal{I}$, we could instead use the ceilings $\Cl_j(\cdot-t)$, which are non-decreasing away from $t$, once $N$ is large enough so that $(T + |t|)\sigma^{-2/3}N^{2/3} \in \frac14 \mathcal{I}$. The same argument below would yield \cref{thm:max-general} with these  ceilings (in particular, $C$ and $c$ do not depend on $t$), \cref{thm:max} with that $t$ follows in exactly the same way as below. We give the proof in two steps, for $j=1$ and for $j \in \llbracket 2, \frac{\kappa}{20}\log_b N\rrbracket$.

\medskip
\noindent 
\textbf{Step 1: $j=1$.} 
Starting from the setup of \cref{thm:max}, we will use monotonicity (Lemma~\ref{l.monotonicity}) to raise the boundary conditions so that \cref{thm:max-general} can be applied. By similar reasoning as in \cref{rmk:remove-top}, monotonicity in the floor allows us to increase $n$ in \cref{thm:max} to $n= \lfloor N^{\delta} \rfloor$, with, say, zero boundary conditions for the additional curves. Next, for $\u, \v \in W_0^n$ satisfying \eqref{def:BL}, we raise each $u_i$ and $v_i$ by the smallest amount such that the second condition in \eqref{eqn:boundary-conditions-recursion2} is satisfied. Call these new boundary conditions $\u^*:= (u_1^*, \dots, u_n^*), \v^*:= (v_1^*,\dots, v_n^*)$. Then
\begin{multline} \label{eq:raising-bc}
	\max(u_1^*, v_1^*) \leq \max(u_1, v_1) + \sum_{j=1}^{m} K \ep_j^{-1} H_j \bigg(\log \frac{|\mathcal I|}{K^{1/2}H_j^2}\bigg)^{2/3} \\
	\leq \bigg(L^{2-\kappa} +  K (\log N)^{2/3}\sum_{j=1}^{\infty} (j+\cC)^2 \lambda_j^{-1/3} \bigg) N^{1/3} \leq CL^{2-\kappa} N^{1/3}\,,
\end{multline}
where in the first inequality, we used \eqref{def:BL}, $H_j\geq 1$ for $j\in\llbracket 1,m\rrbracket$ from \eqref{def:m}, and  $|\mathcal{I}| \leq 2N$ from \eqref{def:boundary-conditions-thm:max}; in the second inequality, we used the upper bound on $K$ in the statement of \cref{thm:max}. Note $C$ above depends only on $a_0$ and $b_0$.
Take $B := CL^{2-\kappa}$ and $r:= \frac23-\frac{\kappa}{6}$ in \cref{thm:max-general}. 
Recalling the bounds on $L$ from \eqref{def:boundary-conditions-thm:max}, we see that $B$  satisfies \eqref{def:BLK} for $N$ large. The assumed upper bound on $K$ in the statement of \cref{thm:max} implies $K$ satisfies \eqref{def:BLK} as well. Moreover, $\eta$ can be chosen sufficiently small with respect to $\kappa$ so that for  all $N$ sufficiently large, 
$B^{1/2+\eta} \leq L$.
In particular, $L$ satisfies \eqref{def:BLK} as well.

Thus, we see that the law $\P_{n,\mathcal{I};0}^{a,b; \u,\v}$ is stochastically dominated by $\P_{n,\mathcal{I};0}^{a,b; \u^*,\v^*}$, and  the latter law satisfies the conditions of \cref{thm:max-general}. 
Fixing $T>0$ and assuming 
$a\leq N^{1/3-r/2} = N^{\kappa/12}$ and $b \leq \exp((\frac23-r)(\log N)^{1/3}) =\exp(\frac{\kappa}{6}(\log N)^{1/3})$, we may apply \cref{thm:max-general} (recalling the definition of $\Cl_1$ from \eqref{def:ceilingj2} and $\ep_1$ from \eqref{def:epj}) to obtain 
\begin{equation}\label{eqn:Thm6.1toThm2.14}
	\P_{n,\mathcal I;0}^{a,b;\u,\v}\Big( \exists x \in [-T\sigma^{-2/3}N^{2/3}, T\sigma^{-2/3}N^{2/3}] : X_1(x) > 2K(1+\cC)^2H_1^2\Upsilon(x)^{2/3}  \Big) \leq Ce^{-cK^{3/2}},
\end{equation}
where\begin{align*}
	\Upsilon(x) := \log \bigg(  \frac{a^{2/3}|x|}{K^{1/2}N^{2/3}} \vee 2\mathcal{T}(1+\cC)\bigg)
	&\leq C\log \bigg(2 + \frac{a^{2/3}|x|}{K^{1/2}N^{2/3}} \bigg) \,,
\end{align*}
with $C$ depending only on $\mathcal{T}$, $\cC$ (which in turn depend only on $a_0$, $b_0$). 
Substituting the above into \eqref{eqn:Thm6.1toThm2.14}, recalling \eqref{def:x_i}, and setting $x = s \sigma^{-2/3}N^{2/3}$ and $\tilde C = \tilde{C}(a_0,b_0) = 2(1+\cC)^2C^{2/3}\sigma^{-2/3}$ gives
\begin{equation*}\label{eqn:Thm6.1toThm2.14-2}
	\P_{n,\mathcal I;0}^{a,b;\u,\v}\bigg( \exists s \in [-T, T] : x^N_1(s) > \tilde CKa^{-1/3}N^{1/3}  \bigg[\log \bigg(2 + \frac{a^{2/3}\sigma^{-2/3}|s|}{K^{1/2}}\bigg) \bigg]^{2/3} \bigg)
	\leq Ce^{-cK^{3/2}}.
\end{equation*}
Take $K_0 = K_0(a_0,b_0)$ large so that $C$ above can be absorbed into $c$. Replacing $K$ with $K/\tilde{C}$, and then modifying $c = c(a_0,b_0)$ so as to replace the $K^{1/2}$ in the denominator with $cK^{1/2}$, we are done.
\medskip
\noindent
\textbf{Step 2: $j \in \llbracket 2, \frac{\kappa}{20}\log_b N \rrbracket$.} 
Using monotonicity to remove the top $j-1$ curves as in Remark \ref{rmk:remove-top}, we stochastically dominate $X_j$ by the top curve of the line ensemble with law $\smash{\mathbb{P}^{\lambda_j,b;\u^{\geq j},\v^{\geq j}}_{n-j+1,I;0}}$ (recalling $\u^{\geq j}$, $\v^{\geq j}$ from \cref{rmk:remove-top}). 
It therefore suffices to establish the bound
\begin{equation}\label{eq:jth-curve-max}
\mathbb{P}^{\lambda_j,b;\u^{\geq j},\v^{\geq j}}_{n-j+1,I;0} \bigg( \exists s\in[-T,T] : x_1^N(s+t) > K \lambda_j^{-1/3} \bigg[\log \bigg(2 + \frac{\lambda_j^{2/3}|s|}{cK^{1/2}} \bigg) \bigg]^{2/3} \bigg)
		\leq e^{-cK^{3/2}}.
\end{equation}
\cref{eq:jth-curve-max} will follow from applying the result we just proved in Step 1 (\cref{thm:max} for $x_1^N$) to the above line ensemble. Observe that, compared to the line ensemble considered in Step 1, the number of curves has been reduced from $n$ to $n-j+1$, the highest boundary conditions have been lowered from $(u_1, v_1)$ to $(u_j, v_j)$, and the area-tilt parameter $a$ has been replaced by $\lambda_j$. The first two differences affect nothing; therefore, we just need to verify that the Step 1 result still applies upon replacing $a$ with $\lambda_j$. 
The only hypothesis on $a$ used in Step 1 was the condition $a \leq N^{1/3-r/2} =N^{\kappa/12}$
from Theorem \ref{thm:max-general}, so for any $j$ such that $\lambda_j = ab^{j-1} \leq N^{\kappa/12},$  
the Step 1 result applies. Since we assumed in Theorem~\ref{thm:max} that $a\leq N^{\kappa/30}$ and $j\leq \frac{\kappa}{20}\log_b N$, this condition on $\lambda_j$ is satisfied, and taking  $a= \lambda_j$ in the Step 1 result (\cref{thm:max} for $x_1^N$) yields \eqref{eq:jth-curve-max}.
\end{proof}

\begin{remark}\label{rk:thm6.1-j}
In the above proof, the $j=1$ version of \cref{thm:max} led to the full result for  $N = N(a,b,\kappa)$ large  by considering the top curve of a line ensemble with area-tilt parameters $(\lambda_j,b)$. In exactly the same way, for all $N= N(a,b,r)$ large and for all $j\leq 1 + (\frac13-\frac{r}2)\log_b N - \log_b a$ (so $\lambda_j \leq N^{1/3-r/2}$), \cref{thm:max-general} can be applied to the top curve of $\mathbb{P}^{\lambda_j,b;\u^{\geq j},\v^{\geq j}}_{n-j+1,I;0}$, yielding
\begin{equation}\label{eqn:max-general-j}
	\P_{n,\mathcal{I};0}^{a,b;\u,\v}\Big(\exists x \in \mathcal{I} : X_j(x) >2\Cl_1\big(x;K, \lambda_j, b, \u^{\geq j}, \v^{\geq j}\big) \Big) \leq Ce^{-cK^{3/2}}\,,
\end{equation}
where, to clarify, $\Cl_1(\cdot ; K,\lambda_j,b, \u^{\geq j}, \v^{\geq j})$ is the right-hand side of \eqref{def:ceilingj2} with $j=1$,  $(u_1,v_1)= (u_j,v_j)$, and $H_1 =H_j$ (since $\lambda_j^{-1/3} N^{1/3} = H_j$). Thus, we have a ceiling on the $j\th$ curve $X_j$.
\end{remark}

\subsection{Outline of the proof of \cref{thm:max-general}} \label{sub:outline_of_the_proof_of_cref_thm_max_general}

We now outline the main ideas underlying the proof of \cref{thm:max-general}. This expands on the discussion in \cref{ssub:reduction_to_single_curve_estimates_via_stochastic_monotonicity,ssub:high_boundary_conditions_and_dropping_estimates}, where it was explained that a recursive argument reduces the problem to showing that, w.h.p.\ summable in $j$, a single curve $X$ with area-tilt coefficient $\lambda_j$ and floor at $\Cl_{j+1}$ lies uniformly below $\Cl_j$. For curves of index $j\leq m$, this is \cref{prop:recursive-bound}, depicted in Figure~\ref{f.recursion first figure}.

For curves of index $j> m$, the area tilt is too large for the single-curve tail bounds or dropping lemma to be applied, as these require $H_{\lambda}$ to not be too small (see, e.g., the hypotheses of \cref{lem:dropping-lemma} and \cref{oneptbd}, as well as \cref{rk:dropping-lemma-condition-1} below). On the other hand, the strong area tilt has the effect of pushing these curves very low, and a crude strategy will show that they are negligible. By monotonicity, we can bound these curves by decreasing all of their area-tilt  parameters to $\lambda_{m+1}$ and applying a similar recursive strategy. This allows us to show, in \cref{prop:max-general-bottom-curves}, that $X_{m+1}$ is w.h.p. dominated by a piecewise constant ceiling $\mathsf{F}$, which itself is dominated by {$\Cl_{m+1}$}.

We prove \cref{prop:recursive-bound} in \cref{sec:recursion-proofs} by creating a ceiling on the random walk starting from the boundary points of $\cI$ and working our way inwards. 
First, as depicted in Figure~\ref{f.recursion first figure}, note that $\Cl_j$ features a high, flat portion on $\cI\setminus\cI_j$; inside of $\cI_j$, it resembles $H_j (\log |x|/H_j^2)^{2/3}$, which is the equilbrium behavior of the $\lambda_j$-tilted walk. This structure can be understood as follows. Because we allow our boundary values $(u_j, v_j)$ to be as large as $BN^{1/3}$ (see \eqref{eqn:boundary-conditions-recursion2}, \eqref{def:BLK}), the walk fluctuates around this height for some time. Moreover, since the walk must stay above $\Cl_{j+1}$, which features an anomalously high portion on $\cI\setminus\cI_{j+1}$, the walk cannot drop down to its equilbrium height until after it moves inside of $\cI_{j+1}$. Hence, we have the nesting of the intervals $\cI_{j} \subset \cI_{j+1}$ as in \eqref{def:Ij}.

This equilibration is proved in \cref{lem:drop-high-bc}, where we construct random ``drop points'' within distance $\ep_j |\mathcal{I}_{j+1}|$ of $\partial\cI_{j+1}$ (in particular, inside of $\cI_{j+1}\setminus \cI_j$) where the walk has fallen to $\approx \ep_j^{-1} H_j + \Cl_{j+1}$; that is, the walk has dropped to approximately equilbrium height above the floor. This is depicted in Figure~\ref{f.recursion second figure}. 
Recall from \eqref{def:epj} that $\ep_j := (j+\cC)^{-2}$. 

The choice of $\ep_j$ comes from the proof of \cref{lem:drop-high-bc}, which uses an iterative scheme driven by the dropping lemma (\cref{lem:dropping-lemma}). For a $\lambda$-tilted random walk and for any $\ep\in(0,1)$, the dropping lemma constructs, inside any $\ep$-fraction of the interval, a random point where that the walk lies below height $\ep^{-1} H_\lambda$. 
Essentially any sequence $\ep_j$ summable in $j$ would suffice; the relevant point is that $\prod_j (1-\ep_j) > 0$ (as noted in \eqref{eqn:epj-is-good}), so that, iterating down to $j=1$, all curves drop to their equilibrium height 
% above the ``bulk'' part of the floor
within a constant fraction of the interval $\mathcal{I}$; in particular, the length of $\cI_1$ is a constant fraction of $|\mathcal I| = 2LN^{2/3}$ (see \eqref{eqn:Ij-lengthbound}). Thus, when $L = B^{1/2+\eta}$,  the top curve equilibriates from its boundary values at $BN^{1/3}$ in $B^{1/2+\eta}N^{2/3}$ time, establishing the nearly parabolic rate of descent mentioned in \cref{rk:parabolic-decay}. 
The $B^{\eta}$ factor in the length of the interval is needed for the aforementioned iterative scheme in the proof of \cref{lem:drop-high-bc}: see \cref{claim:drop-high-bc-input} and Figure~\ref{f.recursion second figure}.
The constant $\cC$ in $\ep_j$ serves to maintain the ordering $\Cl_{j+1}(x) \leq \Cl_j(x)$ for all $x\in \cI$; see \eqref{def:cC}.

Once we have shown that the walk drops from its (potentially very high) boundary values, the remainder of the proof of \cref{prop:recursive-bound} proceeds in two steps. In Step 1, the dropping lemma and the one-point bound Theorem \ref{oneptbd} are used successively in \cref{claim:meshpoints} to argue that the walk drops to height $\ep_j^{-1}H_j$ above $\Cl_{j+1}$ at a deterministic mesh of points approaching the origin on either side. The mesh stops upon crossing $\smash{\pm2\mathcal{T} K^{1/2}\ep_j^{-1/2}H_j^2}$, which are the boundary points of the interval around $0$ on which $\Cl_j$ is flat \eqref{def:ceilingj2}. The role of $\mathcal{T}$ is to serve as a large constant so that the negative $|I|/H_\lambda^2$ term in the exponential in \eqref{eqn:dropping-lemma} dominates the positive terms there.
In Step 2, we ``fill in'' the ceiling on the random walk between mesh points using the max bound \cref{maxbd}.

\subsection{Proof of \cref{thm:max-general}}
\label{subsec:proof_of_cref_thm_max_general}
The proof of \cref{thm:max-general} has two main inputs: \cref{prop:max-general-bottom-curves,prop:recursive-bound}.
Before stating \cref{prop:max-general-bottom-curves}, we define   (recalling $n$ from \eqref{def:rec-length-of-interval2} and $m$ from \eqref{def:m}) 
\begin{align*}
	 \mathsf{F}(x) &= \mathsf{F}(x;n,m,K) := 2n^2KH_{m+1} \log \frac{|\mathcal{I}|}{K^{1/2}H_{m+1}^2} + \max(u_{m+1},v_{m+1})\one_{\mathcal{I}\setminus\mathcal{I}_{m+1}}(x), \quad x\in\mathcal{I}.
\end{align*}
\cref{prop:max-general-bottom-curves} below shows that $\mathsf{F}$ serves as the ceiling function for all curves below level $m$. 

\begin{proposition}\label{prop:max-general-bottom-curves}
Under the same hypotheses as \cref{thm:max-general}, we have
\begin{equation}\label{eq:max-general}
	\P_{n-m,\mathcal{I};0}^{\lambda_{m+1},b;\u^{>m},\v^{>m}}\Big(\exists x \in \cI : X_{1}(x) > \mathsf{F}(x) \Big) \leq Cn e^{-cm^{3/2} K^{3/2}},
\end{equation}
where $\mathbf{u}^{>m} := (u_{m+1},\dots,u_n)$ and $\mathbf{v}^{>m} := (v_{m+1},\dots,v_n)$. {(To clarify, the only condition on $\mathbf{u}^{>m},\mathbf{v}^{>m}$ is that they are ordered and lie below $u_{m+1},v_{m+1}$ respectively, where $(u_i)_{i=1}^{m+1}, (v_i)_{i=1}^{m+1}$ satisfy \eqref{eqn:boundary-conditions-recursion2}.)}
\end{proposition}

\begin{figure}
	\includegraphics[scale=1.1]{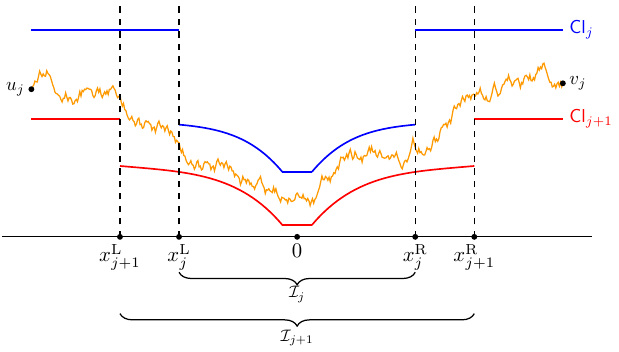}
	\vspace{-0.1in}
	\caption{For $j\leq m$, $\mathsf{Cl}_{j+1}$ (red) features a very high portion outside $\mathcal I_{j+1}$, coming from a global max bound, and a portion that grows like $\smash{\ep_j^{-2}H_{j+1}(\log\frac{|x|}{K^{1/2}H_{j+1}^2})^{2/3}}$ inside $\mathcal I_{j+1}$. The random walk bridge (orange) with area tilt $\lambda_j$, boundary conditions $u_j, v_j \leq BN^{1/3}$, and floor at $\mathsf{Cl}_{j+1}$ is shown in Proposition \ref{prop:recursive-bound} to stay below $\mathsf{Cl}_j$ (blue) with high probability.}
	\label{f.recursion first figure}
\end{figure}

For curves above level $m$, we have the following single-curve estimate.
\begin{proposition} \label{prop:recursive-bound}
Under the same hypotheses as \cref{thm:max-general}, for $j\in \llbracket1,m\rrbracket$,
	\begin{align}
		\P_{\mathcal{I};\Cl_{j+1}}^{\lambda_{j};u_{j}, v_{j}}\Big(\exists x \in \mathcal{I} \ : \ X(x) > \Cl_{j}(x)\Big) \leq  Ce^{-cK^{3/2}\log(j+1)}\,.
	\end{align}
\end{proposition}

\cref{prop:max-general-bottom-curves,prop:recursive-bound} are proved in \cref{sec:recursion-proofs}.

\begin{proof}[Proof of \cref{thm:max-general}]
First observe that for all $x\in\mathcal{I}$, we have the crude bound
\begin{equation}\label{eqn:m+1-ceiling-bound}
	\mathsf{F}(x) \leq \Cl_{m+1}(x) + 2n^2KH_{m+1} \log \frac{|\mathcal{I}|}{K^{1/2}H_{m+1}^2} \leq \Cl_{m+1}(x) + \zeta\,,
\end{equation}
where $\zeta := 2N^{2\delta} KH_{m+1} \log N$. Define the events
\begin{align}\label{def:ceil}
\mathsf{Ceil}_j := \left\{\forall x\in\mathcal{I} : X_j(x) \leq \mathsf{Cl}_j(x) + \zeta \right\}, \qquad j\in\llbracket 1,m+1\rrbracket \,.
\end{align}
Let $\mathcal{F}_j$ denote the $\sigma$-algebra generated by $X_i$ for $i\neq j$. Applying the Gibbs property and monotonicity as in Remark \ref{rmk:remove-top} to remove the top curves, for $j\in\llbracket 1,m\rrbracket$ we can write
\begin{align*}
\mathbb{P}_{n,\mathcal{I};0}^{a,b;\u,\v} ( \neg \mathsf{Ceil}_j) &= \mathbb{E}_{n,\mathcal{I};0}^{a,b;\u,\v} \left[ \mathbb{P}_{n,\mathcal{I};0}^{a,b;\u,\v} \left( \neg\mathsf{Ceil}_j \mid \mathcal{F}_j\right) \right] \leq \mathbb{E}_{n,\mathcal{I};0}^{a,b;\u,\v} \left[ \mathbb{P}_{\mathcal{I};X_{j+1}}^{\lambda_j;u_j,v_j}(\exists x\in\mathcal{I} : X(x) > \mathsf{Cl}_j(x) + \zeta) \right] \\
&\leq \mathbb{P}_{n,\mathcal{I};0}^{a,b;\u,\v} (\neg\mathsf{Ceil}_{j+1}) + \mathbb{P}_{\mathcal{I};\mathsf{Cl}_{j+1} + \zeta}^{\lambda_j;u_j+\zeta,v_j+\zeta}(\exists x\in\mathcal{I} : X(x) > \mathsf{Cl}_j(x) + \zeta)\\
&= \mathbb{P}_{n,\mathcal{I};0}^{a,b;\u,\v} (\neg\mathsf{Ceil}_{j+1}) + \mathbb{P}_{\mathcal{I};\mathsf{Cl}_{j+1}}^{\lambda_j;u_j,v_j}(\exists x\in\mathcal{I} : X(x) > \mathsf{Cl}_j(x))\,.
\end{align*}
In the second line, we used monotonicity first to raise the boundary conditions from $(u_j,v_j)$ to $(u_j+\zeta, v_j+\zeta)$, and then, on the event $\mathsf{Ceil}_{j+1}$, to raise the floor from $X_{j+1}$ to $\mathsf{Cl}_{j+1}+\zeta$. In the last line we used vertical shift invariance, see Remark \ref{rmk:shift}. Iterating over $j\in\llbracket 1,m\rrbracket$, we obtain
\begin{equation}\label{eqn:ceil-sum}
	\mathbb{P}_{n,\mathcal{I};0}^{a,b;\u,\v} ( \neg \mathsf{Ceil}_1) \leq \mathbb{P}^{a,b;\u,\v}_{n,\mathcal{I};0}(\neg\mathsf{Ceil}_{m+1}) + \sum_{j=1}^{m} \mathbb{P}_{\mathcal{I};\mathsf{Cl}_{j+1}}^{\lambda_j;u_j,v_j}(\exists x\in\mathcal{I} : X(x) > \mathsf{Cl}_j(x)).
\end{equation}
For the first term, monotonicity to remove the top $m$ curves, Proposition \ref{prop:max-general-bottom-curves}, and \eqref{eqn:m+1-ceiling-bound} yield
\begin{equation}\label{eqn:ceil-m+1}
\mathbb{P}^{a,b;\u,\v}_{n,\mathcal{I};0}(\neg\mathsf{Ceil}_{m+1}) \leq \mathbb{P}_{n-m,\mathcal{I};0}^{\lambda_{m+1},b;\u^{>m},\v^{>m}} \bigg(\exists x\in\mathcal{I} : X_1(x) > \mathsf{Cl}_{m+1}(x) + \zeta \bigg) \leq Cne^{-c m^{3/2} K^{3/2}}.
\end{equation}
Since $n= N^{\delta}$,  we can absorb the $n$ pre-factor on the right-hand side of \eqref{eqn:ceil-m+1} into $c$ as follows. Note from \eqref{def:m} that the conditions $a \leq N^{1/3-r/2}$ and $b \leq \exp((\frac23-r)(\log N)^{1/3})$ imply $m^{3/2} \geq [\frac{1}{2}(\frac{2}{3}-r)\log_b N - 1]^{3/2} \geq (\frac{2}{3}-r) \log N$ for $N = N(r)$ large enough. Taking $K_0 \geq c^{-2/3}$ and $\delta < 1/3-r/2$ then implies $\log n \leq (c/2)m^{3/2}K^{3/2}$, and replacing $c$ with $c/2$ gives an upper bound in \eqref{eqn:m+1-ceiling-bound} of $Ce^{-cmK^{3/2}}$.
Then applying Proposition \ref{prop:recursive-bound} to \eqref{eqn:ceil-sum} and recalling \eqref{def:ceil} yields
\begin{equation*}
\mathbb{P}_{n,\mathcal{I};0}^{a,b;\u,\v} \Big( \exists x\in\mathcal{I} : X_1(x) > \mathsf{Cl}_1(x) + \zeta \Big) \leq Ce^{-cmK^{3/2}} + \sum_{j=1}^m Ce^{-cK^{3/2}\log(j+1)} \leq Ce^{-cK^{3/2}}.
\end{equation*}
Finally, note that as long as $\delta<\frac{1}{2}(\frac{1}{9}-\frac{r}{6})$, for all $N$ large depending only on $r$, and all $x\in \mathbb{R}$ (using $r < \frac{2}{3}$, \eqref{def:m}, and our assumed upper bounds on $a$ and $b$), we have
\[
\zeta \leq 2K N^{2\delta} \log N \cdot b^{-m/3}H_1 \leq 2K N^{2\delta} \log N \cdot b^{1/3} N^{-(1/9 -r/6)}  H_1 \leq KH_1 \leq \Cl_1(x)\,.
\]
The desired bound \eqref{eqn:max-general} now follows from the last two displays.
\end{proof}

%!TEX root = ./rwareatilt.tex

\section{Near-parabolic dropping and construction of ceilings}
\label{sec:recursion-proofs}

This section is dedicated to the proofs of \cref{prop:max-general-bottom-curves,prop:recursive-bound}. Throughout this section, we use the setup and  notation laid out at the start of \cref{sec:recursion}.
In particular, we fix $a_0>0$, $b_0>1$
 and recall from the hypotheses of \cref{thm:max-general} (recall \cref{prop:max-general-bottom-curves,prop:recursive-bound} share these hypotheses) the constants 
$c= c(a_0,b_0)>0$ and $C=C(a_0,b_0,\eta)>0$. Below, we allow them to change line-to-line while maintaining the same dependencies.
Lastly, fix $a\geq a_0$, $b\geq b_0$ as in the hypotheses of \cref{thm:max-general} (the upper bounds on $a$ and $b$ there do not play a role in this section).

As in previous sections, we will repeatedly invoke the (strong) Gibbs property (\cref{def:gibbs}), monotonicity (\cref{l.monotonicity}, \cref{rmk:remove-top}), and invariance under constant vertical shifts of the floor and walk (\cref{rmk:shift}).

The proof of \cref{prop:recursive-bound} is given in the next three subsections. The proof of \cref{prop:max-general-bottom-curves} is given in the last subsection, Section~\ref{subsec:pf-bottom-curves}, as it is similar but much simpler.

\begin{remark}\label{rk:dropping-lemma-condition-1}
In this section, we will only deal with area tilts $\lambda_j = ab^{j-1}$ for $j\leq m+1$. From \eqref{def:BLK}--\eqref{def:m}, our boundary conditions lie below $BN^{1/3}\leq N^{1/3+r}\leq H_j^3$. This ensures that in every application of \cref{lem:dropping-lemma} that follows, the condition $\max(u,v)\leq H_{\lambda}^3$ will be satisfied. 
\end{remark}

\subsection{Proof of \cref{prop:recursive-bound}}
We prove \cref{prop:recursive-bound} using three inputs: Lemmas~\ref{lem:rec-global-max-bound}, \ref{lem:drop-high-bc}, and \ref{lem:recursive-bluepart}. 
\cref{lem:rec-global-max-bound} is a bound on the maximum height of the walk on all of $\mathcal{I}$. This immediately gives the $x\not\in \mathcal{I}_{j}$ portion of $\Cl_{j}$ in \cref{prop:recursive-bound}. 
\cref{lem:drop-high-bc} allows us to get away from our potentially very high boundary conditions. \cref{lem:recursive-bluepart} considers a random walk with area tilt $\lambda_{j}$ and ``more reasonable'' boundary conditions (in our proof of \cref{prop:recursive-bound}, such boundary conditions will be given to us by \cref{lem:drop-high-bc}, see Figure \ref{f.recursion second figure}) and states that such a random walk, with floor at $\Cl_{j+1}$, is bounded by $\Cl_{j}$ everywhere with high probability.

\begin{lemma} \label{lem:rec-global-max-bound}
Under the hypotheses of \cref{prop:recursive-bound}, there exists $K_0>0$ such that for any $j \in \llbracket 1,m\rrbracket$, 
	\begin{align} \label{eqn:rec-global-max-bound}
		\P_{\mathcal{I};\Cl_{j+1}}^{\lambda_{j};u_{j}, v_{j}}\Bigg( \max_{x \in \mathcal{I}} X(x) > \max(u_{j}, v_{j}) +  KH_{j}\bigg(\log \frac{|\mathcal{I}|}{K^{1/2}H_{j}^2}\bigg)^{2/3}\Bigg) \leq Ce^{-cjK^{3/2}}\,.
	\end{align}
\end{lemma}
The proof of \cref{lem:rec-global-max-bound} is short, so we give it before proceeding.
\begin{proof}[Proof of \cref{lem:rec-global-max-bound}]
		\cref{eqn:boundary-conditions-recursion2} implies $\max(u_j,v_j) \geq \max_{x\in \mathcal{I}} \Cl_{j+1}$.  Monotonicity (\cref{l.monotonicity}) allows us to increase the boundary conditions to $\max(u_{j},v_{j})$ and then raise the floor from $\Cl_{j+1}$ to $\max(u_{j},v_{j})$. Lowering the floor and the walk  by $\max(u_{j},v_{j})$ via shift invariance (\cref{rmk:shift}) then yields that the left-hand side of \eqref{eqn:rec-global-max-bound} is bounded by 
		\[
		\P_{\mathcal{I};0}^{\lambda_{j};0,0}\Bigg( \max_{x \in \mathcal{I}} X(x) >   KH_{j}\bigg(\log \frac{|\mathcal{I}|}{K^{1/2}H_{j}^2}\bigg)^{2/3}\Bigg) \leq C\exp\bigg(-cK^{3/2}\log\frac{|\mc I|}{K^{1/2}H_j^2}\bigg)\,,
		\]
		where the final inequality follows from taking $R= K$ and $A=0$ in \cref{maxbd} and $K_0$ large; the conditions of \cref{maxbd} are satisfied because \eqref{eqn:Ij-lengthbound} implies $|\mathcal I| \geq K^{1/2}H_j^2$, and \eqref{def:BLK},  \eqref{def:m}, and $r<2/3$ imply $K< N^{r} \leq H_{m+1}^2\leq H_j^2$ for $N= N(r)$ large. The bound \eqref{eqn:rec-global-max-bound} follows because \eqref{e.H_j defn} and \eqref{eqn:Ij-lengthbound} give $\log (|\mc I|/(K^{1/2}H_{j}^2)) \geq \log (2b_0^{2(j-1)/3}) \geq c j$ for $c>0$ depending only on $b_0$. 
	\end{proof}

The next two lemmas are proved in Sections~\ref{s.proof of high bc dropping} and \ref{s.proof of recursive blue part}, respectively. 
For Lemma~\ref{lem:drop-high-bc}, recall $\eta$ and $B$ from~\eqref{def:BLK}.

\begin{lemma}\label{lem:drop-high-bc}
Fix $c_0 \in (0,1)$. Under the hypotheses of \cref{prop:recursive-bound}, there exists $D_0>0$ such that
for all $j\in \llbracket 1, m+1\rrbracket$, $D\in [D_0,B]$, $\ep' \geq c_0\lambda_j^{-1/6}$, and all  $y^{\mathrm{L}}\leq y^{\mathrm{R}}$ such that $y^{\mathrm{R}}-y^{\mathrm{L}} \geq B^{1/2+\eta}N^{2/3}$,
\begin{multline}\label{eqn:drop-high-bc-zerofloor}
	\P_{[y^{\mathrm{L}},y^{\mathrm{R}}];0}^{\lambda_j; BN^{1/3}, BN^{1/3}}\bigg(\exists x \in [y^{\mathrm{L}}, y^{\mathrm{L}}+\ep'(y^{\mathrm{R}}-y^{\mathrm{L}})] ,\, x' \in [y^{\mathrm{R}}-\ep'(y^{\mathrm{R}}-y^{\mathrm{L}}), y^{\mathrm{R}}] : \\
	\max\big(X(x) , X(x')\big)\leq \frac{D}{\ep'}H_j \bigg)
	\geq 1- C\exp\Big(-c D^{3/2} \lambda_j^{2/3}\Big)\,.
\end{multline}
\end{lemma}

The value of the exponent $1/6$ in the lower bound on $\ep'$ is insignificant; it arises from \cref{lem:dropping-lemma}.

\begin{lemma}\label{lem:recursive-bluepart}
Recall $\mathcal{I}_j=[x_j^{\mathrm{L}},x_j^{\mathrm{R}}]$ from \eqref{def:Ij}, and consider the hypotheses of \cref{prop:recursive-bound}.
There exists a constant $C_1>0$ large depending only on $a_0$ and $b_0$ such that, letting
\begin{align}\label{def:wj*}
	w_j^* := \tfrac{K}{C_1}\ep_j^{-1}H_j+\Cl_{j+1}(x_{j+1}^{\mathrm{L}}) = \tfrac{K}{C_1} \ep_j^{-1}H_j+\Cl_{j+1}(x_{j+1}^{\mathrm{R}})\,,
\end{align}
the following holds for all $K$ satisfying \eqref{def:BLK} with $K_0>0$ large enough and for all $j\in \llbracket 1, m\rrbracket$:
\begin{align}
	\P_{\mathcal{I}_{j};\Cl_{j+1}}^{\lambda_{j}; w_j^*, w_j^*}\Big(\exists x \in \mathcal{I}_{j} : X(x) > \Cl_{j}(x) \Big) \leq C\exp\big(-cK^{3/2}\log(j+1)\big) \,.
\end{align}
\end{lemma}

Now we may prove \cref{prop:recursive-bound}.

\begin{figure}
	\includegraphics[scale=1.1]{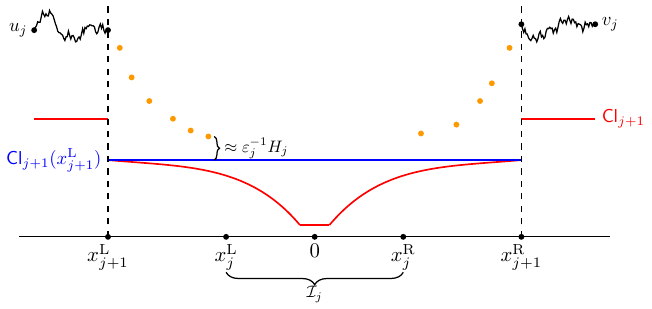}
	\vspace{-0.1in}
	\caption{An illustration of the application of \cref{lem:drop-high-bc} to Proposition \ref{prop:recursive-bound} and the proof of \cref{lem:drop-high-bc}. The floor at $\Cl_{j+1}$ features a potentially large gap at $x_{j+1}^{\mathrm{L}}$ and $x_{j+1}^{\mathrm{R}}$. The first step is to show that, within $[x_{j+1}^{\mathrm{L}},x_{j}^{\mathrm{L}}]$  and $[x_j^{\mathrm{R}} , x_{j+1}^{\mathrm{R}}]$, the random walk drops to $\approx \ep_j^{-1}H_j$ above $\Cl_{j+1}(x_{j+1}^{\mathrm{R}}) = \Cl_{j+1}(x_{j+1}^{\mathrm{L}}) = \max_{x \in \mathcal{I}_{j+1}} \Cl_{j+1}(x)$ (blue).
	The max bound (Lemma~\ref{lem:rec-global-max-bound}) is used to show $X(x_{j+1}^{\mathrm{L}})\vee X(x_{j+1}^{\mathrm{R}})\leq BN^{1/3}$, see \eqref{eqn:max of X for figure}. Lemma \ref{lem:drop-high-bc} then gives random times $x^{\mathrm{L}}, x^{\mathrm{R}} \in \mc{I}_{j+1}\setminus \mc{I}_j$ where $X$ lies roughly $\ep_j^{-1}H_j$ above the floor (the two innermost orange points in the figure). For the proof of \cref{lem:drop-high-bc}, note the boundary conditions  at $BN^{1/3}$ may be as large as $N^{1-\varepsilon}$ for any $\ep>0$, 
	 much higher than the blue line at $\Cl_{j+1}(x_{j+1}^R)$; in particular, the dropping lemma estimate \eqref{eqn:dropping-lemma} is not sufficient to immediately produce $x^{\mathrm{L}}$ and $x^{\mathrm{R}}$.
	Thus, we apply the dropping lemma in an iterative manner to the random walk on $\mathcal I_{j+1}$ with floor at $\mathsf{Cl}_{j+1}(x^{\mathrm{L}}_{j+1})$, yielding a sequence of random drop points at decreasing heights $(A_i H_j)_{i\geq 1}$ (orange, see \cref{claim:drop-high-bc-input}) until $x^{\mathrm{L}}$ and $x^{\mathrm{R}}$ are produced.}
	\label{f.recursion second figure}
\end{figure}

\begin{proof}[Proof of \cref{prop:recursive-bound}]
First, \cref{lem:rec-global-max-bound} with $K$ replaced by $K \ep_j^{-1}$ immediately yields the desired bound on the probability that $X(x) > \Cl_{j}(x)$ for some $x\not\in \mathcal{I}_{j}$.
It therefore suffices to show
\begin{align}\label{eqn:recursive-bound2-goal}
	\P_{\mathcal{I};\Cl_{j+1}}^{\lambda_{j};u_{j}, v_{j}}\Big(\exists x\in \mathcal{I}_{j} : X(x) > \Cl_{j}(x) \Big) \leq Ce^{-cK^{3/2}\log(j+1)}\,.
\end{align}
Note that \cref{lem:rec-global-max-bound} yields
\begin{equation}\label{eqn:max of X for figure}
	\P_{\mathcal{I};\Cl_{j+1}}^{\lambda_{j};u_{j}, v_{j}}\Big( \max\big(X(x_{j+1}^{\mathrm{L}}), X(x_{j+1}^{\mathrm{R}})\big) \leq BN^{1/3} \Big) \geq 1- Ce^{-cjK^{3/2}}\,,
\end{equation}
where we have used also \eqref{eqn:boundary-conditions-recursion2} to upper bound $\max(u_{j},v_{j})+K\ep_j^{-1}H_{j}(\log |\mathcal{I}|/(K^{1/2}H_{j}^2) )^{2/3}$ by $BN^{1/3}$. The Gibbs property (\cref{def:gibbs}) and monotonicity then yield that, with probability exceeding $1- C\exp(-cjK^{3/2})$, we can dominate the restriction to $\mathcal{I}_{j+1}$ of the walk under $\P_{\mathcal{I};\Cl_{j+1}}^{\lambda_{j};u_{j}, v_{j}}$ 
by  the walk on $\mathcal{I}_{j+1}$ under $\smash{\P_{\mathcal{I}_{j+1};\Cl_{j+1}}^{\lambda_{j};BN^{1/3}, BN^{1/3}}}$.
In particular, defining the event
\begin{align*}
	\mathsf{Pt}_j &:= \big\{ \exists x^{\mathrm{L}} \in [x_{j+1}^{\mathrm{L}}, x_{j}^{\mathrm{L}}] ,\, x^{\mathrm{R}}\in [x_{j}^{\mathrm{R}}, x_{j+1}^{\mathrm{R}}] \ : \ 
	X(x^{\mathrm{L}}) \leq w_{j}^{**} \,,\, X(x^{\mathrm{R}}) \leq w_{j}^{**} \big \} \,,
\end{align*}
where $w_{j}^{**} := w_j^* - \tfrac{K}{2C_1}\ep_j^{-1}H_{j}$ (recall $w_j^*$ and $C_1$ from \cref{lem:recursive-bluepart}), we find
\[
	\P_{\mathcal{I};\Cl_{j+1}}^{\lambda_{j};u_{j}, v_{j}}(\neg \mathsf{Pt}_j) \leq \P_{\mathcal{I}_{j+1};\Cl_{j+1}}^{\lambda_{j};BN^{1/3}, BN^{1/3}}(\neg\mathsf{Pt}_j) + Ce^{-cjK^{3/2}}\,.
\]
Now, raising the floor to $\Cl_{j+1}(x_{j+1}^{\mathrm{L}})$ (which equals $\Cl_{j+1}(x_{j+1}^{\mathrm{R}})$ by symmetry of $\Cl_{j+1}$), reducing both the floor and the walk by $\Cl_{j+1}(x_{j+1}^{\mathrm{L}})$ (so that the floor is now at $0$), and then applying \cref{lem:drop-high-bc} with $\ep' := \ep_j = (j+\cC)^{-2}$ (which is $\geq c_0\lambda_j^{-1/6}$ for $c_0$ small enough depending only on $a_0$ and $b_0$), $D := K/2C_1$, 
% (so that $\frac{D}{\ep'} \leq \frac{K}{2C_1}\ep_j^{-1}$), 
$y^{\mathrm{L}} := x_{j+1}^{\mathrm{L}}$, and $y^{\mathrm{R}} := x_{j+1}^{\mathrm{R}}$, we find for $c$ depending only on $a_0$ and $b_0$,
\begin{align}\label{eqn:recursive-bound2-Ec}
	\P_{\mathcal{I};\Cl_{j+1}}^{\lambda_{j};u_{j}, v_{j}}(\neg\mathsf{Pt}_j) \leq Ce^{-cK^{3/2}\lambda_{j}^{2/3}} +Ce^{-cjK^{3/2}} \leq Ce^{-cjK^{3/2}}\,.
\end{align}
Now, on the event $\mathsf{Pt}_j$, let $x^{\mathrm{L}}_*$ denote the leftmost $x^{\mathrm{L}}$ and $x^{\mathrm{R}}_*$ the rightmost $x^{\mathrm{R}}$. Observe $x^{\mathrm L}_j,x^{\mathrm R}_j\in [x^{\mathrm{L}}_*, x^{\mathrm{R}}_*]$.
Applying the strong Gibbs property  to the stopping domain $[x_*^{\mathrm{L}}, x_*^{\mathrm{R}}]$, we find
\begin{equation}\label{eqn:recursivebound2-droppedbc}
	\begin{split}
	\P&_{\mathcal{I};\Cl_{j+1}}^{\lambda_{j};u_{j}, v_{j}}\bigg(\mathsf{Pt}_j \cap \Big\{\max\big(X(x_{j}^{\mathrm{L}}) , X(x_{j}^{\mathrm{R}}) \big)> w_{j}^* \Big\}\bigg)  \\
	&\leq \E_{\mathcal{I};\Cl_{j+1}}^{\lambda_{j};u_{j}, v_{j}}\bigg[\indset{\mathsf{Pt}_j} \ \P_{[x_*^{\mathrm{L}}, x_*^{\mathrm{R}}];\Cl_{j+1}(x_{j+1}^{\mathrm{L}})}^{\lambda_{j}; w_{j}^{**},w_{j}^{**}}\Big(\max\big\{X(x_{j}^{\mathrm{L}}) , X(x_{j}^{\mathrm{R}}) \big\} > w_{j}^*\Big) \bigg] \leq Ce^{-cj^3K^{3/2}}\,.
	\end{split}
\end{equation}
In the first inequality, we used
monotonicity to raise  $X(x^{\mathrm{L}}_*)$ and $X(x^{\mathrm{R}}_*)$ to $w_{j}^{**}$ as well as the floor to $\Cl_{j+1}(x_{j+1}^{\mathrm{L}})$ (since $x_{j+1}^{\mathrm{L}}< x^{\mathrm{L}}_*  < 0$ and $0 <x_*^{\mathrm{R}} \leq x_{j+1}^{\mathrm{R}}$, so $\Cl_{j+1}(x^{\mathrm{L}}_{j+1})$ is indeed an upper bound on $\Cl_{j+1}(\cdot)$ on $[x^{\mathrm{L}}_*,x^{\mathrm{R}}_*]$); in the second inequality, we lowered the floor and the walk by $\Cl_{j+1}(x_{j+1}^{\mathrm{L}})$ inside of the expectation, and then applied \cref{oneptbd} with $R =A = (K/2C_1) \ep_j^{-1} = w_j^*- w_j^{**}$.

 Combining \eqref{eqn:recursive-bound2-Ec}, \eqref{eqn:recursivebound2-droppedbc}, and the Gibbs property, and using monotonicity to raise the boundary conditions on $\mathcal{I}_{j}$ from $X(x_j^{\mathrm{L}})$ and $X(x_j^{\mathrm{R}})$ to $w_j^*$ yields
 \begin{align*}
 	\P&_{\mathcal{I};\Cl_{j+1}}^{\lambda_{j};u_{j}, v_{j}}\Big(\exists x\in \mathcal{I}_{j} : X(x) > \Cl_{j}(x) \Big) \\
	&\leq \E_{\mathcal{I};\Cl_{j+1}}^{\lambda_{j};u_{j}, v_{j}}\Big[\one_{X(x_{j}^{\mathrm{L}}) \vee X(x_{j}^{\mathrm{R}})\leq w_{j}^*} \P_{\mathcal{I}_{j};\Cl_{j+1}}^{\lambda_{j}; X(x_{j}^{\mathrm{L}}), X(x_{j}^{\mathrm{R}})}\Big(\exists x\in \mathcal{I}_{j} : X(x)>\Cl_{j}(x) \Big) \Big] +
 	 Ce^{-cjK^{3/2}} \\
 	&\leq \P_{\mathcal{I}_{j};\Cl_{j+1}}^{\lambda_{j}; w_{j}^*, w_{j}^*}\Big(\exists x\in \mathcal{I}_{j} : X(x)>\Cl_{j}(x) \Big) +Ce^{-cjK^{3/2}}\,.
 \end{align*}
 Applying \cref{lem:recursive-bluepart} to the probability in the last line above yields \eqref{eqn:recursive-bound2-goal}.
\end{proof}

\subsection{Proof of \cref{lem:drop-high-bc}}\label{s.proof of high bc dropping}
The proof of \cref{lem:drop-high-bc} follows by repeatedly applying the next claim. (The parameter $\ep$ below is not to be confused with $\ep$ in Theorems \ref{thm:simplifiedmain} and \ref{thm:simplifiedmain2}.)
\begin{claim}\label{claim:drop-high-bc-input}
Fix any $\nu, \mathfrak{c}>0$. There exists $C_0>0$ such that for all $D\in [C_0,B]$, $\ep \in [ \mathfrak{c} \hspace{0.08em} \lambda_j^{-1/6},1)$, and all $y^{\mathrm{L}}\leq y^{\mathrm{R}}$ such that $y^{\mathrm{R}}-y^{\mathrm{L}} \geq C_0B^{1/2+\nu}N^{2/3}$, the following holds. Define $A_0 := BN^{1/3}$, and for each $i\geq 1$ define $f_i := 2[(3/2)^i-1]$ and $A_i :=  D/\ep +(B^{1-f_{i}{\nu}} \vee 1)$. For all $N$ large, $i\geq 0$, and $j\in \llbracket 1,m+1\rrbracket$, 
\begin{multline}\label{eqn:drop-high-bc-2}
	\P_{[y^{\mathrm{L}},y^{\mathrm{R}}];0}^{\lambda_j; A_i H_j, A_i H_j}\Big(\exists x \in [y^{\mathrm{L}}, y^{\mathrm{L}}+\ep(y^{\mathrm{R}}-y^{\mathrm{L}})] \,,\, x' \in [y^{\mathrm{R}}-\ep(y^{\mathrm{R}}-y^{\mathrm{L}}), y^{\mathrm{R}}] : \\
	\max\big(X(x) , X(x')\big)\leq A_{i+1}H_j \Big) \geq 1- C\exp\Big(-cD^{3/2} \lambda_j^{2/3}\Big)\,.
\end{multline}

\begin{proof}
Recall the  constant $c^*$ from \cref{lem:dropping-lemma}. 
We assume, and use repeatedly, the following:
\begin{align}\label{eqn:interval-length-bigenough}
	|y^{\mathrm{R}}-y^{\mathrm{L}}| \geq  C_0 B^{1/2+\nu}N^{2/3} \quad \text{and} \quad D\geq C_0 \geq 2c^*\,.
\end{align}
We begin with the case $i=0$.
Take $A= A_1$, $I = [y^{\mathrm{L}},y^{\mathrm{R}}]$, and $u=v=BN^{1/3}$ in \cref{lem:dropping-lemma}. In light of \cref{rk:dropping-lemma-condition-1}, 
the assumptions of \cref{lem:dropping-lemma} are satisfied. From \eqref{eqn:interval-length-bigenough}, $\ep > \mathfrak{c}\hspace{0.08em}\lambda_j^{-1/6}$, and $N^{2/3}/H_j^2= \lambda_j^{2/3}$ (from \eqref{e.H_j defn}), we have for any $C_0>0$, 
\[
	A\ep\frac{|I|}{H_j^2} \geq D\frac{|I|}{H_j^2} + \mathfrak{c}\hspace{0.08em}C_0 B^{3/2}\lambda_j^{1/2} \qquad \text{and} \qquad \frac{u^{3/2} + v^{3/2}}{H_j^{3/2}} = 2\Big(\frac{BN^{1/3}}{H_j}\Big)^{3/2} = 2B^{3/2}\lambda_j^{1/2}\,.
\]
Taking $C_0 > 2/\mathfrak{c}$, a union bound, \cref{lem:dropping-lemma}, \cref{eqn:interval-length-bigenough}, and $|I| \geq C_0B^{1/2+\nu}N^{{2/3}} \geq D^{1/2}N^{2/3}$ then bounds the left-hand side of \eqref{eqn:drop-high-bc-2} for $i=0$  from below by
\[
	1- 2c^* \exp\bigg(-(D-c^*)\frac{|I|}{H_j^2}\bigg) \geq  1- 2c^* \exp\bigg(-cD^{3/2} \lambda_j^{3/2}\bigg)\,.
\]

The proof for $i\geq 1$ is almost identical. This time, take $A = A_{i+1}$ and $u=v= A_iH_j$ in \cref{lem:dropping-lemma}. Note that $f_{i+1} - \frac32f_i = 1$. This, \eqref{eqn:interval-length-bigenough}, $B\geq D$, and $\ep\geq \mathfrak{c}\hspace{0.08em}\lambda_j^{-1/6}$ as before yields
\begin{multline*}
	A \ep \frac{|I|}{H_j^2} \geq D^{3/2}B^{\nu} \lambda_j^{2/3}+ \mathfrak{c}\hspace{0.08em}C_0 B^{1/2+\nu}(B^{1-f_{i+1}\nu}\vee 1) \lambda_j^{1/2} \\
	\text{and} \qquad \frac{u^{3/2} + v^{3/2}}{H_j^{3/2}} \leq  C\left[\Big(\frac{D}{\ep}\Big)^{3/2} + \Big(B^{3/2+ (1-f_{i+1})\nu}\vee 1\Big)\right] \leq \frac12 A \ep \frac{|I|}{H_j^2}\,,
\end{multline*}
where the last  inequality comes from the first inequality in the display, the bound $\ep \geq \mathfrak{c}\hspace{0.08em}\lambda_j^{-1/6}$, taking $C_0 \geq C/\mathfrak{c}$, and taking $N$ large (depending on $\nu$, as we want to make $B^{\nu}$ large).
A union bound and \cref{lem:dropping-lemma} then bounds the probability in \eqref{eqn:drop-high-bc-2} for $i\geq 1$ by 
\[
1- 2 c^* \exp\Big( - c(A \ep-c^*) \frac{|I|}{H_j^2}\Big) \geq 1- C\exp \Big(- c D^{3/2} \lambda_j^{2/3} \Big)\,.
\]
The last inequality follows from \eqref{eqn:interval-length-bigenough} (so that $A\ep -c^*\geq D/2$) and from $|I|/H_j^2 \geq D^{1/2} \lambda_j^{2/3}$.
\end{proof}
\end{claim}

\begin{proof}[Proof of \cref{lem:drop-high-bc}]
Fix any $\nu \in (0,\eta)$, and recall  $A_i$ from \cref{claim:drop-high-bc-input}. 
Let $r_{\nu} := \min\{i\in\mathbb{N} : 1-f_i \nu \leq 0\}$, so that $A_{r_{\nu}+i} = D/\ep + 1$ for all $i\geq 0$. Let $\ep :=r_{\nu}^{-1}\ep'$, $\sigma_0 := y^{\mathrm{L}}$, $\tau_0 := y^{\mathrm{R}}$, and recursively define the random times
\begin{align*}
	\sigma_{i+1} &:= \inf\{x\in [\sigma_i, \sigma_i + \ep(\tau_i-\sigma_i)] : X(x) \leq A_{i+1} H_j\} \,,\\
	\tau_{i+1} &:= \sup\{x' \in [\tau_i-\ep(\tau_i-\sigma_i), \tau_i] : X(x') \leq A_{i+1} H_j\} \,.
\end{align*}
If $\sigma_{i+1}$ or $\tau_{i+1}$ do not exist, set $\sigma_{i+1} := \sigma_i + \ep(\tau_i-\sigma_i)$ or $\tau_{i+1} := \tau_i - \ep(\tau_i-\sigma_i)$. 
 Also define the events $\mathsf{Drop}_0 := \Omega$ and
\[
\mathsf{Drop}_{i+1} := \Bigl\{\exists x\in[\sigma_i,\sigma_i+\ep(\tau_i-\sigma_i)], x'\in[\tau_i-\ep(\tau_i-\sigma_i),\tau_i] : \max(X(x),X(x')) \leq A_{i+1}H_j\Bigr\}.
\]
Observe the probability on the left of \eqref{eqn:drop-high-bc-zerofloor} is bounded below by
\begin{align}\label{eqn:drop-prod}
	\mathbb{P}^{\lambda_j; A_0H_j,A_0H_j}_{[y^{\mathrm{L}},y^{\mathrm{R}}];0}(\mathsf{Drop}_{r_\nu}) 
& \geq  \mathbb{P}^{\lambda_j; A_0H_j,A_0H_j}_{[y^{\mathrm{L}},y^{\mathrm{R}}];0}(\mathsf{Drop}_{r_\nu} \given \mathsf{Drop}_{r_{\nu}-1} ) \cdot \mathbb{P}^{\lambda_j; A_0H_j,A_0H_j}_{[y^{\mathrm{L}},y^{\mathrm{R}}];0}(\mathsf{Drop}_{r_{\nu}-1}) 
 \nonumber \\
	& \geq  \prod_{i=0}^{r_\nu-1} \mathbb{P}^{\lambda_j; A_0H_j,A_0H_j}_{[y^{\mathrm{L}},y^{\mathrm{R}}];0}( \mathsf{Drop}_{i+1} \mid \mathsf{Drop}_i).
\end{align}
For the $i$\th factor in the product, we condition further on the $\sigma$-algebra $\mathcal{F}_i$ generated by $\sigma_i$, $\tau_i$, and $X(x)$ for $x\notin(\sigma_i,\tau_i)$. Note that $\mathsf{Drop}_i \in \mathcal{F}_i$, so the tower property, the strong Gibbs property, and monotonicity yield
\begin{align*}
	\mathbb{P}^{\lambda_j; A_0H_j,A_0H_j}_{[y^{\mathrm{L}},y^{\mathrm{R}}];0}( \mathsf{Drop}_{i+1} \cap \mathsf{Drop}_i) 
	&= \mathbb{E}^{\lambda_j; A_0H_j,A_0H_j}_{[y^{\mathrm{L}},y^{\mathrm{R}}];0} \left[ \one_{\mathsf{Drop}_i} \mathbb{P}^{\lambda_j; A_0H_j,A_0H_j}_{[y^{\mathrm{L}},y^{\mathrm{R}}];0} ( \mathsf{Drop}_{i+1} \mid \mathcal{F}_i) \right]\\
	&\geq \mathbb{E}^{\lambda_j; A_0H_j,A_0H_j}_{[y^{\mathrm{L}},y^{\mathrm{R}}];0} \left[ \one_{\mathsf{Drop}_i} \mathbb{P}^{\lambda_j; A_iH_j, A_iH_j}_{[\sigma_i,\tau_i];0} ( \mathsf{Drop}_{i+1}) \right].
\end{align*}
Now by \cref{claim:drop-high-bc-input}, the probability inside the expectation is bounded below by $1-Ce^{-cD^{3/2}\lambda_j^{2/3}}$, uniformly over all possible values of $\sigma_i, \tau_i$. Indeed, the claim applies for each $i$ with $\mathfrak{c} = c_0 r_\nu^{-1}$ (so that $\ep \in [\mathfrak{c}\hspace{0.08em}\lambda_j^{-1/6},1)$), since by assumption, $\tau_i-\sigma_i \geq (1-\ep')(y^{\mathrm{R}}-y^{\mathrm{L}}) \geq (1-\ep')B^{1/2+\eta}N^{2/3}$ for all $i\leq r_\nu$, and for $N$ large depending on $\eta-\nu$ we have $(1-\ep')B^{\eta-\nu} \geq C_0$. Rearranging the last display and applying Bayes' rule, we then see that all of the conditional probabilities on the right of \eqref{eqn:drop-prod} are bounded below by the same expression. Thus the left-hand side of \eqref{eqn:drop-prod} is at most
\[
\Big(1-Ce^{-cD^{3/2}\lambda_j^{2/3}}\Big)^{r_\nu} \geq 1-Cr_\nu \exp\left(-cD^{3/2}\lambda_j^{2/3}\right).
\]
Taking say $\nu = \eta/2$ and absorbing the $r_{\nu}$ factor into $C$, we are done.
\end{proof}

\subsection{Proof of \cref{lem:recursive-bluepart}} \label{s.proof of recursive blue part}
We fix $j\in \llbracket 1, m\rrbracket$ throughout. Fix also  a universal constant
\[
\ep \in (0,1/8]
\]
throughout (not to be confused with $\ep$ from \cref{thm:simplifiedmain,thm:simplifiedmain2}). 
The only properties of the interval $\mathcal{I}_{j}$ we  use below are that its length exceeds $4\mathcal TK^{1/2}\ep_j^{-1/2}H_j^2$ for $\mathcal{T}$ fixed but large depending only on $a_0>0$ and $b_0>1$, and that it is centered at $0$. Recall from \eqref{def:ceilingj2}, \eqref{eq:boundaryconditions-above-floor} that each $\Cl_j$ is symmetric about $0$, non-decreasing in $|x|$, and continuous on $\mathcal{I}_k$ for all $k\leq j$. 

The proof of the lemma proceeds in two steps, which are depicted in Figure~\ref{f.mesh point control}.
In Step 1, we iteratively construct a mesh of appropriately-spaced points at which the walk stays below a fraction of $\Cl_j$ (\cref{claim:meshpoints}).  The mechanism for accomplishing this is the dropping lemma (\cref{lem:dropping-lemma}), which can be used to construct random points $x$ where the walk $X(x)$ lies roughly $\ep_j^{-1}H_j$ above the floor $\Cl_{j+1}(x)$, w.h.p.\ (\cref{claim:meshpoints-dropping}). This way, $X(x)$ falls down as $x$ approaches $0$, just as $\Cl_{j+1}(x)$ does. Since the dropping lemma produces random points, we use the one-point bound and strong Gibbs property to show that at deterministic locations $x(k)$ in between these random points, the walk also stays below a fraction of $\Cl_j$, w.h.p. These $x(k)$ are the aforementioned mesh points.
In Step 2, we use \cref{maxbd} to ``fill in the mesh,'' i.e., upper bound the  walk in between mesh points. The resulting upper bound  is dominated by $\Cl_{j}$, concluding the proof of \cref{lem:recursive-bluepart}.

\subsection*{Step 1: A mesh of good points} The first step is to prove \cref{claim:meshpoints} below. For $C_1 \geq 1$, define 
\[
	D := K/C_1 \,.
\]
Recall $x_j^{\L/\rR}$ from \eqref{def:Ij} (we write $\L/\rR$ to indicate a statement holding for both $\L$ and $\rR$). Define 
\begin{align*}
	x(0) := x_j^\rR \quad \text{(so that $-x(0)  = x_j^{\L}$)} \quad\qquad \text{and} \quad\qquad I(0) :=  [-x(0), x(0)] =\mathcal{I}_j\,,
\end{align*}
and for $k\geq 1$, define
\begin{align}\label{def:xkIk}
	x(k) := x(k-1) - \ep |I(k-1)| \qquad \text{and} \qquad I(k) := [-x(k),x(k)]\,.
\end{align}
Observe for all $k\geq 0$,
\begin{align}\label{eqn:xkIk-size}
	x(k) = \tfrac12 |I(k)| \,,\, \ \ 
	x(k) = (1-2\varepsilon)x(k-1)= (1-2\ep)^k x(0) \,,\, \ \  \text{and} \  |I(k)| = (1-2\ep)^k |I(0)| \,.
\end{align}
Define
\begin{align}\label{def:bar-k}
	\bar{k} := \min \Big\{ k \in\N  :  |I(k)| \leq 4 \mathcal{T}K^{1/2}\ep_j^{-1/2}H_j^2 \Big\}\,.
\end{align}
Combining \eqref{eqn:xkIk-size}, \eqref{def:bar-k}, and $\ep <1/8$, we have 
$x(\bar{k}) = (1-2\ep)^{-1} x(\bar{k}-1) \geq \frac43 \mathcal{T} K^{1/2}{\ep_j^{-1/2}H_j^2}$, so 
\begin{align}\label{eqn:last-mesh-point}
	x(\bar{k}) \in [\mathcal{T}K^{1/2}{\ep_j^{-1/2}H_j^2}, 2\mathcal{T}K^{1/2}{\ep_j^{-1/2}H_j^2}] \qquad \text{and} \qquad |I(\bar{k})| \geq \mathcal{T}K^{1/2}\ep_j^{-1/2}H_j^2\,.
\end{align}
It is helpful to note (see \eqref{def:ceilingj2} for the definition of $\Cl_j$)
\begin{equation}\label{eqn:ceiling-relations}
\begin{split}
	\Cl_j(x) &\geq K \ep_j^{-1}H_j \bigg(\log \frac{|x|}{K^{1/2}H_j^2} \bigg)^{2/3} \quad \text{for all } x \in \mathcal{I}, \\
	\Cl_j\big(x(k)\big) &= K \ep_j^{-1}H_j \bigg(\log \frac{|x(k)|}{K^{1/2}H_j^2} \bigg)^{2/3} \quad \text{for all } k \in \llbracket 1, \bar{k}-1\rrbracket.
\end{split}
\end{equation}
Next, for all $k\geq 1$, define (recalling $D = K/C_1$) 
\begin{align}\label{def:frakB}
	\fB_j(k) &=  \fB_j(k, C_1) := D\ep_j^{-1}H_j + \Cl_{j+1}\big(x(k-1)\big) + D\ep_j^{-1}H_j \Big(\log \tfrac{x(k)}{D^{1/2}H_j^2} \Big)^{2/3}\,.
\end{align}	
\cref{claim:meshpoints} states that the walk with floor $\Cl_{j+1}$ is bounded by $\fB_j(k)$ at each $x(k)$. Going back to the proof outline given at the start of the subsection (c.f., Figure~\ref{f.mesh point control}), the first two terms in \eqref{def:frakB} will come from applying dropping lemma over the floor raised to $\Cl_{j+1}(x(k-1))$ (recall $x(k-1) > x(k)$ and $\Cl_{j+1}(x)$ is non-decreasing in $|x|$); the last term will follow from applying \cref{oneptbd} in between the dropped points. The following inequality shows that $\fB_j(k)$ is a fraction of $\Cl_j(x(k))$: 
for any $C_1 > 1$, we may take $\mathcal{T}$ large enough (depending on $a_0, b_0, C_1$) so that for all $k \in \llbracket 1, \bar{k}\rrbracket$, 
\begin{align}\label{eqn:frakB-bound-barClj}
	\fB_j(k) &= D\ep_j^{-1}H_j \bigg[1+ C_1 \frac{\ep_{j+1}^{-1}H_{j+1}}{\ep_j^{-1}H_j}\bigg(\log\frac{x(k-1)}{K^{1/2}H_{j+1}^2}\bigg)^{2/3} + \bigg(\log\frac{x(k)}{D^{1/2}H_{j}^2}\bigg)^{2/3}  \Bigg]  \nonumber \\
	&\leq  D\ep_j^{-1}H_j \bigg[1+ C_1 b^{-1/6}
	\bigg(\log\frac{x(k)}{K^{1/2}H_{j}^2}+ \log \frac{b^{2/3}}{1-2\ep} \bigg)^{2/3} + \bigg(\log\frac{x(k)}{K^{1/2}H_{j}^2} + \log C_1^{1/2}\bigg)^{2/3}  \Bigg] \nonumber \\
	&\leq b_0^{-2/15} K\ep_j^{-1} H_j \bigg(\log \frac{x(k)}{K^{1/2}H_j^2} \bigg)^{2/3}
	\leq b_0^{-2/15} \Cl_j\big( x(k)\big)\,.
\end{align}
In the equality, we used \eqref{eqn:ceiling-relations}.
In the first inequality, we used \eqref{def:cC}, \eqref{e.H_j defn} to replace $H_{j+1}^2$ in the denominator with $\smash{H_j^2}$, and \eqref{def:xkIk}. In the second inequality, we bounded the second bracketed term by using $\smash{x(k)/(K^{1/2}H_j^2) \geq \mathcal{T}a_0^{2/3}}$, so that taking $\mathcal{T}$ large enough with respect to $a_0$ and $b_0$ yields 
\[
b^{-1/4} \log\frac{x(k)}{K^{1/2}H_{j}^2}+ b^{-1/4}\log \frac{b^{2/3}}{1-2\ep} \leq b^{-2/9} \log\frac{x(k)}{K^{1/2}H_{j}^2} \left(1+\frac{\tilde C}{\log (\mathcal T a_0^{2/3})}\right) \leq b_0^{-1/5} \log\frac{x(k)}{K^{1/2}H_{j}^2} \,,
\]
where $\tilde C:= \max_{z>1} z^{-1/4+2/9}\log \frac{z^{2/3}}{1-2\varepsilon}$ is an absolute constant.
We bounded the third bracketed term similarly, taking $\mathcal{T}$ large with respect to $C_1$. In the third inequality, we used \eqref{eqn:ceiling-relations}.
We are now ready to state \cref{claim:meshpoints}, the main result of ``Step 1.''

\begin{figure}
\includegraphics[scale=1.3]{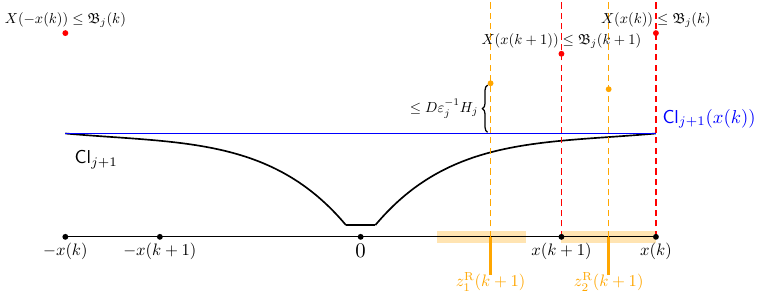}
\vspace{-0.1in}
\caption{Step 1 involves bounding the walk $X$ at a sequence of mesh points $x(0) > x(1) > \cdots$ and $-x(0) < -x(1) < \cdots$ by a quantity $\mathfrak{B}_j(k)$ at $\pm x(k)$.
Suppose we have already shown $X(-x(k)) \vee X(x(k)) \leq \mathfrak{B}_j(k)$, for some $k$ (outermost red dots).
Restrict the walk to $I(k) = [-x(k), x(k)]$ and raise the floor from $\Cl_{j+1}$ (black curve) to $\Cl_{j+1}(x(k))$ (blue).
Use the dropping lemma to produce points $z_1^{\mathrm{R}}(k+1)$ and $z_2^{\mathrm{R}}(k+1)$ to the left and right of $x(k+1)$, where the walk has dropped to $\leq D\ep_j^{-1}H_j$ above the blue floor (orange dots).
Restrict the walk to $[z_2^{\mathrm{R}}(k+1), z_1^{\mathrm{R}}(k+1)]$ (orange, dashed), raise the boundary conditions to $D\ep_j^{-1}H_j + \Cl_{j+1}(x(k))$, and use the one-point bound (\cref{oneptbd}) to show $X(x(k+1))\leq \mathfrak{B}_j(k+1)$ with high probability.
Step 2 of the proof involves restricting the walk to $[x(k+1), x(k)]$ (red, dashed), raising the boundary conditions to $\mathfrak{B}_j(k)$ (which exceeds $\mathfrak{B}_j(k+1)$, red dots), and then using the max bound \cref{maxbd} to show that, w.h.p., $\max_{x \in [x(k+1), x(k)]} X(x)$ is bounded by a quantity which is less than $\Cl_j(x(k+1))= \min_{x \in [x(k+1), x(k)]}\Cl_j(x)$.
By symmetry, the same argument applies on the left side of $0$.}\label{f.mesh point control}
\end{figure}

\begin{claim} \label{claim:meshpoints}
Under the assumptions of \cref{lem:recursive-bluepart},
the following holds for any $C_1 \geq 1$:
\[
	\P_{\mathcal{I}_{j};\Cl_{j+1}}^{\lambda_{j};w_{j}^*, w_{j}^*}\Big(\exists k \in \llbracket 1, \bar{k} \rrbracket : \max\big(X(-x(k)), X(x(k))\big) > \mathfrak{B}_j(k) \Big) \leq C \exp\Big(-cK^{3/2} \log (j+1) \Big) \,.
\]
\end{claim}

\begin{proof} 
The proof proceeds via the following inductive scheme. For $k\geq 0$, suppose we have already shown $X(-x(k)) \vee X(x(k)) \leq \fB_j(k)$ with high probability. Recall Figure~\ref{f.mesh point control}, and define the event 
\begin{align}\label{def:Pt}
	\mathsf{Drop}(k+1) := \Big\{\exists &z_1^\L \in -x(k) + [0,\ep]|I(k)| \,,\,  z_2^\L \in -x(k) + [2\ep,3\ep]|I(k)| \,,\, \nonumber \\
	&z_1^\rR \in x(k) - [2\ep, 3\ep]|I(k)| \,,\, z_2^\rR \in x(k) - [0,\ep]|I(k)| : \nonumber \\
	&\max_{i \in \{1,2\}, \, \theta \in \{\L,\rR\}} X(z_i^\theta) \leq D\ep_j^{-1}H_j + \Cl_{j+1}\big(x(k)\big) \Big\}\,.
\end{align}

\begin{claim}\label{claim:meshpoints-dropping}
Recall  $w_j^*$ and $I(k)$ from \eqref{def:wj*} and \eqref{def:xkIk}. Define $\fB_j(0) := w_j^*$. For all $k\in\llbracket 0, \overline{k}-1\rrbracket$, 
\begin{align}
	\P_{I(k); \Cl_{j+1}}^{\lambda_j; \fB_j(k), \B_j(k)}\Big( \mathsf{Drop}(k+1) \Big) 
	&\geq 1- C \exp \bigg(-cKj^2 \frac{|I(k)|}{H_j^2} \bigg) \text{ for all }k \in \llbracket 1, \overline{k}-1\rrbracket. \label{eqn:meshpoints-dropping-k}
\end{align}
\begin{proof}[Proof of \cref{claim:meshpoints-dropping}]
Recalling  $(1-\ep_j)x_{j+1}^{\mathrm{R}} = x_j^{\rR}= x(0)$ from \eqref{def:Ij},
\begin{align*}
	w_j^* - \Cl_{j+1}\big(x(0)\big) &= D \ep_j^{-1} H_j + K\ep_{j+1}^{-1} H_{j+1}\bigg[\bigg(\log \frac{x(0)}{K^{1/2}H_{j+1}^2} + \log \frac{1}{1-\ep_j} \bigg)^{2/3} -\bigg(\log \frac{x(0)}{K^{1/2}H_{j+1}^2} \bigg)^{2/3} \bigg] \\
	&\leq 2K\ep_j^{-1}H_j =: \tilde{w}_j\,,
\end{align*}
where in the  inequality we used \eqref{def:cC}, as well as \eqref{eqn:last-mesh-point} to deduce $x(0)/K^{1/2}H_{j+1}^2 \geq \mathcal{T}$ and took $\mathcal{T}$ large compared to $1/(1-\ep_j) \leq 2$. Raising the floor  to $\Cl_{j+1}(x(0))$, lowering the walk and floor by $\Cl_{j+1}(x(0))$, a  union bound, and \cref{lem:dropping-lemma} with $A:= D\ep_j^{-1}$ and $u= v:= w_j^*$ give
		\begin{align}
			\P_{I(0);\Cl_{j+1}}^{\lambda_{j}; w_{j}^*;w_{j}^*}\big(\mathsf{Drop}(1)\big) 
			&\geq \P_{I(0);\Cl_{j+1}(x(0))}^{\lambda_{j}; w_{j}^*;w_{j}^*}\big(\mathsf{Drop}(1)\big) \nonumber  \\ 
			&\geq 1- 4c^*\exp\bigg( -\Big(D\ep_j^{-1}\ep -c^*\Big) \frac{|I(0)|}{H_{j}^2} + 2c^*\cdot 2^{3/2}K^{3/2}\ep_j^{-3/2} \bigg)\,. \label{eqn:mesh-drop1-last}
		\end{align}
		Take $K_0$ large so that $D\ep_j^{-1}\ep -c^* > cK\ep_j^{-1}$ for $K\geq K_0$ and some $c>0$. \cref{eqn:last-mesh-point} gives $|I(0)|/H_j^2 \geq \mathcal{T}K^{1/2}\ep_j^{-1/2}$, so the first term in the exponential dominates the second term for $\mathcal{T}$ large with respect to $c$ and $c^*$. This yields \eqref{eqn:meshpoints-dropping-k} for $k=0$.

		The proof for $k\in \llbracket 1, \bar{k}-1\rrbracket$ 
		is similar: we use monotonicity to raise the floor on $I(k)$ to $\Cl_{j+1}(x(k))$, and then lower the walk and the floor by $\Cl_{j+1}(x(k))$. A union bound and \cref{lem:dropping-lemma} again yield
		\begin{align*}
			\P_{I(k); \Cl_{j+1}}^{\lambda_j; \fB_j(k), \B_j(k)}&\Big( \mathsf{Drop}(k+1) \Big) \\
			\geq& \ 1- 4c^* \exp \bigg(-\Big(D\ep_j^{-1}\ep - c^*\Big) \frac{|I(k)|}{H_j^2} + 2c^*\bigg(\frac{\fB_j(k) - \Cl_{j+1}(x(k))}{H_j}\bigg)^{3/2} \bigg)\,.
		\end{align*}
		We use $\fB_j(k) - \Cl_{j+1}(x(k))\leq \fB_j(k) \leq \Cl_j(x(k))$ by \eqref{eqn:frakB-bound-barClj}, so that 
		the second term in the exponential is at most $C K^{3/2}\ep_j^{-3/2} \log (x(k)/(K^{1/2}H_j^2))$. Using $x(k) = |I(k)|/2$ and $\log (p/2)\leq p$ for $p\geq 1$, we see that the second term in the exponential is dominated by the first term for $\mathcal{T}$ large, just as below \eqref{eqn:mesh-drop1-last}.
		This yields \eqref{eqn:meshpoints-dropping-k}.
\end{proof}
\end{claim}

On $\mathsf{Drop}(k+1)$, let $z_1^{\L/\rR}(k+1)$ be the leftmost such $z_1^{\L/\rR}$; similarly, let $z_2^{\L/\rR}(k+1)$ be the rightmost such $z_2^{\L/\rR}$. Note  $[z_1^{\L/\rR}(k+1), z_2^{\L/\rR}(k+1)]$ each define stopping domains. 
Recall
$\fB_j(0) := w_j^*$. Below, we will have $X(\pm x(0)) = \fB_j(0)$ deterministically from the boundary conditions. 
For each $k\in \llbracket 0, \bar{k}-1 \rrbracket$, a union bound, symmetry of the random walk bridge, monotonicity in the boundary conditions, and the strong Gibbs property on $[z_1^{\rR}(i+1), z_2^{\rR}(i+1)]$ yield
\begin{align*}
	\P&_{\mathcal{I}_j ; \Cl_{j+1}}^{\lambda_j; w_j^*, w_j^*}\Big(X\big(\!-x(k+1)\big) \vee X\big(x(k+1)\big) > \fB_j(k+1) \Big) \\
	&\leq 
	\sum_{i=0}^k  \P_{\mathcal{I}_j ; \Cl_{j+1}}^{\lambda_j; w_j^*, w_j^*}\Big(X\big(\!-x(i+1)\big) \vee X\big(x(i+1)\big) > \fB_j(i+1) \,,\, \mathsf{Drop}(i+1) \Big) \\
	&\qquad\qquad+\P_{\mathcal{I}_j ; \Cl_{j+1}}^{\lambda_j; w_j^*, w_j^*}\Big( \neg \mathsf{Drop}(i+1) \,,\, X\big(\!-x(i)\big) \vee X\big(x(i)\big)\leq \fB_j(i)\Big)  \\
	&\leq \sum_{i=0}^k  2\cdot \E_{\mathcal{I}_j ; \Cl_{j+1}}^{\lambda_j; w_j^*, w_j^*}\left[\P_{[z_1^\rR(i+1), z_2^\rR(i+1)]; \Cl_{j+1}}^{\lambda_j; D\ep_j^{-1}H_j + \Cl_{j+1}(x(i)); D\ep_j^{-1}H_j + \Cl_{j+1}(x(i))}\Big(X\big(x(i+1)\big) > \fB_j(i+1) \Big)\right] \\
	&\qquad\qquad+ \P_{I(i); \Cl_{j+1}}^{\lambda_j; \fB_j(i), \fB_j(i)}\Big(\neg\mathsf{Drop}(i+1)\Big) \,.
\end{align*}
For each $i$, the inner probability in the first term in the $i\th$ summand above can be bounded by raising $\Cl_{j+1}$ to $\Cl_{j+1}(x(i))$ (since $x(i) \geq z_2^\rR(i+1)$), lowering both the floor and the walk by $\Cl_{j+1}(x(i))$, and then applying \cref{oneptbd} with  $A= D\ep_j^{-1}$ and $I:= [z_1^\rR(i+1), z_2^\rR(i+1)]$ (which satisfies $|I| \geq \ep|I(k)|$ by construction, so $|I| \geq H_j^2$ by \eqref{eqn:last-mesh-point}). Altogether, this shows 
\begin{multline*}
	2\cdot \E_{\mathcal{I}_j ; \Cl_{j+1}}^{\lambda_j; w_j^*, w_j^*}\left[\P_{[z_1^\rR(i+1), z_2^\rR(i+1)]; 0}^{\lambda_j; D\ep_j^{-1}H_j, D\ep_j^{-1}H_j}\bigg(X\big(x(i+1)\big) > D\ep_j^{-1}H_j+D\ep_j^{-1}H_j\bigg(\log \frac{x(i+1)}{D^{1/2}H_j^2} \bigg)^{2/3} \bigg)\right]\\ 
	\leq C \exp\bigg( - c K^{3/2} \log \frac{x(i+1)}{D^{1/2}H_j^2}\bigg)\,.
\end{multline*}
Combined with \cref{claim:meshpoints-dropping} and \eqref{eqn:xkIk-size}, the previous two displays yield
\begin{multline*}
	\P_{\mathcal{I}_j ; \Cl_{j+1}}^{\lambda_j; w_j^*, w_j^*}\Big(X\big(\!-x(k+1)\big) \vee X\big(x(k+1)\big) > \fB_j(k+1) \Big)\\
	\leq C \sum_{i=0}^k \exp\bigg( - c K^{3/2} \log \frac{x(i+1)}{D^{1/2}H_j^2}\bigg) \leq C\exp\bigg( - c K^{3/2} \log \frac{x(k+1)}{D^{1/2}H_j^2}\bigg).
\end{multline*}
A union bound over $k\in \llbracket 0,\bar{k}-1 \rrbracket$ and again recalling \eqref{eqn:xkIk-size} then yields
\[
	\P_{\mathcal{I}_{j};\Cl_{j+1}}^{\lambda_{j};w_{j}^*, w_{j}^*}\Big(\exists k \in \llbracket 1, \bar{k} \rrbracket : \max\big(X(-x(k)), X(x(k))\big) > \mathfrak{B}_j(k) \Big) \leq C \exp\bigg( - c K^{3/2} \log \frac{x(\bar{k})}{D^{1/2}H_j^2}\bigg) \,.
\]
\cref{claim:meshpoints} now follows from the last display and \eqref{eqn:last-mesh-point}.
\end{proof}

\subsection*{Step 2: Filling in the mesh}
After \cref{claim:meshpoints} and the upper bound \eqref{eqn:frakB-bound-barClj}, we have constructed a ceiling on the walk at a mesh of points $\{x(k)\}_{k =0}^{\bar{k}}$. 
The max bound \cref{maxbd} will allow us to ``fill in'' the ceiling between the mesh points, completing the proof of \cref{lem:recursive-bluepart}.

\begin{proof}[Proof of \cref{lem:recursive-bluepart}]
A union bound, the Gibbs property, and monotonicity in the boundary values yield the following:
\begin{equation}\label{eqn:filling-in-mesh-1}
\begin{split}
	&\P_{\mathcal{I}_{j};\Cl_{j+1}}^{\lambda_{j};w_{j}^*, w_{j}^*}\Big(\exists x \in \mathcal{I}_j : X(x) > \Cl_j(x) \,,\, \forall k \in \llbracket 1, \bar{k}\rrbracket : X\big(\!-x(k)\big) \vee X\big(x(k)\big) \leq \fB_j(k)  \Big) \\
	&\quad \leq \ \sum_{k=0}^{\bar{k}-1} 
	\bigg[ \P_{[x(k+1), x(k)]; \Cl_{j+1}}^{\lambda_j ; \fB_j(k+1) , \fB_j(k)}\Big( \exists x \in [x(k+1), x(k)] : X(x) > \Cl_j(x) \Big) \\
	&\qquad\quad + \P_{[-x(k), -x(k+1)] ; \Cl_{j+1}}^{\lambda_j; \fB_j(k) , \fB_j(k+1)} \Big(\exists x \in [-x(k), -x(k+1)] : X(x) > \Cl_j(x) \Big)\bigg] \\
	&\qquad\quad + \P_{I(\bar{k}) ; \Cl_{j+1}}^{\lambda_j; \fB_j(\bar{k}) , \fB_j(\bar{k})}\Big(\exists x \in [-x(\bar{k}), x(\bar{k})] : X(x) > \Cl_j(x) \Big)   =: \sum_{k=0}^{\bar{k}-1} \big[ (\mathrm{I})_k + (\mathrm{II})_k \big] + (\mathrm{III})_{\overline{k}}.
\end{split}
\end{equation}
Note by symmetry,
\begin{align}\label{eqn:I=II}
	(\mathrm{I})_k = (\mathrm{II})_k  \qquad \text{for each } k \in \llbracket0, \bar{k}-1\rrbracket \,.
\end{align}
We proceed by bounding $(\mathrm{I})_k$. 
Recall $\fB_j(k) \geq \fB_j(k+1)$ and $\Cl_{j+1}$ is increasing in $|x|$. By monotonicity, we raise both boundary conditions to $\fB_j(k)$ and raise the floor to $\Cl_{j+1}(x(k))$; decreasing both the floor and the walk by $\Cl_{j+1}(x(k))$ and using $\min_{x\in[x(k+1), x(k)]} \Cl_j(x) = \Cl_j(x(k+1))$ yields
\begin{align}\label{eqn:Term1-1}
	(\mathrm{I})_k &\leq \P_{[x(k+1), x(k)]; 0}^{\lambda_j; , \fB_j(k) - \Cl_{j+1}(x(k))}\Big(\exists x \in [x(k+1), x(k)] : X(x) > \Cl_j(x) - \Cl_{j+1}\big(x(k)\big) \Big) \\
	&\leq \P_{[x(k+1), x(k)]; 0}^{\lambda_j; \fB_j(k) - \Cl_{j+1}(x(k)), \fB_j(k) - \Cl_{j+1}(x(k))}\bigg(\max_{x \in [x(k+1), x(k)]} X(x) > \Cl_j\big(x(k+1)\big) - \Cl_{j+1}\big(x(k)\big) \bigg)\,.\nonumber
\end{align}
Recalling \eqref{eqn:ceiling-relations}, $k \leq \bar{k}-1$, $x(k+1) = (1-2\ep)x(k)$, and \eqref{eqn:last-mesh-point}, we can take $\mathcal{T}$ large so that
\begin{align*}
	\Cl_j\big(x(k+1)\big) &\geq K\ep_j^{-1} H_j \bigg(\log \frac{x(k)}{K^{1/2}H_j^2} +\log (1-2\ep) \bigg)^{2/3} \\
	&\geq b_0^{-1/8}K\ep_j^{-1} H_j \bigg(\log \frac{x(k)}{K^{1/2}H_j^2}\bigg)^{2/3} = b_0^{-1/8} \Cl_j\big(x(k)\big)\,.
\end{align*}
From \eqref{eqn:frakB-bound-barClj} and the above, we have the following for $\mathcal{T}$ and $C_1$ sufficiently large: 
\[
	\Cl_j\big(x(k+1)\big) - \fB_j(k) \geq (b_0^{-1/8} - b_0^{-2/15}) \Cl_j(x(k)) \geq D H_j \bigg(\log \frac{x(k)-x(k+1)}{K^{1/2}H_j^2}\bigg)^{2/3}\,.
\]
Substituting into \eqref{eqn:Term1-1} and applying \cref{maxbd} with $I = [x(k+1), x(k)]$ and $R= D$, we find
\begin{align}\label{eqn:Term1-final}
	(\mathrm{I})_k \leq C\exp\bigg(-cD^{3/2} \log \frac{x(k)-x(k+1)}{K^{1/2}H_j^2} \bigg) \leq C\exp\bigg(-cK^{3/2} \log \frac{|I(k)|}{K^{1/2}H_j^2} \bigg)\,.
\end{align}
The term $(\mathrm{III})_{\overline{k}}$ is bounded similarly. Following the steps leading to \eqref{eqn:Term1-1}, raise the floor to $\Cl_{j+1}(x(\bar{k}))$, then lower the walk and the floor by $\Cl_{j+1}(x(\bar{k}))$:
\begin{align*}
	(\mathrm{III})_{\overline{k}} \leq \P_{I(\bar{k});0}^{\lambda_j ; \fB_j(\bar{k}) - \Cl_{j+1}(x(\bar{k})), \fB_j(\bar{k}) - \Cl_{j+1}(x(\bar{k}))} \bigg(\max_{x\in I(\bar{k})} X(x) > \Cl_j\big(x(\bar{k})\big) - \Cl_{j+1}\big( x(\bar{k})\big) \bigg)\,.
\end{align*}
From \eqref{eqn:frakB-bound-barClj},  \eqref{eqn:ceiling-relations}, and $x(\bar{k}) = |I(\bar{k})|/2$, we find the following for $\mathcal{T}$ and $C_1$ large depending on $b_0$: 
\begin{multline*}
	\Cl_j\big(x(\bar{k})\big) - \fB_j(\bar{k}) 
	\geq (1- b_0^{-2/15}) \Cl_j\big(x(\bar{k})\big)
	\\
	\geq (1- b_0^{-2/15}) K \ep_j^{-1} H_j  \bigg(\log \frac{x(\bar{k})}{K^{1/2}H_j^2} \bigg)^{2/3} 
	\geq D H_j \bigg(\log \frac{|I(\bar{k})|}{D^{1/2}H_j^2} \bigg)^{2/3}\,.
\end{multline*}
The last two displays and \cref{maxbd} with $I = I(\bar{k})$ and $R= D$ yield
\begin{align}\label{eqn:Term3}
	(\mathrm{III})_{\overline{k}} \leq C \exp\bigg(-cK^{3/2} \log \frac{|I(\bar{k})|}{K^{1/2}H_j^2} \bigg)\,.
\end{align}
Substituting \eqref{eqn:I=II}, \eqref{eqn:Term1-final}, and \eqref{eqn:Term3} into \eqref{eqn:filling-in-mesh-1}, we find 
\begin{multline*}
	\P_{\mathcal{I}_{j};\Cl_{j+1}}^{\lambda_{j};w_{j}^*, w_{j}^*}\Big(\exists x \in \mathcal{I}_j : X(x) > \Cl_j(x) \,,\, \forall k \in \llbracket 1, \bar{k}\rrbracket : X\big(\!-x(k)\big) \vee X\big(x(k)\big) \leq \fB_j(k)  \Big) \\
	\leq C \sum_{k=0}^{\bar{k}} \exp\bigg(-cK^{3/2} \log \frac{|I(k)|}{K^{1/2}H_j^2} \bigg) 
	\leq C\exp\bigg(-cK^{3/2} \log \frac{|I(\bar{k})|}{K^{1/2}H_j^2} \bigg) \leq C \exp\big(-cK^{3/2} \log \ep_j^{-1/2} \big)\,,
\end{multline*}
where the last inequality follows from \eqref{eqn:last-mesh-point}.
\cref{lem:recursive-bluepart} follows from the above and \cref{claim:meshpoints}.
\end{proof}

\subsection{Proof of \cref{prop:max-general-bottom-curves}}
\label{subsec:pf-bottom-curves}
The proof follows a recursive strategy as in the proof of \cref{thm:max-general}: we start from the bottom curve, remove all curves above it (see \cref{rmk:remove-top}), produce a ceiling on this curve with high probability, and then iterate upwards using this ceiling as the floor for the next curve. However, these curves experience extremely large area tilts for which our results do not apply (e.g., \cref{maxbd,lem:dropping-lemma}, see \cref{rk:dropping-lemma-condition-1}). We begin by decreasing the tilt parameters from $(\lambda_{m+1},b)$ to $(\lambda_{m+1},1)$ so that all curves experience the same area tilt. 
Next, we raise both boundary conditions from $\u^{>m}$ and $\v^{>m}$ to $\mathbf{w} = (w_1, \dots, w_{n-m})$, where
\[
	w_j := \max(u_{m+1}, v_{m+1}) + z_j, \qquad z_j := 2n(n-m-j) KH_{m+1} \log \frac{|\mathcal{I}|}{K^{1/2}H_{m+1}^2} \,. 
\]
Altogether, we have
\begin{equation}\label{eq:bottom-curves-mono}
	\P_{n-m,\mathcal{I};0}^{\lambda_{m+1},b;\u^{>m},\v^{>m}}
	\Big(\exists x \in \cI : X_{1}(x) >\mathsf{F}(x) \Big)
	\leq 
	\P_{n-m,\mathcal{I};0}^{\lambda_{m+1},1;\mathbf{w}, \mathbf{w}}\Big(\exists x \in \cI : X_{1}(x) >\mathsf{F}(x) \Big).
\end{equation}
We define, for $j\in \llbracket 1, n-m\rrbracket$, the intervals
\[
 	\widetilde{\mathcal{I}}_j = [y_j^{\mathrm{L}}, y_j^{\mathrm{R}}] := \left(1-\frac{1}{2n}\right)^{n-m-j+1}\, \mathcal{I} 
\]
and the piecewise constant functions $\mathsf{F}_j$ on $\mathcal{I}$ given by
 \begin{equation}\label{eq:Fjdef}
 	\mathsf{F}_{j} := z_{j-1} + \max(u_{m+1},v_{m+1}) \one_{\mathcal{I}\setminus \widetilde{\mathcal{I}}_j}, \qquad j\in\llbracket 1,n-m\rrbracket,
 \end{equation}
and $\mathsf{F}_{n-m+1} := 0$. Note that each $\mathsf{F}_j$ is increasing in $|x|$, 
and we have trivially
  \begin{align}\label{eqn:lowcurves-floor-diff}
 	\mathsf{F}_j(x) - \mathsf{F}_{j+1}(x) \geq KH_{m+1}\log \frac{|\mathcal{I}|}{K^{1/2}H_{m+1}^2} \text{ for all } x\in\mathcal{I}.
 \end{align}
Observe in particular that $\mathsf{F}_1 \leq \mathsf{F}$ for large $N$. This is clear from the definitions inside of $\widetilde{\mathcal{I}}_1 \cap \mathcal{I}_{m+1}$ and outside of $\widetilde{\mathcal{I}}_1 \cup \mathcal{I}_{m+1}$. Since $(1-\frac{1}{2n})^{n-m} \geq (1-\frac{1}{2n})^n \to e^{-1/2} > 1/2$ as $N\to\infty$, we in fact have $\widetilde{\mathcal{I}}_1 \supset \frac{1}{2}\mathcal{I} = \mathcal{I}_{m+1}$ for large $N$ (depending only on $r$, as $n = \lfloor N^\delta \rfloor$ where $\delta = \delta(r)$), which verifies that $\mathsf{F}_1 \leq \mathsf{F}$ on all of $\mathcal{I}$. It follows from \eqref{eq:bottom-curves-mono} and an iterative Gibbs argument with monotonicity, similar to the one leading to \eqref{eqn:ceil-sum}, that
\begin{align}\label{eqn:bottom-curves-recursion}
	\P_{n-m,\mathcal{I};0}^{\lambda_{m+1},b;\u^{>m},\v^{>m}}
	\Big(\exists x \in \cI : X_{1}(x) >\mathsf{F}(x) \Big)
	\leq 
	\sum_{j=1}^{n-m} \P_{\mathcal{I}; \mathsf{F}_{j+1}}^{\lambda_{m+1}; w_j, w_j} \Big( \exists x \in \mathcal{I} : X(x) > \mathsf{F}_j(x)\Big).
\end{align}
Proposition \ref{prop:max-general-bottom-curves} follows from \eqref{eqn:bottom-curves-recursion} and \cref{claim:bottom-curves-recursion} below, which plays the role of \cref{prop:recursive-bound}. \qed

 \begin{claim}\label{claim:bottom-curves-recursion}
 There exist $K_0, N_0>0$ such that under the assumptions of \cref{prop:max-general-bottom-curves}, for $j\in\intint{1,n-m}$,
 \begin{align*}
 	\P_{\mathcal{I}; \mathsf{F}_{j+1}}^{\lambda_{m+1}; w_j, w_j}\Big( \exists x \in \mathcal{I} : X(x) > \mathsf{F}_{j}(x) \Big) \leq Ce^{-c(m+1)^{3/2}K^{3/2}}\,.
 \end{align*}
 
 \begin{proof}
 The argument follows the proof of \cref{prop:recursive-bound}; however, after applying \cref{lem:drop-high-bc}, all that needs to be done is a max bound. Indeed, as in the proof of \cref{lem:rec-global-max-bound}, \cref{maxbd} and \eqref{eqn:lowcurves-floor-diff} yield $X(x)\leq \mathsf{F}_j(x)$ for all $x \in\mathcal{I}\setminus \widetilde{\mathcal{I}}_j$ with probability exceeding 
 \[
 	1-Ce^{-cK^{3/2} (\log (|\mathcal{I}|/K^{1/2}H_{m+1}^2))^{3/2}} \geq 1- Ce^{-c(m+1)^{3/2}K^{3/2} }\,.
 \]
If the minimum in the definition of $\mathsf{F}_j$ on $\widetilde{\mathcal{I}}_j$ in \eqref{eq:Fjdef} is $w_{j-1}$, then we are done.

Otherwise, one may now apply Lemma \ref{lem:drop-high-bc} with $y^{\mathrm{L}} :=y_{j+1}^{\mathrm{L}}$, $y^{\mathrm{R}}:= y_{j+1}^{\mathrm{R}}$, $j:=m+1$, $\ep' := 1/n$ (which from \eqref{def:m} satisfies the condition $\ep'>c_0 \lambda_{m+1}^{-1/6}$ for $\delta$ and $c_0$ sufficiently small), $D=K$, and boundary conditions $w_{j-1}$,  
 to produce a leftmost random point
 $
 	x^{\mathrm{L}} \in (\widetilde{\mathcal{I}}_{j+1} \setminus \widetilde{\mathcal{I}}_j) \cap (-\infty,0)$ and a rightmost 
 	$x^{\mathrm{R}} \in (\widetilde{\mathcal{I}}_{j+1} \setminus \widetilde{\mathcal{I}}_j) \cap (0,\infty)
 $
 such that with probability exceeding $1- Ce^{-c\lambda_{m+1}^{2/3}K^{3/2}}$,  
 \begin{align*}
 	X(x^{\mathrm{L}/\mathrm{R}}) &\leq  nKH_{m+1} + n(n-m-j)KH_{m+1} \bigg(\log\frac{|\mathcal{I}|}{K^{1/2}H_{m+1}^2}\bigg)^{2/3} =: w_j'.
 \end{align*}
(Although \cref{lem:drop-high-bc} is for walks with floor at $0$, note $\mathsf{F}_{j+1}$ is constant on $\widetilde{\mathcal{I}}_j$, and we can shift the floor and the walk down by this value to apply the lemma.) For large $N = N(b_0,r)$, \eqref{def:m} gives  $\lambda_{m+1}^{2/3} > (m+1)^{3/2}$, so this error probability can be absorbed into the previous one. The proof is finished by applying the strong Gibbs property on  $[x^{\mathrm{L}}, x^{\mathrm{R}}]$, monotonicity, and \cref{maxbd} for the walk with law $\P_{[x^{\mathrm{L}}, x^{\mathrm{R}}]; 0}^{\lambda_{m+1};w_j', w_j'}$,  noting $w_j' + KH_{m+1}\log(|\mathcal{I}|/K^{1/2}H_{m+1}^2) \leq \mathsf{F}_j(x)$ for all $x\in\widetilde{\mathcal{I}}_j$.
 \end{proof}
 \end{claim}

%!TEX root = ./rwareatilt.tex

\section{Tightness and Gibbs property}\label{sec:tight}

In this section we prove Propositions \ref{prop:tight} and \ref{prop:gibbs}. The large majority of the work is required to prove tightness, which will occupy the first three subsections below. The proof strategy is similar to that of numerous previous works on Gibbsian line ensembles, in particular \cite{ach24}. In Section \ref{limitgibbs}, we briefly prove the Gibbs property for subsequential limits; this is the only point where the scaling factors of $\sigma^{-2/3}$ in \eqref{def:x_i} are relevant.

\subsection{Tightness} In this subsection and the two following, we assume that the walk increments have variance $\sigma^2 = 1$. This causes no loss in generality, as clearly rescaling by a constant does not affect tightness of the process. Thus we will use the notation $\mathbf{x}^N = (x_1^N,\dots,x_n^N)$,
\[
x_i^N(t) := N^{-1/3} X_i(tN^{2/3}), 
\]
where the law of the line ensemble $\mathbf{X} = (X_1,\dots,X_n)$ will be clear from context. (The domain of definition and number of curves may change throughout.) We will be working with the usual law $\mathbb{P}_{n,N;0}^{a,b;\mathbf{u},\mathbf{v}}$, with $n,\mathbf{u},\mathbf{v}$ as in the statement of Theorem \ref{thm:main}. As the parameters in this law are fixed throughout, for brevity we will just write 
\[
\mathbb{P}_N := \mathbb{P}_{n,N;0}^{a,b;\mathbf{u},\mathbf{v}},
\]
and $\mathbb{E}_N$ for the corresponding expectation.

The proof of tightness will proceed by a standard Gibbs argument using the following lower bound on the partition function.

\begin{proposition}\label{Zlbd}
	Let $X \sim \mathbb{P}_N$. Then for any $\varepsilon>0$, there exists $\eta>0$ such that
	\[
	\mathbb{P}_N \left( Z_{k,TN^{2/3}; X_{k+1}}^{a,b;\mathbf{z},\mathbf{w}} > \eta \right) > 1 - \varepsilon,
	\]
	where $\mathbf{z} = (X_1(-TN^{2/3}), \dots, X_k(-TN^{2/3}))$, $\mathbf{w} = (X_1(TN^{2/3}), \dots, X_k(TN^{2/3}))$, and in the subscript $X_{k+1}$ is implicitly restricted to $[-TN^{2/3},TN^{2/3}]$.
\end{proposition}

With this result, we are equipped to prove Proposition~\ref{prop:tight} on tightness.

\begin{proof}[Proof of Proposition \ref{prop:tight}] 
	
	We first show that the sequence $\{\mu_N\}_{N\geq 1}$ of laws of $\mathbf{x}^N$, where $X\sim\mathbb{P}_N$, is tight. For $k\in\mathbb{N}$, $T>0$, a function $\mathbf{f} = (f_1,\dots,f_k) : \llbracket 1,k\rrbracket \times[-T,T] \to \mathbb{R}$, and $\delta>0$, define the modulus of continuity
	\[
	w(\mathbf{f},\delta) := \sup_{\substack{s,t\in[-T,T]\\|s-t|<\delta}} \max_{1\leq i\leq k} |f_i(s)-f_i(t)|.
	\]
	By a standard Arzel\`a--Ascoli type argument, see e.g. \cite[Lemma 2.4]{dff}, (and the fact that $X_n$ is nonnegative), it suffices to prove the following two statements:
	\begin{enumerate}
		\item For all $\varepsilon>0$, there exists $M>0$ so that
		\[
		\limsup_{N\to\infty} \mathbb{P}_N \left( x_1^N(0) > M \right) < \varepsilon.
		\]
		
		\item For all $k\in\mathbb{N}$, $T,\varepsilon,\rho>0$, there exists $\delta>0$ so that
		\[
		\limsup_{N\to\infty} \mathbb{P}_N \left(w(\mathbf{x}^N|_{\llbracket 1,k\rrbracket\times[-T,T]}, \delta) > \varepsilon\right) < \rho.
		\]
	\end{enumerate}
	
	Statement (1) follows immediately from Theorem \ref{thm:max}. In the remainder we prove (2). We adopt the notation of Proposition \ref{Zlbd}. We will use the shorthand $\mathbb{P}_N^{k,T}$ for $\mathbb{P}_{k,TN^{2/3};X_{k+1}}^{a,b;\mathbf{z},\mathbf{w}}$, and $Z_N^{k,T}$ for the corresponding partition function. For $\delta,\varepsilon,\eta,M>0$, define the events
	\[
	\mathsf{BadMod}(\delta,\varepsilon) := \left\{w(\mathbf{x}^N|_{\llbracket 1,k\rrbracket\times[-T,T]},\delta) > \varepsilon\right\},  \qquad \mathsf{Fav}(\eta,M):= \left\{Z_N^{k,T} > \eta\right\} \cap \left\{z_1,w_1 \leq MN^{1/3} \right\}.
	\]
	By Proposition \ref{Zlbd} and Theorem \ref{thm:max}, we can choose $\eta,M$ so that
	\begin{equation}\label{Favlbd}
		\limsup_{N\to\infty} \mathbb{P}_{n,N;0}^{a,b;\mathbf{u},\mathbf{v}} (\mathsf{Fav}(\eta,M)) > 1-\rho/2.	\end{equation} 
		Let $\mathcal{F}_{k,T}$ denote the $\sigma$-algebra generated by $\mathbf{z}$, $\mathbf{w}$, and $X_{k+1}^N|_{[-TN^{2/3},TN^{2/3}]}$. By the Gibbs property,
	\begin{equation}
		\begin{split}\label{BadModGibbs}
			\mathbb{P}_N \left( \mathsf{BadMod}(\delta,\varepsilon) \cap \mathsf{Fav}(\eta,M) \right) &= \mathbb{E}_N \left[ \one_{\mathsf{Fav}(\eta,M)}\, \mathbb{E}_N \left[\one_{\mathsf{BadMod}(\delta,\varepsilon)} \mid \mathcal{F}_{k,T} \right] \right]\\
			&= \mathbb{E}_N \left[ \one_{\mathsf{Fav}(\eta,M)}\, \mathbb{P}_N^{k,T} \left(\mathsf{BadMod}(\delta,\varepsilon) \right) \right].
		\end{split}
	\end{equation}
	 On the event $\{Z_N^{k,T} > \eta\} \supset \mathsf{Fav}(\eta,M)$, we have
	\begin{equation}\label{BadModprod}
		\mathbb{P}_N^{k,T} \left(\mathsf{BadMod}(\delta,\varepsilon)\right) \leq \eta^{-1} \prod_{i=1}^k \mathbb{P}_{TN^{2/3}}^{z_i,w_i}\left(w(y^N_i,\delta) > \varepsilon\right),
	\end{equation}
	where $Y_i \sim \mathbb{P}^{z_i,w_i}_{TN^{2/3}}$ are random walk bridges with no area tilt or floor, and $y_i^N(t) = N^{-1/3}Y_i(tN^{2/3})$. We claim that uniformly over the event $\{z_1,w_1 \leq MN^{1/3}\} \supset \mathsf{Fav}(\eta,M)$, each factor on the right of \eqref{BadModprod} can be made less than $\nu := (\eta\rho/2)^{1/k}$ by choosing $\delta$ small and $N$ large enough. Combining this with \eqref{BadModGibbs}, \eqref{BadModprod}, and a union bound with \eqref{Favlbd} implies (2). After shifting vertically and replacing $M$ with $2M$, we can assume $z_i=0$. The proof is then concluded by Lemma \ref{bridgemoc} below.
\end{proof}

\begin{lemma}\label{bridgemoc}
	For any $\nu,\varepsilon,T,M>0$, we can find $N_0\in\mathbb{N}$ and $\delta>0$ such that for all $N\geq N_0$,
	\begin{equation}\label{mocsup}
		\sup_{|z| \leq MN^{1/3}} \mathbb{P}^{0,z}_{TN^{2/3}} \left(w(y^N,\delta) \geq \varepsilon\right) < \nu.
	\end{equation}
\end{lemma}

\begin{proof}
	
	For each $z\in[-M,M]$, let $B^z$ denote a Brownian bridge on $[-T,T]$ from 0 to $z$. We can choose $\delta>0$ so that $\mathbb{P}(w(B^z,\delta) \geq \varepsilon) < \nu/2$ for all $z\in[-M,M]$. To see this, note that $B^z$ is equal in law to $B^0 + \ell^z$, where $\ell^z$ is the line segment from $(-T,0)$ to $(T,z)$, and $w(B^0+\ell^z,\delta) \leq w(B^0,\delta) + \frac{\delta M}{2T}$. Then take $\delta < \frac{\varepsilon T}{M}$ small enough so that $\mathbb{P}(w(B^0,\delta) \geq  \varepsilon/2) < \nu/2$. This is possible because $B^0$ is a.s. uniformly continuous, so that $w(B^0,\delta)\to 0$ a.s. and hence in probability as $\delta\downarrow 0$.
	
	Next, for each fixed $N$, choose $z_N\in [-MN^{1/3},MN^{1/3}]$ so that $P_N := \mathbb{P}^{0,z_N}_{TN^{2/3}} (w(y^N,\delta) \geq \varepsilon)$ is within $\nu/2$ of the supremum on the left-hand side of \eqref{mocsup}. Then it suffices to show that any subsequential limit as $N\to\infty$ of the sequence $\{P_N\}_{N\geq 1}$ is less than $\nu/2$. Indeed, if the statement of the lemma were to fail then we could find a subsequence of $\{P_N\}$ with all elements at least $\nu/2$, and as the sequence is bounded by 1 this would imply a further subsequential limit of at least $\nu/2$.
	
	Suppose $P$ is a subsequential limit of $\{P_N\}$. Since $z_N N^{-1/3} \in [-M,M]$, we can pass to a further subsequence $\{N_\ell\}_{\ell\geq 1}$ such that $\{z_{N_\ell}N_\ell^{-1/3}\}_{\ell\geq 1}$ converges as $\ell\to\infty$ to some $z\in[-M,M]$, and still $P_{N_\ell} \to P$. Then by the invariance principle Lemma \ref{l.invar}, the law of $y^{N_\ell}$ under $\mathbb{P}^{0,z_{N_\ell}}_{TN_\ell^{2/3}}$ converges as $\ell\to\infty$ to that of $B^z$ above. It is easy to see that the set $\{f \in C([-T,T]) : w(f,\delta) \geq \varepsilon\}$ is closed in the uniform topology, so the portmanteau theorem implies
	\[
	P = \limsup_{\ell\to\infty} P_{N_\ell} \leq \mathbb{P} \left(w(B^z,\delta) \geq \varepsilon\right) < \nu/2.
	\]
	Since $P$ was arbitrary, we are done.	
\end{proof}

\subsection{Partition function lower bound}

In this section we prove Proposition \ref{Zlbd}. The proof strategy follows that of \cite[Section 11]{ach24}, although we modify some of the arguments to avoid using KMT coupling, which would require stronger assumptions on our walk distributions. Throughout this subsection the area tilt parameters $a$ and $b$ will be fixed, so we will omit them from the notation.

We first introduce a family of auxiliary measures to be used in the proof, a trick known as ``size-biasing.'' Fix $k\in\mathbb{N}$, $T>0$, an interval $I = [\ell N^{2/3},rN^{2/3}]$ with $[-\frac{3}{2}TN^{2/3}, \frac{3}{2}TN^{2/3}] \subset I \subset [-2TN^{2/3}, 2TN^{2/3}]$, $\mathbf{x},\mathbf{y}\in\mathbb{A}_k^+$, and a function $h : I \to \mathbb{R}^+$. Let $I_T := \mathbb{Z}\cap I \setminus [-TN^{2/3}, TN^{2/3}]$. We define a measure $\widetilde{\mathbb{P}}_{k,I;h}^{\mathbf{x},\mathbf{y}}$ by taking $k$ independent random walk bridges $\widetilde{X}_1,\dots,\widetilde{X}_k$ on $I$ from $\mathbf{x}$ to $\mathbf{y}$, with Radon--Nikodym derivative proportional to
\[
\exp\left(-\frac{a}{N}\sum_{i=1}^k b^{i-1} \sum_{j\in I_T} \widetilde{X}_i(j) \right) \one_{\widetilde{X}_1(j) > \cdots > \widetilde{X}_k(j) > h(j) \mbox{ for } j\in I_T}.
\]
This is the usual measure but with the interaction and area tilt ``turned off'' on $[-TN^{2/3},TN^{2/3}]$. It is straightforward to observe that the Radon--Nikodym derivative between these two measures takes the form
\begin{equation}\label{RNderiv}
	\frac{\mathrm{d}\mathbb{P}_{k,I;h}^{\mathbf{x},\mathbf{y}}}{\mathrm{d}\widetilde{\mathbb{P}}_{k,I;h}^{\mathbf{x},\mathbf{y}}} = \frac{Z_{k,TN^{2/3};h}^{\widetilde{X}(-TN^{2/3}), \widetilde{X}(TN^{2/3})} }{\widetilde{\mathbb{E}}_{k,I;h}^{\mathbf{x},\mathbf{y}} \big[ Z_{k,TN^{2/3};h}^{\widetilde{X}(-TN^{2/3}), \widetilde{X}(TN^{2/3})} \big]}.
\end{equation}
If it can be shown that the denominator on the righthand side of \eqref{RNderiv} is uniformly positive for all $N$, then the identity \eqref{RNderiv} effectively ``biases'' the measure $\mathbb{P}_{k,I;h}^{\mathbf{x},\mathbf{y}}$ away from events where the partition function is very small, as the Radon--Nikodym derivative is small on such events. Combined with a simple Gibbs argument, this will be sufficient to prove the high-probability lower bound on $Z_N^{k,T}$ of Proposition \ref{Zlbd}.

The following technical lemma, an analogue of \cite[Lemma 11.1]{ach24}, will provide the necessary lower bound on the normalization in \eqref{RNderiv} under some natural separation conditions between neighboring curves and the floor $h$ at the boundary. The proof is illustrated in Figure \ref{fig:corr}.

%Namely, this lemma says that if the $k$ curves have some positive separation from each other and $h$ at $\ell N^{2/3}$ and $rN^{2/3}$, and if on scale $N^{1/3}$ the floor $h$ lies below line segments of some finite slope $P$ near these points and is bounded on $I$, then the partition function on the smaller interval $[-TN^{2/3}, TN^{2/3}]$ will be bounded away from 0 at least with some positive probability under the modified measures introduced above. The proof is illustrated in Figure \ref{fig:corr}.

\begin{lemma}\label{Ztechlbd}
	Fix $k,T,I$ as above, and $P,M,\eta > 0$. Suppose $h : I\to\mathbb{R}$ satisfies
	\begin{align*}
		h(tN^{2/3}) &\leq h(\ell N^{2/3}) + (t-\ell)PN^{1/3} \qquad \mathrm{for} \qquad t \in [\ell,\ell+1],\\
		h(tN^{2/3}) &\leq h(rN^{2/3}) + (r-t)PN^{1/3} \qquad \mathrm{for} \qquad t \in [r-1,r].
	\end{align*}
	Assume $\sup_I h\leq MN^{1/3}$, and $\mathbf{x},\mathbf{y}\in W^n_0$ are such that $\min(x_i - x_{i+1}, y_i-y_{i+1}) \geq \eta N^{1/3}$ for $1\leq i\leq k$ where $x_{k+1} = h(\ell N^{2/3})$, $y_{k+1} = h(rN^{2/3})$. Then there exist $\delta,\varepsilon>0$ and $N_0\in\mathbb{N}$ such that for all $N\geq N_0$,
	\begin{equation}\label{Ztildelbd}
		\widetilde{\mathbb{P}}_{k,I;h}^{\mathbf{x},\mathbf{y}} \left(Z_{k,TN^{2/3};h}^{\widetilde{X}(-TN^{2/3}), \widetilde{X}(TN^{2/3})} > \delta \right) > \varepsilon.
	\end{equation}
\end{lemma}

\begin{figure}
	\includegraphics[scale=2.25]{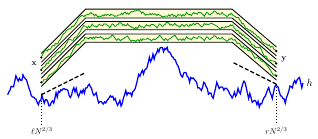}
	\vspace{-0.1 in}
	\caption{An illustration of the proof of Lemma \ref{Ztechlbd}. The two dashed lines at $\ell N^{2/3}$ and $rN^{2/3}$ have slope $PN^{1/3}$ and $-PN^{1/3}$ respectively. The rectangular corridors have width $\eta N^{1/3}$, and $\mathsf{Corr}(\eta)$ is the event that the curves remain within these corridors. The partition function is bounded below on this event. Assuming the slope condition and a global bound on the bottom curve, $\mathsf{Corr}(\eta)$ occurs with positive probability by Brownian estimates.}\label{fig:corr}
\end{figure}

\begin{proof}
	The proof is essentially the same as in \cite{ach24}. For $1\leq i\leq k$ define the piecewise linear functions $\mathrm{Corr}_i : [\ell, r] \to \mathbb{R}$ to have $\mathrm{Corr}_i(\ell) = x_i$, $\mathrm{Corr}_i(r) = y_i$, slope $Q_i := (\max(2M,P)+2(k-i))N^{1/3}$ on $[a,a+1]$ and $-Q$ on $[b-1,b]$, and constant on $[a+1,b-1]$. Let $\mathsf{Corr}(\eta)$ denote the event that $|X_i(tN^{2/3}) - \mathrm{Corr}_i(t)| \leq \frac{1}{2}\eta N^{1/3}$ for all $t\in[\ell,r]$. Note in particular this implies that the curves $X_i$ are uniformly pairwise separated by at least $\frac{1}{2}\eta N^{1/3}$, hence are non-intersecting. It is not hard to see that $\widetilde{\mathbb{P}}_{k,I;h}^{a,b;\mathbf{u},\mathbf{v}}(\mathsf{Corr}(\eta)) > \varepsilon > 0$ for all $N$. The normalization constant for this measure is clearly less than 1, so it can be ignored for a lower bound. Note that on $\mathsf{Corr}(\eta)$, the top curve $X_1$ is uniformly bounded by $Q_1 N^{2/3} = (\max(2M,P) + 2k-2)N^{1/3}$, so the area tilt is bounded below on this event by $\exp(-ab^k\cdot 2T(\max(2M,P)+2k-2))$. It remains to lower bound the probability of independent random walk bridges satisfying the conditions of $\mathsf{Corr}(\eta)$, and this follows from the invariance principle and straightforward Brownian tube estimates. 
	
	It is also easy to see that the partition function in \eqref{Ztildelbd} is bounded below by some $\delta>0$ on the event $\mathsf{Corr}(\eta)$. As already noted, the area tilt is bounded below on this event. The conditions above imply that at $\pm TN^{2/3}$ the curves $X_i$ are pairwise separated by at least $N^{1/3}$, and again by the invariance principle there is a positive probability that $k$ independent random walk bridges with such boundary conditions stay in horizontal tubes of width $\frac{1}{2}N^{1/3}$. This yields a $\delta>0$ in \eqref{Ztildelbd}. 
\end{proof}

The next lemma will establish that the conditions in Lemma \ref{Ztechlbd}  hold for some random interval $I$ with high probability. Fix $k\in\mathbb{N}$ and $P>0$, and define $U \in [-2T, -\frac{3}{2}T]$ to be the first time and $V \in [\frac{3}{2}T, 2T]$ the last time such that
\begin{equation}\label{UVdef}
	\begin{split}
		X_{k+1}(tN^{2/3}) &\leq X_{k+1}(UN^{2/3}) + (t-U)PN^{1/3} \qquad \mbox{for} \quad t \in [U,U+1],\\
		X_{k+1}(tN^{2/3}) &\leq X_{k+1}(VN^{2/3}) + (V-t)PN^{1/3} \qquad \mbox{for} \quad t \in [V-1,V].
	\end{split}
\end{equation}
If such a $U$ does not exist we set it to $-2T$, and if such a $V$ does not exist we set it to $2T$. Note the interval $[UN^{2/3},VN^{2/3}]$ defines a stopping domain. For $\eta>0$ define
\[
\mathsf{Sep}_k(\eta) := \left\{\min_{t\in\{U,V\}} \min_{1\leq i\leq k-1} (X_i(tN^{2/3}) - X_{i+1}(tN^{2/3})) > \eta N^{1/3} \right\}.
\]

\begin{lemma}\label{UVsep}
	Fix $k\in\mathbb{N}$. For any $\varepsilon>0$, there exist $\eta,P>0$ such that
	\[
	\mathbb{P}_N \left(\{-2T<U<V<2T\} \cap \mathsf{Sep}_k(\eta)\right) > 1-\varepsilon/4.
	\]
\end{lemma}

Assuming this lemma, we can prove the high-probability lower bound on the partition function.

\begin{proof}[Proof of Proposition \ref{Zlbd}]
	For $M>0$, define the event
	\[
	\mathsf{Fav}_k(\eta,M) := \{-2T < U < V < 2T\} \cap \mathsf{Sep}_k(\eta) \cap \left\{\sup_{t\in [-2T,2T]} X_{k+1}(tN^{2/3}) \leq MN^{1/3}\right\}.
	\]
	Combining Lemma \ref{UVsep} with Theorem \ref{thm:max} shows that for suitable $\eta,P,M$ we have
	\begin{equation}\label{favkem}
		\mathbb{P}_N \left(\mathsf{Fav}_k(\eta,M)\right) > 1 - \varepsilon/2.
	\end{equation}
	Define the interval $J:= [UN^{2/3},VN^{2/3}]$, and let $\mathcal{F}_{k,J}$ denote the $\sigma$-algebra generated by $X_1,\dots,X_k$ outside of $J$ and by $X_j$ for $j>k$. Write $\mathbf{x} := X(UN^{2/3})$, $\mathbf{y} := X(VN^{2/3})$. By the strong Gibbs property and \eqref{favkem},
	\begin{equation}\label{Zetagibbs}
		\begin{split}
			\mathbb{P}_N \left(Z_N^{k,T} \leq \eta\right) &\leq \mathbb{E}_N \left[ \one_{\mathsf{Fav}_k(\eta,M)}  \mathbb{P}_N \left( Z_N^{k,T} \leq \eta \mid \mathcal{F}_{k,J}\right)\right] + \varepsilon/2\\
			&= \mathbb{E}_N \left[ \one_{\mathsf{Fav}_k(\eta,M)}  \mathbb{P}_{k,J;X_{k+1}}^{a,b;\mathbf{x},\mathbf{y}} \left( Z_N^{k,T} \leq \eta \right)\right] + \varepsilon/2.
		\end{split}
	\end{equation}
	Now recall the measures $\widetilde{\mathbb{P}}_{k,J;X_{k+1}}^{\mathbf{x},\mathbf{y}}$ defined at the beginning of this section. With $\widetilde{X} \sim \widetilde{\mathbb{P}}_{k,J;X_{k+1}}^{\mathbf{x},\mathbf{y}}$, we adopt the shorthand $\widetilde{Z}_N^{k,T} := Z_{k,TN^{2/3};h}^{\widetilde{X}(-TN^{2/3}),\widetilde{X}(TN^{2/3})}$. Using \eqref{RNderiv}, the probability inside the last expectation can be rewritten as
	\begin{equation}\label{Zetafrac}
		\mathbb{P}_{k,J;X_{k+1}}^{\mathbf{x},\mathbf{y}} \left( \widetilde{Z}_N^{k,T} \leq \eta \right) = \frac{\widetilde{\mathbb{E}}_{k,J;X_{k+1}}^{\mathbf{x},\mathbf{y}} \big[ \one_{\widetilde{Z}_N^{k,T} \leq \eta} \widetilde{Z}_N^{k,T} \big]}{\widetilde{\mathbb{E}}_{k,J;X_{k+1}}^{\mathbf{x},\mathbf{y}} [ \widetilde{Z}_N^{k,T} ]}.
	\end{equation}
	On the event $\mathsf{Fav}_k(\eta,M)$, the conditions in Lemma \ref{Ztechlbd} are satisfied with $\ell=U$, $r=V$, $h=X_{k+1}$, which implies that there exists $\rho>0$ (namely $\rho = \delta\varepsilon$ from the statement of the lemma) such that
	\[
	\widetilde{\mathbb{E}}_{k,J;X_{k+1}}^{\mathbf{x},\mathbf{y}} [ \widetilde{Z}_N^{k,T} ] \geq \rho.
	\]
	Inserting this into \eqref{Zetafrac} gives an upper bound of $\rho^{-1}\eta$, and so choosing $\eta = \frac{1}{2}\rho\varepsilon$ and combining with \eqref{Zetagibbs} completes the proof.
\end{proof}

\subsection{Proof of Lemma \ref{UVsep}}

We follow a similar strategy to that in \cite[Section 11]{ach24}. The proof of separation will proceed by induction on $k$. We will use the fact, proven in the previous subsection, that if the statement of Lemma \ref{Ztechlbd} is known for index $k$, then the partition function lower bound Proposition \ref{Zlbd} for index $k$ follows. The existence of the random times $U$ and $V$ will follow from a deterministic statement about continuous functions, Lemma 10.1 in \cite{ach24}, which we record here for reference.

\begin{lemma}\label{lem:ach}\cite[Lemma 10.1]{ach24}
	Let $T>0$, $\delta\in(0,\frac{1}{30}T)$, and $f : [0,T] \to\mathbb{R}$ continuous. Suppose $M>0$ is such that $\sup_{x,y\in[0,T]} |f(x)-f(y)| \leq M$. Then with $P = 3MT^{-1}$, there exists $U \in [0,\frac{1}{2}T]$ such that
	\[
	f(x) \leq f(U) + P(x-U) \quad \mathrm{for}\quad x\in [U,U+\delta].
	\]
\end{lemma}

\begin{figure}
	\includegraphics[scale=1.75]{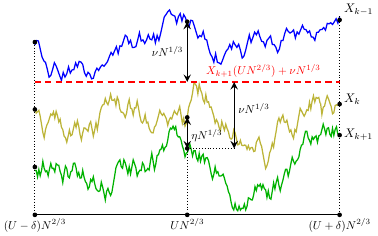}
	\vspace{-0.1in}
	\caption{An illustration of the proof of Lemma \ref{UVsep}. We have conditioned on the blue ceiling $X_{k-1}$, the green floor $X_{k+1}$, and the boundary conditions of the yellow curve $X_k$. The proof amounts to arguing that for a sufficiently small $\eta>0$, at time $UN^{2/3}$ the yellow curve will remain at least some positive distance $\eta N^{1/3}$ above the green floor. On the high-probability event $\mathsf{Fav}_k(\nu,\delta,M)$, the blue curve lies uniformly above the red dashed line at distance $\nu N^{1/3}$ above $X_{k+1}(UN^{2/3})$. Monotonicity allows us to decrease the blue ceiling to the red dashed line, decrease the green floor to height $-N^{1/3}$ on the whole interval except at $UN^{2/3}$, and decrease the boundary conditions of $X_k$ to 0.}\label{fig:sep}
\end{figure}

\begin{proof}[Proof of Lemma \ref{UVsep}] See Figure \ref{fig:sep} for an illustration of the proof. The existence of the random times $U, V$ satisfying the conditions of \eqref{UVdef} (for any fixed $k\geq 0$ and $T>30$, which causes no loss of generality) follows from Lemma \ref{lem:ach}. Namely, we apply this lemma with the random continuous function $f = x_{k+1}^N$. By Theorem \ref{thm:max}, there exists $M>0$ such that with probability at least $1-\varepsilon/2$ we have $\sup_{s,t\in[-2T,2T]} |x_{k+1}^N(s)-x_{k+1}^N(t)| \leq M$. The lemma then guarantees, for $P = 3MT^{-1}$, that
	\begin{equation}\label{UVexist}
		\mathbb{P}_N(-2T < U < V < 2T) \geq 1- \varepsilon/2.
	\end{equation}
	
	We will now show by induction on $k$ that the event $\mathsf{Sep}_k(\eta)$ holds for some $\eta>0$ with probability at least $1-\varepsilon/2$. In fact, we will only use the measurability of the times $U, V$ with respect to $\sigma(X_j : j \geq k+1)$, so this argument will prove high probability separation at any such time.
	
	We may take the base case to be $k=0$, in which case $\mathsf{Sep}_0(\eta)$ holds trivially. Assume for some $k\geq 1$ that we have a $\nu>0$ such that for large $N$,
	\begin{equation}\label{Sepk-12e}
		\mathbb{P}_N \left(\mathsf{Sep}_{k-1}(2\nu)\right) \geq 1-\varepsilon/4.
	\end{equation}
	Then it suffices to prove separation between $X_k$ and $X_{k+1}$ at $U$ with probability at least $1-\varepsilon/8$; the proof of separation at $V$ is analogous. Namely, for $j<k$ and $\eta>0$ define the event
	\[
	\mathsf{Sep}_j^k(\eta) := \left\{X_j(UN^{2/3}) - X_k(UN^{2/3}) > \eta N^{1/3} \right\}.
	\]
	We will show that there exists $\eta>0$ such that for large enough $N$,
	\begin{equation}\label{Sepk,k+1}
		\mathbb{P}_N \left(\neg \mathsf{Sep}_{k+1}^k(\eta)\right) < \varepsilon/8.
	\end{equation}
	For $\delta,\rho>0$, define the event
	\[
	\mathsf{Cont}_{k-1}(\rho,\delta) := \left\{\sup_{|t-U|<\delta} \left(X_{k-1}(tN^{2/3}) - X_{k-1}(UN^{2/3})\right) > -\rho N^{1/3}\right\}.
	\]
	By the inductive hypothesis, the partition function lower bound of Proposition \ref{Zlbd} holds for index $k-1$. Therefore, we can bound below the probability of $\mathsf{Cont}_{k-1}(\rho,\delta)$ for any $\rho$ using the same argument as in \eqref{BadModprod} with Lemma \ref{bridgemoc}. In addition, the inductive hypothesis \eqref{Sepk-12e} implies a lower bound on $\mathsf{Sep}_{k+1}^{k-1}(2\nu)$. Thus, defining
	\[
	\mathsf{Fav}_k(\nu,\delta, M) := \mathsf{Cont}_{k-1}(\nu,\delta) \cap \mathsf{Sep}_{k+1}^{k-1}(2\nu) \cap \left\{\sup_{t\in[-2T,2T]} X_{k-1}(tN^{2/3}) \leq MN^{1/3}\right\},
	\]
	we can find $\nu,\delta,M>0$ so that
	\begin{equation}\label{contk-1}
		\limsup_{N\to\infty} \mathbb{P}_N\left(\mathsf{Fav}_k(\nu,\delta,M)\right) \geq 1- \varepsilon/16.
	\end{equation}
	We will work on this event. Define $\mathcal{F}_{k,U,\delta}$ to be the $\sigma$-algebra generated by $X_j$ for $j\neq k$, and by $X_k$ outside of the interval $I_{U,\delta} := ((U-\delta)N^{2/3}, (U+\delta)N^{2/3})$. By the strong Gibbs property,
	\begin{equation}\label{sepfav}
		\begin{split}
			&\mathbb{P}_N\left(\neg \mathsf{Sep}_{k+1}^k(\eta) \cap \mathsf{Fav}_k(\nu,\delta,M)\right)  = \mathbb{E}_N \left[ \one_{\mathsf{Fav}_k(\nu,\delta,M)} \mathbb{E}_N \left[\one_{\neg\mathsf{Sep}_{k+1}^k(\eta)} \mid \mathcal{F}_{k,U,\delta} \right] \right]\\
			&\qquad\qquad\qquad\qquad\qquad\qquad = \mathbb{E}_N \left[\one_{\mathsf{Fav}_k(\nu,\delta,M)}  \mathbb{P}_{I_{U,\delta};X_{k-1},X_{k+1}}^{ab^{k-1}; w,z} \left( y^N(U) - x_{k+1}^N(U) \leq \eta \right) \right],
		\end{split}
	\end{equation}
	where $w := X_k((U-\delta)N^{2/3})$, $z := X_k((U+\delta)N^{2/3})$, and $y^N(t) := N^{-1/3}Y(tN^{2/3})$ for $Y\sim \mathbb{P}_{I_{U,\delta};X_{k-1},X_{k+1}}^{ab^{k-1}; w,z}$, the single-curve law with floor $X_{k+1}$ and ceiling $X_{k-1}$ as defined in Section \ref{s.monotonicity statement}.
	
	We now upper bound the probability inside the last expectation using monotonicity. Note that if we decrease $w$ and $z$ to height 0, and replace the floor $X_{k+1}$ with a floor at $-N^{1/3}$ except at the point $UN^{2/3}$ where we leave the floor at height $X_{k+1}(UN^{2/3})$, then the above probability increases. In order to prove \eqref{contk-1}, it suffices to show that any subsequential limit of the left-hand side satisfies the desired lower bound. Since $x_{k+1}^N(U) \leq M$ on the event $\mathsf{Fav}_k(\nu,\delta,M)$, we can pass to a further subsequence and assume that $x_{k+1}^N(U) \to x \in [0,M]$ as $N\to\infty$. Furthermore, on $\mathsf{Fav}_k(\nu,\delta,M)$ we have $x_{k-1}^N(U) > x_{k+1}^N(U) + 2\nu$ and $x_{k-1}^N(t) > x_{k-1}(U) - \nu$ for all $t\in [U-\delta,U+\delta]$, which implies that $x_{k-1}(t) > x_{k+1}^N(U) + \nu$ for all $t\in[U-\delta,U+\delta]$. Thus applying monotonicity again, we can replace the ceiling of $x_{k-1}^N$ in the last line of \eqref{sepcond} with a ceiling of $x_{k+1}^N(U)+\nu$. Altogether, the probability inside the expectation in the last line of \eqref{sepfav} is bounded above by
	\begin{equation}\label{sepcond}
		\begin{split}
			&\mathbb{P}_{I_{U,\delta};-N^{1/3}}^{ab^{k-1};0,0} \left( y^N(U) - x_{k+1}^N(U) \leq \eta \, \bigg| \, y^N(U) > x_{k+1}^N(U),\; \sup_{I_{U,\delta}} y^N < x_{k-1}^N(U)+\nu \right)\\
			&\qquad\qquad\qquad\qquad \leq \frac{\mathbb{P}_{I_{U,\delta};-N^{1/3}}^{ab^{k-1};0,0} \left(  x_{k+1}^N(U) < y^N(U) \leq x_{k+1}^N(U) + \eta \right)}{\mathbb{P}_{I_{U,\delta};-N^{1/3}}^{ab^{k-1};0,0} \left(y^N(U) > x_{k+1}^N(U), \; \sup_{I_{U,\delta}} y^N < x_{k-1}^N(U)+\nu \right)}.
		\end{split}
	\end{equation}
	We now apply the invariance principle to bound the last expression. Recall from Definition \ref{def:brownian} that $\mathbf{P}_{\delta;-1}^{ab^{k-1};0,0}$ denotes the law of a Brownian bridge $B$ on $[-\delta,\delta]$ from 0 to 0 with floor at $-1$ and area tilt $\exp(-ab^{k-1}\int_{-\delta}^\delta (B(t)+1)\,\mathrm{d}t)$. Then it is not hard to check using the invariance principle (see Lemma \ref{lem:invar} below) that the last line of \eqref{sepcond} converges as $N\to\infty$ to
	\begin{equation}\label{BBfrac}
		\frac{\mathbf{P}_{\delta;-1}^{ab^{k-1};0,0} \left( x < B(0) \leq x+\eta \right)}{\mathbf{P}_{\delta;-1}^{ab^{k-1};0,0} \left(B(0) > x, \; \sup_{|t|<\delta} B(t) < x+\nu \right)}.
	\end{equation}
	It is straightforward to see that for $\eta>0$ sufficiently small depending on $\nu$, the ratio in \eqref{BBfrac} can be made arbitrarily small uniformly over $x\in[0,M]$. Indeed, note that we can ignore the partition function in both numerator and denominator since it cancels. We then bound the unnormalized denominator below by
	\[
	e^{-2ab^{k-1}\delta(x+\nu+1)} \cdot \mathbf{P}_{\delta}^{0,0} \left(W(0) > x \mbox{ and }-1 < W(t) < x+\nu \mbox{ for all } t\in[-\delta,\delta]\right), 
	\]
	where we recall $\mathbf{P}_{\delta}^{0,0}$ is the law of a Brownian bridge $W$ on $[-\delta,\delta]$ from 0 to 0 with no floor or area tilt. This expression is strictly positive by straightforward Brownian estimates, and is clearly continuous as a function of $x\in[0,M]$, so it is bounded below by a positive constant for all such $x$. On the other hand, the numerator in \eqref{BBfrac} can be bounded above by $\mathbf{P}_{\delta}^{0,0}(x<W(0)\leq x+\eta)$ since the area tilt is at most 1. Under this law, $W(0)$ is a Gaussian with mean 0 and some variance $\sigma^2$ depending only on $\delta$, so this probability is bounded by $C\eta$ for $C>0$ independent of $x$. This allows us to choose $\eta>0$ such that the last line of \eqref{sepcond} is less than $\varepsilon/16$ for large $N$ uniformly over the event $\mathsf{Fav}_k(\nu,\delta,M)$, and combining with \eqref{sepfav} and \eqref{contk-1} implies \eqref{Sepk,k+1}. Along with \eqref{Sepk-12e} this completes the induction, and finally combining with \eqref{UVexist} finishes the proof.
\end{proof}

\subsection{Gibbs property for subsequential limits}\label{limitgibbs}

In this subsection we prove Proposition \ref{prop:gibbs}. The argument is standard, so we will be brief. The idea is to exploit the discrete Gibbs property (see Remark \ref{def:gibbs}) along with the following invariance principle for area-tilted random walk bridges on the diffusive scale. Recall the area-tilted Brownian Gibbs measures $\mathbf{P}_{k,I;g,h}^{a,b;\mathbf{u},\mathbf{v}}$ of Definition \ref{def:brownian}.

\begin{lemma}\label{lem:invar}
	Fix $I = [\ell,r]$, $h : I \to \mathbb{R}$, $\mathbf{u}\in W^k_{h(\ell)}$, $\mathbf{v}\in W^k_{h(r)}$. Let $I_N = \sigma^{-2/3}\llbracket \ell N^{2/3}, r N^{2/3}\rrbracket$, and suppose $h_N : I_N\to\mathbb{R}$, $\mathbf{u}^N \in W^k_{h_N(\ell)}$, $\mathbf{v}^N \in W^k_{h_N(r)}$ satisfy $\sigma^{-2/3} N^{-1/3}h_N(t N^{2/3}) \to h(t)$ uniformly in $t\in I$, $\sigma^{-2/3} N^{-1/3}\mathbf{u}^N \to \mathbf{u}$, and $\sigma^{-2/3} N^{-1/3}\mathbf{v}^N \to \mathbf{v}$ as $N\to\infty$. Let $\mathbf{Y}^N \sim \mathbb{P}_{k,I_N;h_N}^{a,b;\mathbf{u}^N,\mathbf{v}^N}$ and $\mathbf{y}^N(t) := \sigma^{-2/3} N^{-1/3}\mathbf{Y}^N(t\sigma^{-2/3} N^{2/3})$ for $t\in I$. Then $\mathbf{y}^N$ converges in law as $N\to\infty$ to $\mathbf{P}_{k,I;h}^{a,b;\mathbf{u},\mathbf{v}}$.
\end{lemma}

This is essentially the content of Lemma 2.1 of \cite{ser23}. That lemma is stated for lattice walks with $\sigma=1$, but the proof only uses the invariance principle for lattice random walk bridges from \cite{liggett}, and this can be directly adapted to the nonlattice case using the result of \cite{borisov}; see Lemma \ref{l.invar} with $\alpha=-2/3$. \cite[Lemma 2.1]{ser23} is also stated only for $h_N=0$, but the adaptation to the more general case here is trivial. To deal with $\sigma\neq 1$, the analogue of the area tilt calculation in \cite{ser23} is that if $y_i^N \to y_i$ uniformly on $I$, then we have the Riemann sum convergence
\[
\frac{1}{N} \sum_{j\in I} Y^N_i(j) = \sigma^{2/3}N^{-2/3}\sum_{j\in I} y_i^N(j\sigma^{2/3}N^{-2/3}) \longrightarrow \int_\ell^r y_i(t)\,\mathrm{d}t,
\] 
so that the limiting area tilt indeed matches that of $\mu_{a,b}$. This is the reason for the choice of the scaling $\sigma^{-2/3}$. We refer the reader to \cite{ser23} for further details of the argument.

\begin{proof}[Proof of Proposition \ref{prop:gibbs}]
	
	Let $\mathbf{x}$ be any weak subsequential limit of $\{\mathbf{x}^N\}_{N\geq 1}$, where we recall $\mathbf{x}^N(t) = N^{-1/3}\mathbf{X}(tN^{2/3})$ with $\mathbf{X}\sim\mathbb{P}_{n,N;0}^{a,b;\mathbf{u},\mathbf{v}}$. Without loss of generality we may assume that in fact $\mathbf{x}^N \to \mathbf{x}$ weakly. By the Skorohod representation theorem, by passing to another probability space we can assume that $\mathbf{x}^N\to\mathbf{x}$ uniformly on compact sets almost surely. 
	
	By a monotone class argument and the definition of conditional expectation, in order to verify \eqref{eqn:bgp} it suffices to prove the following. Fix $I = [\ell,r]$ and $\Sigma = \llbracket 1,k\rrbracket$. Then for any $\Sigma ' = \llbracket 1,k'\rrbracket \supseteq \Sigma$, $I = [\ell',r'] \supseteq I$, and any bounded continuous functional $G : C((\Sigma'\times I') \setminus (\Sigma\times I)) \to \mathbb{R}$,
	\begin{equation}\label{condexp}
		\mathbb{E}[F(\mathbf{x})G(\mathbf{x})] = \mathbb{E}\left[\mathbf{E}_{k,I;x_{k+1}}^{a,b;\mathbf{x}(\ell),\mathbf{x}(r)} [ F] G(\mathbf{x}) \right].
	\end{equation}
	Here we are using the shorthand $F(\mathbf{x}) = F(\mathbf{x}|_{\Sigma\times I})$ and $G(\mathbf{x}) = G(\mathbf{x}|_{(\Sigma'\times I')\setminus(\Sigma\times I)})$; and in the inner expectation on the right we are implicitly restricting $x_{k+1}$ to $I$ and $\mathbf{x}(\ell),\mathbf{x}(r)$ to $\Sigma$.
	
	By the a.s. uniform convergence $\mathbf{x}^N \to \mathbf{x}$ and the dominated convergence theorem,
	\begin{equation}\label{skorohod}
		\mathbb{E}[F(\mathbf{x})G(\mathbf{x})] = \lim_{N\to\infty} \mathbb{E}[F(\mathbf{x}^N)G(\mathbf{x}^N)].
	\end{equation}
	Now let $J=\llbracket \lceil \ell \sigma^{-2/3} N^{2/3}\rceil , \lfloor r\sigma^{-2/3} N^{2/3}\rfloor \rrbracket$, and let $\mathcal{F}^N_{k,J}$ denote the $\sigma$-algebra generated by $\{X_i(j) : i\notin \Sigma \mbox{ or } j\notin J\}$. Let $\mathbf{Y}\sim \mathbb{P}_{k,J;X_{k+1}}^{a,b;\mathbf{X}(\ell\sigma^{-2/3}  N^{2/3}), \mathbf{X}(r\sigma^{-2/3} N^{2/3})}$ (with shorthand as above) and define $\mathbf{y}^N(t) := \sigma^{-2/3} N^{-1/3}\mathbf{Y}(t\sigma^{-2/3} N^{2/3})$ for $t\in I$. Note that $G(\mathbf{x}^N)$ is $\mathcal{F}^N_{k,J}$-measurable. The tower property and the discrete Gibbs property for $\mathbf{X}$, see Remark \ref{def:gibbs}, then imply that
	\begin{equation}\label{eqgibbs}
		\mathbb{E}[F(\mathbf{x}^N)G(\mathbf{x}^N)] =  \mathbb{E} \left[ \mathbb{E}\left[F(\mathbf{x}^N)\mid \mathcal{F}^N_{k,J}\right] G(\mathbf{x}^N)  \right] =  \mathbb{E} \left[\mathbb{E}[F(\mathbf{y}^N)] G(\mathbf{x}^N)\right].
	\end{equation}
	Again using the a.s. uniform convergence $\mathbf{x}^N \to \mathbf{x}$, Lemma \ref{lem:invar} implies that $\mathbf{y}^N$ converges in law, a.s., to $\mathbf{P}_{k,I;x_{k+1}}^{a,b;\mathbf{x}(\ell),\mathbf{x}(r)}$ as $N\to\infty$, so that $\mathbb{E}[F(\mathbf{y}^N)] \to \mathbf{E}_{k,I;x_{k+1}}^{a,b;\mathbf{x}(\ell),\mathbf{x}(r)}[F]$ a.s. Applying this fact with dominated convergence in \eqref{eqgibbs} yields
	\[
	\lim_{N\to\infty} \mathbb{E}[F(\mathbf{x}^N)G(\mathbf{x}^N)] = \mathbb{E} \left[\mathbf{E}_{k,I;x_{k+1}}^{a,b;\mathbf{x}(\ell),\mathbf{x}(r)}[F] G(\mathbf{x})\right].
	\]
	Combining with \eqref{skorohod} proves \eqref{condexp}.
\end{proof}

%!TEX root = ./rwareatilt.tex

\appendix
\section{Stochastic monotonicity}\label{s.monotonicity}

In this section we give the proof of Lemma \ref{l.monotonicity}. The argument is a standard application of Glauber dynamics to construct a monotone coupling of two appropriate Markov chains. In particular, our argument closely follows the one presented in \cite[Appendix B]{wu2019tightness}, while the idea of constructing Markov chains to showcase monotone couplings originates goes back to \cite{corwin2014brownian,corwin2016kpz}. Among the conditions on the random walk bridges stated in Section \ref{s.RWest}, we require here only the convexity of the random walk Hamiltonian $H_{\mathrm{RW}}$, Assumption \ref{a.convex}.

\begin{proof}
	We will first construct two Markov chains on the space of trajectories whose dynamics are coupled such that the pointwise ordering is  preserved between them. This will yield a coupling of the stationary distributions $\nu^{\shortuparrow}$ and $\nu^{\shortdownarrow}$ of the constructed Markov chains which also exhibits the pointwise ordering (assuming the Markov chains converge). We will arrange so that $\nu^* = \P^{a^*, b^*;  \mathbf{u}^*, \mathbf{v}^*}_{n,[0,T]; g^*, h^*}$ for $*\in\{\shortuparrow,\shortdownarrow\}$. We present the argument only in the nonlattice random walk case (2) in Section \ref{s.RWest}; the lattice case (1) is analogous and in fact simpler.
	
	Note that the desired stationary measures $\nu^*$ are laws on $\mathbb{R}^{n|I|}$. So as to work on a finite state space, we first discretize. We will then return to our original state space by taking a further weak limit of the stationary distributions of the discretized chains to obtain the desired laws. 
	
	Our discretization parameter will be $\varepsilon$. Consider the finite space $\mc X_{\varepsilon}:=\varepsilon\Z\cap [-\varepsilon^{-1},\varepsilon^{-1}]$ and the random variable supported on $\mc X_{\varepsilon}$ whose probability mass function at $k\varepsilon$, for $k\in\intint{-\floor{\varepsilon^{-2}}, \floor{\varepsilon^{-2}}}$, is proportional to
	\begin{align}\label{e.discretized pmf}
		\varepsilon\exp\left(-H_{\mrm{RW}}(k\varepsilon)\right);
	\end{align}
	it follows that the normalization constant is of unit order as a function of $\varepsilon$ and, indeed, converges to the normalization $\int_{\mathbb{R}} \exp(-H_{\mathrm{RW}}(x))\,dx$ for $P_{\mathrm{RW},n}$ as $\varepsilon\downarrow 0$. Let $P^{\varepsilon}_{\mrm{RW}, n}$ be the transition probability kernel of $n$ independent random walks whose increment distribution is given by the law described in \eqref{e.discretized pmf} and $P_{\mrm{RW}, n}$ be that of $n$ independent random walks with increment distribution having density proportional to $\exp(-H_{\mrm{RW}}(x))$. So, if $\mathbf{x}^\varepsilon\to \mathbf{x}$ and ${ \mathbf{x}^\varepsilon}'\to \mathbf{x}'$ as $\varepsilon\downarrow0$, then $P^{\varepsilon}_{\mrm{RW},n}(\mathbf{x}^\varepsilon, {\mathbf{x}^{\varepsilon}}') \to P_{\mrm{RW}, n}(\mathbf{x}, \mathbf{x}')$ as $\varepsilon\downarrow 0$.
	
	For $w,z\in\mathcal{X}_\ep$ with $w \geq z$, let us use the notation $\mathcal{X}^n_{\ep;w,z} := \{\mathbf{x} \in (\mathcal{X}_\ep)^n : w \geq x_1 \geq \cdots \geq x_n \geq z\}$. For $g^\varepsilon, h^\varepsilon: I\to \mc X_\varepsilon$, let us further define $\mathcal{X}_{\varepsilon;g^\ep,h^\ep}^{n,I}$ to be the set of trajectories $\mathcal{Q} = (\mathcal{Q}_i(x) : i\in \llbracket 1,n\rrbracket, \, x\in I)$ such that $\mathcal{Q}(x) \in \mathcal{X}_{\ep;g^{\ep}(x),h^{\ep}(x)}^n$ for all $x\in I$. Writing $I = [\ell, r]$, let $\u \in \mathcal{X}^n_{\ep; g^\ep(\ell),h^\ep(\ell)}$, $\v\in \mathcal{X}^n_{\ep;g^\ep(r),h^\ep(r)}$, and consider the law on $\mathcal{X}_{\varepsilon;g^\ep,h^\ep}^{n,I}$ given by the analogue of $\P^{a,b; \mathbf{u}^\varepsilon, \mathbf{v}^\varepsilon}_{n,I; g^\varepsilon, h^{\varepsilon}}$ with the random walk increment distribution defined by \eqref{e.discretized pmf} instead of having density proportional to $\exp(-H_{\mrm{RW}}(x))$; we denote this law by $\P^{a,b; \mathbf{u}^\varepsilon, \mathbf{v}^\varepsilon; \varepsilon}_{n,I; g^\varepsilon, h^\varepsilon}$. It is immediate from the portmanteau theorem that if $(\mathbf{u}^\varepsilon,\mathbf{v}^\varepsilon,g^\varepsilon, h^\varepsilon)\to (\mathbf{u},\mathbf{v},g,h)$ as $\varepsilon\downarrow0$ then this law converges weakly to $\P^{a,b; \mathbf{u},\mathbf{v}}_{n,I; g,h}$, so it suffices to prove the monotone coupling for the discretized laws. For convenience we will drop the $\varepsilon$ superscripts for $\mathbf{u},\mathbf{v},g,h$ in the following.
	We will consider a Glauber-type discrete-time dynamics for the measures $\smash{\P^{a^{\shortdownarrow}, b^{\shortdownarrow}; \mathbf{u}^{\shortdownarrow}, \mathbf{v}^{\shortdownarrow}; \varepsilon}_{n,I; g^{\shortdownarrow}, h^{\shortdownarrow}}}$ and $\smash{\P^{a^{\shortuparrow}, b^{\shortuparrow}; \mathbf{u}^{\shortuparrow}, \mathbf{v}^{\shortuparrow}; \varepsilon}_{n,I; g^{\shortuparrow}, h^{\shortuparrow}}}$. The strategy will be to couple these dynamics across $*\in\{\shortuparrow,\shortdownarrow\}$ by using the same randomness and show that then the desired pointwise ordering is maintained at every time step under the dynamics. Denote the two chains by ${\mc Q}^{\shortuparrow} = ({\mc Q}^{\shortuparrow, t})_{t\in\N}$ and ${\mc Q}^{\shortdownarrow} = ({\mc Q}^{\shortdownarrow, t})_{t\in\N}$, where ${\mc Q}^{*, t}\in \mathcal{X}_{\varepsilon;g^\ep,h^\ep}^{n,I}$ for all $t\in\N$ and $*\in\{\shortuparrow,\shortdownarrow\}$.
	
	The dynamics are as follows. Let $\mc Q^{\shortuparrow,0}$  and $\mc Q^{\shortdownarrow,0}$ be chosen to be the same (arbitrary) deterministic state. Let $(U_{t})_{t\in\N}$ be a collection of \iid random variables distributed uniformly on $[0,1]$. If $I=\llbracket \ell, r\rrbracket$, at every time step $t$ we choose $(K, X, \sigma)\in\intint{1,n}\times\intint{\ell+1,r-1}\times\{\pm 1\}$ uniformly at random and independent of everything else. Define $\widetilde{\mc Q}^{*, t+1}$ for $*\in\{\shortuparrow,\shortdownarrow\}$ by $\widetilde{\mc Q}^{*, t+1}_K(X) = {\mc Q}^{*, t}_K(X) + \varepsilon \sigma$ and $\widetilde{\mc Q}^{*, t+1}_{k}(x) = {\mc Q}^{*, t}_{k}(x)$ for all $(k,x)\neq (K,X)$. Define  $R^{*}_t$ for $*\in\{\shortuparrow,\shortdownarrow\}$ by the following:
	\begin{equation}\label{e.R definition}
		\begin{split}
			R^{*}_t &= \frac{\mathbb{P}_{n,I;g^*,h^*}^{a^*,b^*;\mathbf{u}^*,\mathbf{v}^*;\varepsilon} (\widetilde{\mathcal{Q}}^{*,t+1})}{\mathbb{P}_{n,I;g^*,h^*}^{a^*,b^*;\mathbf{u}^*,\mathbf{v}^*;\varepsilon} (\mathcal{Q}^{*,t})}\\
			&= \frac{\exp\left(-\frac{a^*}{N}\sum_{i=1}^n(b^{*})^{i-1}\mc A(\widetilde{\mc Q}^{*,t+1}_i)\right)}{\exp\left(-\frac{a^*}{N}\sum_{i=1}^n(b^{*})^{i-1}\mc A(\mc Q^{*,t}_i)\right)}\cdot \prod_{j=0}^1
			\frac{P^\varepsilon_{\mrm{RW},n}(\widetilde{\mc Q}^{*,t+1}(X+j-1),\widetilde{\mc Q}^{*,t+1}(X+j))}{P^\varepsilon_{\mrm{RW},n}(\mc Q^{*,t}(X+j-1),\mc Q^{*,t}(X+j))}.
		\end{split}
	\end{equation}
	For $*\in\{\shortuparrow,\shortdownarrow\}$, we set $\mc Q^{*,t+1} = \widetilde{\mc Q}^{*,t+1}$ if both $R^*_{t} \geq U_{t}$ and $\widetilde{\mc Q}^{*,t+1}\in \mathcal{X}_{\varepsilon;g^\ep,h^\ep}^{n,I}$, and $\mc Q^{*,t+1} = \mc Q^{*,t}$ otherwise. 
	It is an easy check that $\P^{a^*, b^*; \mathbf{u}^{*}, \mathbf{v}^{*}; \varepsilon}_{n,I; g^*, h^{*}}$ is stationary for these dynamics, in fact reversible.
	
	Note in the above that the dynamics of the $\shortuparrow$ and $\shortdownarrow$ chains are coupled by using the same uniform random variables for deciding acceptance. We claim that these dynamics are such that $\mc Q^{\shortdownarrow, t}_i(x) \leq \mc Q^{\shortuparrow, t}_i(x)$ for all $i\in\intint{1,N}$, $x\in I$, $t\in\N$. We prove this by induction on $t$. The claim holds at $t=0$ since we adopt the same starting state for both chains; assume it holds for some $t\geq 0$. In going from $t$ to $t+1$, note that at most one location on one curve is changed in each of $\mc Q^{\shortuparrow, t}$ and $\mc Q^{\shortdownarrow, t}$ (and the same site in both, if a change occurs), and the change is by $\pm\varepsilon$. Thus there are two cases where the pointwise ordering could be violated in the transition from $t$ to $t+1$: 
	\begin{enumerate}
		\item $(K,X,\sigma)$ is chosen with $\sigma=+1$, $\mc Q^{\shortuparrow, t}_K(X) = \mc Q^{\shortdownarrow, t}_K(X)$, $R^{\shortuparrow}_{t} < U_{t} \leq R^{\shortdownarrow}_{t}$, and $\widetilde{\mc Q}^{*,t+1}\in \mathcal{X}_{\varepsilon;g^\ep,h^\ep}^{n,I}$ for $*\in\{\shortuparrow,\shortdownarrow\}$, or 
		
		\item $(K,X,\sigma)$ is chosen with $\sigma=-1$, $\mc Q^{\shortuparrow, t}_K(X) = \mc Q^{\shortdownarrow, t}_K(X)$, $R^{\shortdownarrow}_{t} < U_{t} \leq R^{\shortuparrow}_{t}$, and $\widetilde{\mc Q}^{*,t+1}\in \mathcal{X}_{\varepsilon;g^\ep,h^\ep}^{n,I}$ for $*\in\{\shortuparrow,\shortdownarrow\}$.
	\end{enumerate}
	We argue that the first case cannot occur, by showing that $R^{\shortuparrow}_{t}\geq R^{\shortdownarrow}_{t}$ in that case; an analogous argument shows that the second case also cannot occur. This will complete the proof.
	
		We have by the inductive hypothesis that $\mc Q^{\shortdownarrow, t}_{k}(x)\leq \mc Q^{\shortuparrow, t}_{k}(x)$ for all $k,x$. We have to compare $R_t^{\shortuparrow}$ and $R_t^{\shortdownarrow}$ as defined in \eqref{e.R definition}. We start with the second factor there, associated with the random walk. 
		Recall $H_{\mrm{RW}}$ is convex as a function on $\R$, and let $z := \mc Q^{\shortuparrow,t}_K(X) = \mc Q^{\shortdownarrow,t}_K(X)$. We observe that
	\begin{align}
		\MoveEqLeft[8]
		\prod_{j=0}^{1}\frac{P^{\varepsilon}_{\mrm{RW},n}(\widetilde{\mc Q}^{*,t+1}(X+j-1), \widetilde{\mc Q}^{*,t+1}(X+j))}{P^{\varepsilon}_{\mrm{RW},n}(\mc Q^{*,t}(X+j-1), \mc Q^{*,t}(X+j)) }\nonumber\\
		&= \displaystyle\frac{\exp\left(-H_{\mrm{RW}}\bigl(z + \varepsilon-\mc Q^{*,t}_K(X-1)\bigr) -H_{\mrm{RW}}\bigl(\mc Q^{*,t}_K(X+1) -z -\varepsilon\bigr)\right)}{\exp\left(-H_{\mrm{RW}}\bigl(z-\mc Q^{*,t}_K(X-1)\bigr)-H_{\mrm{RW}}\bigl(\mc Q^{*,t}_K(X+1) -z\bigr)\right)}.
		\label{e.rw transition prob ratio}
	\end{align}
	For any convex function $f:\R\to\R$ and for any $z,a,b\in\R$ with $a\geq b$ and $\varepsilon>0$, it holds that
	\begin{equation}\label{e.convexity relations}
		\begin{split}
			-f(z+\varepsilon -a) + f(z-a) &\geq - f(z+\varepsilon -b) + f(z-b) \quad\text{and}\\
			-f(a - z-\varepsilon) + f(a-z) &\geq - f(b-z-\varepsilon) + f(b-z).
		\end{split}
	\end{equation}
	
	Applying this inequality in \eqref{e.rw transition prob ratio} with $f=H_{\mrm{RW}}$, and using the inductive hypothesis that $\mc Q^{\shortuparrow,t}_k(x) \geq \mc Q^{\shortdownarrow,t}_k(x)$ for $x\in\{X-1,X+1\}$, yields that
	\begin{align}\label{e.RW facto ratio}
		\prod_{j=0}^{1}\frac{P^{\varepsilon}_{\mrm{RW},n}(\widetilde{\mc Q}^{\shortuparrow,t+1}(X+j-1), \widetilde{\mc Q}^{\shortuparrow,t+1}(X+j))}{P^{\varepsilon}_{\mrm{RW},n}(\mc Q^{\shortuparrow,t}(X+j-1), \mc Q^{\shortuparrow,t}(X+j)) } \geq 
		\prod_{j=0}^{1}\frac{P^{\varepsilon}_{\mrm{RW},n}(\widetilde{\mc Q}^{\shortdownarrow,t+1}(X+j-1), \widetilde{\mc Q}^{\shortdownarrow,t+1}(X+j))}{P^{\varepsilon}_{\mrm{RW},n}(\mc Q^{\shortdownarrow,t}(X+j-1), \mc Q^{\shortdownarrow,t}(X+j)) }.
	\end{align}
	
	So far we heave dealt with the ratio of the $P^{\varepsilon}_{\mrm{RW}, n}$ factors in \eqref{e.R definition}, and next we turn to the ratio of the area-tilt factors. The conditions of case 1 imply $\widetilde{\mc Q}^{*,t+1}_K(X) = {\mc Q}^{*,t}_K(X) + \varepsilon$ and $\widetilde{\mc Q}^{*,t+1}_k(x) = {\mc Q}^{*,t}_k(x)$ for all $(k,x)\neq(K,X)$. So
	% %
	\begin{align*}
		\frac{\exp\left(-\frac{a^*}{N}\sum_{i=1}^n(b^{*})^{i-1}\mc A(\widetilde{\mc Q}^{*,t+1}_i)\right)}{\exp\left(-\frac{a^*}{N}\sum_{i=1}^n(b^{*})^{i-1}\mc A(\mc Q^{*,t}_i)\right)}
		= \exp(-\tfrac{a^*}{N}(b^{*})^{K-1}\varepsilon).
	\end{align*}
	Since $a^{\shortuparrow}\leq a^{\shortdownarrow}$ and $b^{\shortuparrow} \leq b^{\shortdownarrow}$, it holds from the previous display and \eqref{e.RW facto ratio} that $R^{\shortuparrow}_t \geq R^{\shortdownarrow}_t$ in case 1, as desired. This completes the proof.
\end{proof}

\bibliographystyle{alpha}
\bibliography{rwareatilt}

\end{document}